\documentclass{lmcs}
\pdfoutput=1
\usepackage[utf8]{inputenc}

\usepackage{lastpage}
\lmcsdoi{19}{2}{1}
\lmcsheading{}{\pageref{LastPage}}{}{}%
{Jul.~10,~2021}{Apr.~05,~2023}{}

\usepackage{microtype}
\microtypesetup{expansion=false}

\usepackage{amssymb,amsmath,stmaryrd,mathrsfs,cmll}
\usepackage{mathtools,url,xspace,mathpartir,bold-extra}
\usepackage{rotating,makecell,subcaption}
\usepackage{tikz,tikz-cd}
\usetikzlibrary{arrows}
\usepackage[pagebackref]{hyperref}
\renewcommand*{\backref}[1]{}
\renewcommand*{\backrefalt}[4]{({%
    \ifcase #1 Not cited.%
    \or Cited on p.~#2%
    \else Cited on pp.~#2%
    \fi%
  })}

\makeatletter
\let\ea\expandafter

\newcount\foreachcount

\def\foreachletter#1#2#3{\foreachcount=#1
  \ea\loop\ea\ea\ea#3\@alph\foreachcount
  \advance\foreachcount by 1
  \ifnum\foreachcount<#2\repeat}

\def\foreachLetter#1#2#3{\foreachcount=#1
  \ea\loop\ea\ea\ea#3\@Alph\foreachcount
  \advance\foreachcount by 1
  \ifnum\foreachcount<#2\repeat}

\def\definescr#1{\ea\gdef\csname s#1\endcsname{\ensuremath{\mathscr{#1}}\xspace}}
\foreachLetter{1}{27}{\definescr}
\def\definecal#1{\ea\gdef\csname c#1\endcsname{\ensuremath{\mathcal{#1}}\xspace}}
\foreachLetter{1}{27}{\definecal}
\def\definebb#1{\ea\gdef\csname d#1\endcsname{\ensuremath{\mathbb{#1}}\xspace}}
\foreachLetter{1}{27}{\definebb}
\def\definesf#1{\ea\gdef\csname i#1\endcsname{\ensuremath{\mathsf{#1}}\xspace}}
\foreachletter{1}{6}{\definesf}
\foreachletter{7}{14}{\definesf}
\foreachletter{15}{20}{\definesf}
\foreachletter{21}{27}{\definesf}
\foreachLetter{1}{27}{\definesf}

\def\ftil{\widetilde{f}}
\def\gtil{\widetilde{g}}
\def\Rtil{\widetilde{R}}
\def\Stil{\widetilde{S}}

\def\autofmt@d#1#2\autofmt@end{\mathbb{#1}\mathsf{#2}}
\def\autofmt@c#1#2\autofmt@end{\mathcal{#1}\mathit{#2}}
\def\autofmt@f#1\autofmt@end{\mathsf{#1}}

\def\auto@drop#1{}
\def\autodef#1{\ea\ea\ea\@autodef\ea\ea\ea#1\ea\auto@drop\string#1\autodef@end}
\def\@autodef#1#2#3\autodef@end{%
  \ea\def\ea#1\ea{\ea\ensuremath\ea{\csname autofmt@#2\endcsname#3\autofmt@end}\xspace}}
\def\autodefs@end{blarg!}
\def\autodefs#1{\@autodefs#1\autodefs@end}
\def\@autodefs#1{\ifx#1\autodefs@end%
  \def\autodefs@next{}%
  \else%
  \def\autodefs@next{\autodef#1\@autodefs}%
  \fi\autodefs@next}

\autodefs{\fMod\fEnv\cSh\cShEnv\fSet\dChu\cGl\dGl\fChu\fPoly\fGl\cCat\cProf\fMulti\fKl\fCartMulti\fSymMulti\fSketch\fRep\fMAdj\fCplt\fCat\fsCat\dCSU\dDuals\dChuData\dLNLPoly\fClub}

\DeclareSymbolFont{bbold}{U}{bbold}{m}{n}
\DeclareSymbolFontAlphabet{\mathbbb}{bbold}
\newcommand{\unit}{\ensuremath{\mathbbb{1}}\xspace}

\newcommand{\op}{^{\mathrm{op}}}

\let\adj\dashv
\def\ten{\mathrel{\otimes}}
\DeclareMathOperator\colim{colim}
\DeclareMathOperator\llim{lim}
\let\lim\llim
\newcommand{\too}[1][]{\ensuremath{\overset{#1}{\longrightarrow}}}
\let\toot\rightleftarrows
\let\toto\rightrightarrows
\let\into\hookrightarrow
\let\xto\xrightarrow
\def\toiso{\xto{\smash{\raisebox{-.5mm}{$\scriptstyle\sim$}}}}

\newcommand{\drpullback}[1][dr]{\ar[#1,phantom,near start,"\lrcorner"]}

\newcommand{\ulpullback}[1][ul]{\ar[#1,phantom,near start,"\ulcorner"]}

\let\ep\varepsilon

\let\al\alpha

\let\Gm\Gamma

\let\De\Delta
\let\si\sigma

\let\om\omega
\let\ka\kappa

\let\Th\Theta

\DeclareSymbolFont{symbolsC}{U}{txsyc}{m}{n}
\SetSymbolFont{symbolsC}{bold}{U}{txsyc}{bx}{n}
\DeclareFontSubstitution{U}{txsyc}{m}{n}
\DeclareMathSymbol{\multimapinv}{\mathrel}{symbolsC}{18}
\DeclareMathSymbol{\multimapboth}{\mathrel}{symbolsC}{19}
\DeclareMathSymbol{\multimapdot}{\mathrel}{symbolsC}{20}
\DeclareMathSymbol{\multimapdotinv}{\mathrel}{symbolsC}{21}
\DeclareMathSymbol{\multimapdotboth}{\mathrel}{symbolsC}{22}
\DeclareMathSymbol{\multimapdotbothA}{\mathrel}{symbolsC}{23}
\DeclareMathSymbol{\multimapdotbothB}{\mathrel}{symbolsC}{24}

\let\hom\multimap
\let\cohom\lhd
\def\mixedhom{\mathbin{\mathrlap{\to}\multimap}}
\def\mixedeechom{\mathbin{\mathrlap{\to}\multimapdot}}
\let\eechom\multimapdot
\let\coten\invamp
\let\counit\bot
\def\d#1{#1^*}

\def\duals{\d{(\cdot)}}

\def\Umulti{\mathsc{symmulti}^*}

\def\Pl(#1|#2;#3){\cP\big(#1\mid #2\mathbin{;} #3\big)}
\def\El(#1|#2;#3){\cE\big(#1\mid #2\mathbin{;} #3\big)}
\def\Ul(#1|#2;#3){\cU\big(#1\mid #2\mathbin{;} #3\big)}
\def\linhom#1(#2|#3;#4){#1\big(#2\mid #3\mathbin{;} #4\big)}
\def\Pnl(#1;#2){\cP\big(#1 \mathbin{;} #2\big)}
\def\Enl(#1;#2){\cE\big(#1 \mathbin{;} #2\big)}
\def\nonlinhom#1(#2;#3){#1\big(#2 \mathbin{;} #3\big)}
\def\Pnonlin{\cP^{\mathrm{NL}}}
\def\nonlin#1{#1^{\mathrm{NL}}}
\def\Plin{\cP^{\mathrm{L}}}
\def\lin#1{#1^{\mathrm{L}}}

\def\act#1#2#3{{}^{#1}(#2)^{#3}}
\def\foc{\iF}
\def\uoc{\iU}
\def\fwn{\raisebox{2.4mm}{\rotatebox{180}{\iF}}}
\def\uwn{\raisebox{2.4mm}{\rotatebox{180}{\iU}}}
\def\bkll{\cL_{\oc,\wn}}
\def\lop{^{\mathrm{L\cdot op}}}
\def\coun#1{\mathsc{lnlmulti}^*(#1)}
\def\shift[#1|#2;#3]{[#1\,|\,#2\,;\,#3]}
\def\lnl{\text{\textsc{lnl}}\xspace}
\def\mathsc#1{\text{\textsc{#1}}}

\def\lnlpoly{\text{\textsc{lnl}}\fPoly}

\def\D{^{\bullet}}
\def\dsketch{\dD\text{-}\fSketch}
\def\sketch#1{#1\text{-}\fSketch}

\def\dcat{\dD\text{-}\fCat}
\def\dscat{\dD\text{-}\fsCat}
\def\cat#1{#1\text{-}\fCat}
\def\scat#1{#1\text{-}\fsCat}
\def\sdhat{\widehat{\cS}_{\dD}}
\def\abs#1{{|#1|}}
\def\ladj#1{#1_*}
\def\radj#1{#1^*}
\def\F{\ensuremath{\mathfrak{F}}\xspace}
\def\Fhat{\ensuremath{\widehat{\mathfrak{F}}}\xspace}
\def\Ftil{\ensuremath{\widetilde{\mathfrak{F}}}\xspace}

\def\istype#1{\;\mathsf{type}^{#1}}
\def\ttype{\;\mathsf{type}^{\tau}}
\def\ltype{\;\mathsf{type}^{\mathrm{L}}}
\def\nltype{\;\mathsf{type}^{\mathrm{NL}}}

\def\boc{\mathord{\textstyle\bigodot}_{\cC}}

\newcommand\expand[2]{{#1}_{/#2}}
\newcommand\expandp[2]{(#1)_{/#2}}
\newcommand\preexpand[2]{\partial({#1}_{/#2})}
\newcommand\preexpandp[2]{\partial((#1)_{/#2})}
\newcommand\cartcone[2]{\cC\mathit{art}_{#1/#2}}
\newcommand\limconel[1]{\cL\mathit{im}^{\mathrm{L}}_{#1}}
\newcommand\colimconel[1]{\cC\mathit{olim}^{\mathrm{L}}_{#1}}
\newcommand\limconenl[1]{\cL\mathit{im}^{\mathrm{NL}}_{#1}}
\newcommand\colimconenl[1]{\cC\mathit{olim}^{\mathrm{NL}}_{#1}}
\newcommand\reduct[1]{\partial #1}
\newcommand\newterm[1]{#1^{\triangleright}}
\newcommand\newinit[1]{#1^{\triangleleft}}

\usepackage[nosort]{cleveref}
\let\mylabel\label
\renewcommand{\label}[1]{%
  \mylabel[\@currenvir]{#1}%
}
\crefname{thm}{Theorem}{Theorems}
\crefname{cor}{Corollary}{Corollaries}
\crefname{lem}{Lemma}{Lemmas}
\crefname{prop}{Proposition}{Propositions}
\crefname{rem}{Remark}{Remarks}
\crefname{exa}{Example}{Examples}
\crefname{defi}{Definition}{Definitions}
\crefformat{section}{Section~#2#1#3}
\Crefformat{section}{Section~#2#1#3}
\crefrangeformat{section}{Sections #3#1#4--#5#2#6}
\Crefrangeformat{section}{Sections #3#1#4--#5#2#6}
\crefmultiformat{section}{Sections #2#1#3}{ and~#2#1#3}{, #2#1#3}{ and~#2#1#3}
\Crefmultiformat{section}{Sections~#2#1#3}{ and~#2#1#3}{, #2#1#3}{ and~#2#1#3}
\crefrangemultiformat{section}{Sections #3#1#4--#5#2#6}{ and~#3#1#4--#5#2#6}{, #3#1#4--#5#2#6}{ and~#3#1#4--#5#2#6}
\Crefrangemultiformat{section}{Sections #3#1#4--#5#2#6}{ and~#3#1#4--#5#2#6}{, #3#1#4--#5#2#6}{ and~#3#1#4--#5#2#6}
\crefname{figure}{Figure}{Figures}
\crefname{table}{Table}{Tables}

\makeatother

\title{LNL polycategories and doctrines of linear logic}
\thanks{This material is based on research sponsored by The United States Air Force Research Laboratory under agreement numbers FA9550-15-1-0053 and FA9550-21-1-0009.}
\author[M.~Shulman]{Michael Shulman}
\address{University of San Diego}
\email{shulman@sandiego.edu}

\begin{document}

\begin{abstract}
  We define and study LNL polycategories, which abstract the judgmental structure of classical linear logic with exponentials.
  Many existing structures can be represented as LNL polycategories, including LNL adjunctions, linear exponential comonads, LNL multicategories, IL-indexed categories, linearly distributive categories with storage, commutative and strong monads, CBPV-structures, models of polarized calculi, Freyd-categories, and skew multicategories, as well as ordinary cartesian, symmetric, and planar multicategories and monoidal categories, symmetric polycategories, and linearly distributive and *-autonomous categories.
  To study such classes of structures uniformly, we define a notion of LNL doctrine, such that each of these classes of structures can be identified with the algebras for some such doctrine.
  We show that free algebras for LNL doctrines can be presented by a sequent calculus, and that every morphism of doctrines induces an adjunction between their 2-categories of algebras.
\end{abstract}

\maketitle

\setcounter{tocdepth}{1}
\tableofcontents

\section{Introduction}
\label{sec:introduction}

When presenting logics and type theories, it is generally useful to separate the \emph{structural} rules, such as exchange, weakening, contraction, identity, and cut, from the \emph{logical} rules governing particular connectives.
This separation of concerns can be reflected in categorical semantics by starting with a kind of \emph{multicategory}~\cite{lambek:dedsys-ii,hermida:multicats,leinster:higher-opds} or \emph{polycategory}~\cite{szabo:polycats} encapsulating the structural rules, in which we can formulate universal properties of objects that correspond to the connectives.

A multicategory is like a category, but allows the domain of a morphism to be a finite list of objects; a polycategory allows both the domain and codomain to be such a list.
Such morphisms correspond respectively to intuitionistic sequents $A_1,\dots,A_m \vdash B$ and classical sequents $A_1,\dots,A_m \vdash B_1,\dots, B_n$.
One can then formulate universal properties for ``tensor products'' as representing objects for such morphisms, generalizing the classical characterization of the tensor product of vector spaces as a representing object for multilinear maps.

The choice of structural rules in a logic is reflected by an action on the morphisms of a multi- or polycategory that modifies the elements in the domain or codomain lists.
For instance, the exchange rule is reflected by an operation taking any morphism $(\Gamma,A,B,\Delta) \to C$ to a morphism $(\Gamma,B,A,\Delta) \to C$.
This leads to different kinds of multi- and polycategory, such as the following.
\begin{itemize}
\item Cartesian multicategories (a.k.a.\ abstract clones) correspond to intuitionistic nonlinear logic, with all structural rules.
  A cartesian multicategory with enough representing objects is equivalent to a cartesian monoidal category or a cartesian closed category.
\item Symmetric multicategories correspond to intuitionistic multiplicative-additive linear logic, with exchange but no weakening or contraction.
  A symmetric multicategory with enough representing objects is equivalent to a symmetric monoidal category, possibly closed.
\item Symmetric polycategories correspond to classical multiplicative-additive linear logic.
  A symmetric polycategory with enough representing objects is equivalent to a linearly distributive category or a $\ast$-autonomous category.
\end{itemize}

Multicategories and polycategories also have advantages from a purely category-the\-o\-ret\-ic standpoint.
They
can simplify coherence problems, since operations defined by universal properties generally do not require explicit coherence axioms.
They can also enable the unification of different-looking structures in a larger context; for instance, monoidal categories and closed categories can both be represented as multicategories~\cite{hermida:multicats,manzyuk:closed}, and the Chu and Dialectica constructions are both instances of one polycategorical operation~\cite{shulman:dialectica}.

It seems, however, that no polycategorical structure exists in the literature to correspond to \emph{classical} linear logic \emph{with exponentials}.
Structured categories with exponential modalities have certainly been studied, such as \lnl adjunctions~\cite{benton:lnl} and linearly distributive categories with storage~\cite{bcs:storage}.
And a multicategorical version, corresponding to \emph{intuitionistic} linear logic with exponentials, is suggested in~\cite{ht:lnl-2mnd}.
But the polycategorical case appears to be missing.

In this paper we fill this gap by defining \emph{LNL polycategories}.
An \lnl polycategory has two classes of objects, called \emph{linear} and \emph{nonlinear}.
The linear objects form a symmetric polycategory, while the nonlinear objects form a cartesian multicategory, and there are additional morphisms relating the two classes of objects, enabling a description of the modalities $\oc$ and $\wn$ by universal properties.
This can be regarded as a semantic counterpart of split-context presentations of linear logic, such as~\cite{benton:lnl,barber:dill,wadler:syn-ll} in the intuitionistic case and~\cite{girard:unity} in the classical one.

Like their syntactic counterpart of full classical linear logic, \lnl polycategories are an extremely rich structure.
In addition to \lnl adjunctions and linearly distributive categories with storage, they include cartesian multicategories (if all objects are nonlinear), symmetric polycategories (if all objects are linear), symmetric multicategories (if all objects are linear and all codomains are unary), and CBPV structures (if all linear codomains are unary and all linear domains are subunary).
Thus, any structured category that can be represented by any of these multi- or polycategorical notions can also be regarded as an \lnl polycategory.

This suggests that \lnl polycategories should provide a unifying context to compare different kinds of structured category, and to study the correspondence between logic and category.
To facilitate this, we define a notion of \emph{LNL doctrine} \dD, whose ``algebras'' (which we call \emph{\dD-categories}) are \lnl polycategories satisfying certain object and arity restrictions and in which objects having certain universal properties exist.
Inspired by~\cite{hermida:fib-multi,lsr:multi,bz:bifib-poly}, we express these universal properties \emph{fibrationally}: an \lnl doctrine \dD is an \lnl polycategory $\abs{\dD}$ equipped with a collection of distinguished ``cones'', and a \dD-category is an \lnl polycategory \cP equipped with a functor $\cP \to \abs{\dD}$ admitting a ``cartesian'' lift for each distinguished cone.
We also incorporate a ``well-sortedness'' condition that allows a restriction to Kleisli adjunctions if desired.
In this way, we can represent all of the following kinds of structured category, and many more, as the algebras for \lnl doctrines:
\begin{itemize}
\item Cartesian multicategories, symmetric multicategories, symmetric polycategories, \lnl multicategories, and skew multicategories.
\item Symmetric monoidal categories, closed symmetric monoidal categories, and symmetric closed categories.
\item Cartesian monoidal categories and cartesian closed categories.
\item Cartesian monoidal categories with a commutative strong monad.
\item Symmetric monoidal categories with a strong monad.
\item CBPV adjunction models, EEC+ models, and ECBV models.
\item Freyd-categories and Freyd-multicategories.
\item Linearly distributive categories and $\ast$-autonomous categories.
\item \lnl adjunctions, possibly closed or $\ast$-autonomous.
\item Symmetric monoidal categories with a linear exponential comonad, linearly distributive categories with storage, and $\ast$-autonomous categories with storage.
\item Any of the above with any specified family of limits and/or colimits.
\end{itemize}

We also argue that \lnl doctrines provide a unifying context to study substructural logics, and to compare the corresponding kinds of monoidal category.
Specifically, we will use a well-known iterative category-theoretic construction, known as the \emph{small object argument}, to present the \emph{free} \dD-category $\sdhat$ generated by an input datum \cS that we call a \dD-\emph{sketch}.
This has the following two consequences.

Firstly, from this construction we can extract a syntactic sequent calculus that also presents free \dD-categories.
The iterative small object argument corresponds naturally to the inductive definition of sequent calculus derivations.
The structural rules arise since each stage is an \lnl polycategory, while the logical rules are inserted by iterative pushouts that enforce the existence of objects with universal properties.
Thus, there is a precise correspondence between the syntactic and semantic versions of the separation of concerns between structural and logical rules.

Secondly, we use the free \dD-category on a sketch to show that any morphism of doctrines $\F:\dD_1\to\dD_2$ induces a pseudo 2-adjunction between $\dD_1$-categories and $\dD_2$-categories.
That is, any $\dD_2$-category $\cT$ has an underlying $\dD_1$-category $\radj{\Fhat}\cT$, and any $\dD_1$-category \cS generates a free $\dD_2$-category $\ladj{\Fhat}\cS$.
Thus, \lnl doctrines also supply a uniform way to relate different sorts of monoidal category, potentially with exponential monads and comonads.

\section{\lnl polycategories}
\label{sec:lnl-polycategories}

The different kinds of multicategories mentioned in \cref{sec:introduction}, corresponding to logics with different structural rules, are all instances of a well-developed theory of ``generalized multicategories'' parametrized by a monad on a bicategory or double category of spans or profunctors.\footnote{See~\cite{cs:multicats} for a general framework, building on much prior work cited therein.}
This theory was used for instance in~\cite{ht:lnl-2mnd} to begin defining an analogue of \lnl polycategories for intuitionistic linear logic (see our discussion of ``\lnl multicategories'' below).
\lnl polycategories ought to be an instance of a similar theory of ``generalized polycategories'',
but unfortunately, no such general theory has been formulated yet (though~\cite{garner:polycats} provides strong evidence for its existence).
Thus, in this paper we simply give the definitions explicitly.

\begin{defi}
  A \textbf{linear-nonlinear (\lnl) polycategory} \cP consists of:
  \begin{enumerate}
  \item A set of \textbf{nonlinear objects}, which we denote by letters near the end of the Roman alphabet such as $X,Y,Z$.
    We denote finite lists of nonlinear objects by the Greek letters $\Theta,\Upsilon$.
    If $(X_1,\dots, X_m)$ is such a list and $\si : \{1,\dots,n\} \to \{1,\dots,m\}$ is a function, we write $\si : (X_1,\dots, X_m)\to (X_{\si 1},\dots, X_{\si n})$ and call it a \textbf{structural map}.\label{item:lnl-nlobj}
  \item For each $\Theta,X$, a \textbf{nonlinear hom-set} $\Pnl(\Theta;X)$ containing \textbf{nonlinear morphisms}, 
    with a functorial action by any structural map $\si : \Theta \to \Upsilon$:\label{item:lnl-mixedhom}
    \[\act{}{-}{\si} : \Pnl(\Upsilon;X) \to \Pnl(\Theta;X). \]
  \item Compositions and identities for the nonlinear hom-sets
    \[ \circ_X : \Pnl(\Theta_1,X,\Theta_2;Y) \times \Pnl(\Upsilon;X) \to \Pnl(\Theta_1,\Upsilon,\Theta_2;Y)
      \qquad 1_X \in \Pnl(X;X)\]
    satisfying the multicategory axioms and equivariant for the structural actions.\label{item:lnl-cart}
  \item A set of \textbf{linear objects}, which we denote by letters near the beginning of the Roman alphabet such as $A,B,C$.
    We denote finite lists of linear objects by the Greek letters $\Gamma,\Delta$.
    If $(A_1,\dots, A_n)$ is such a list and $\tau : \{1,\dots,n\} \toiso \{1,\dots,n\}$ is a permutation, we write $\tau : (A_1,\dots, A_n)\toiso (A_{\si 1},\dots, A_{\si n})$ and call it a \textbf{structural permutation}.
  \item For each $\Theta$ and $\Gamma,\Delta$, a \textbf{linear hom-set} $\Pl(\Theta|\Gamma;\Delta)$ containing \textbf{linear morphisms},
    with a functorial action by a structural map $\si : \Theta' \to \Theta$ and structural permutations $\tau : \Gamma'\to \Gamma$ and $\rho : \Delta \to \Delta'$:
    \[ \act{\rho}{-}{\si|\tau} : \Pl(\Theta|\Gamma;\Delta) \to \Pl(\Theta'|\Gamma';\Delta'). \]
  \item For each $A$ an identity morphism $1_A \in \Pl(|A;A)$.
  \item Composition morphisms
    \begin{alignat*}{2}
      \circ_A &:  \mathrlap{\Pl(\Theta|\Gamma_1,A,\Gamma_2;\Delta) \times \Pl(\Theta'|\Gamma';\Delta'_1,A,\Delta'_2)}\\
      &&&\too \Pl(\Theta,\Theta'|\Gamma_1,\Gamma',\Gamma_2;\Delta'_1,\Delta,\Delta'_2)\\
      \circ_X &: \Pl(\Theta_1,X,\Theta_2|\Gamma;\Delta) \times \Pnl(\Upsilon;X) &&\too \Pl(\Theta_1,\Upsilon,\Theta_2|\Gamma;\Delta)
    \end{alignat*}
    that are associative, unital, and equivariant in all reasonable ways.
    (Note that by equivariance, all the compositions are uniquely determined by those in which $\Theta_2,\Gamma_2,\Delta_2'$ are empty.)
  \end{enumerate}
\end{defi}

\begin{defi}\label{defn:lnlpoly-2cat}
  A \textbf{functor} $H:\cP\to\cQ$ between \lnl polycategories consists of functions between their linear and nonlinear objects and morphisms, preserving domains, codomains, structural actions, identities, and composites.
  A \textbf{transformation} $\al : H \Rightarrow K : \cP\to\cQ$ between functors consists of:
  \begin{enumerate}
  \item For each nonlinear object $X$ of \cP, a nonlinear morphism $\al_X \in \nonlinhom{\cQ}(H X; K X)$.
  \item For each linear object $A$ of \cP, a linear morphism $\al_A \in \linhom{\cQ}(|H A ; K A)$.
  \item For each nonlinear $f\in \Pnl(\Theta;Y)$, we have $\al_Y \circ H f = K f \circ (\al_\Theta)$.\footnote{Here if $\Theta = (X_1,\dots,X_n)$ then $K f \circ (\al_\Theta)$ denotes $(\cdots(K f \circ_{X_1} \al_{X_1}) \circ_{X_2} \al_{X_2} \cdots )\circ_{X_n} \al_{X_n} $, and similarly elsewhere.}
  \item For each linear $f\in \Pl(\Theta|\Gamma;\Delta)$, we have $(\al_\Delta) \circ H f = K f \circ (\al_\Theta\mid \al_\Gamma)$.
  \end{enumerate}
  This defines a strict 2-category $\lnlpoly$.
\end{defi}

\lnl polycategories are such a rich structure that they include many better-known structures as special cases.
(The reader unfamiliar with any of the structures mentioned below is free to take the asserted characterization as a definition.)
\begin{itemize}
\item \textbf{Symmetric polycategories} can be identified with \lnl polycategories having no nonlinear objects (and hence no nonlinear morphisms).
  These model the judgmental structure of classical multiplicative-additive linear logic.
\item \textbf{Symmetric multicategories} can be identified with \lnl polycategories having no nonlinear objects and in which all (linear) morphisms are \emph{co-unary}, i.e.\ have a codomain of length 1.
  These model the judgmental structure of intuitionistic multiplicative-additive linear logic.
\item Even more degenerately, ordinary \textbf{categories} can be identified with \lnl polycategories having no nonlinear objects and in which all (linear) morphisms are both unary and co-unary.
\item \textbf{Cartesian multicategories} can be identified with \lnl polycategories having no linear objects and no linear morphisms (here the former does not quite imply the latter, as there are homsets $\Pl(\Theta|;\,)$).
  These model the judgmental structure of intuitionistic (nonlinear) logic.
\item By an \textbf{\lnl multicategory} we will mean an \lnl polycategory in which all linear morphisms are co-unary.
  These model the judgmental structure of intuitionistic linear logic (with exponentials); they do not quite appear in the literature, though a structure like them is the goal of~\cite{ht:lnl-2mnd} (see \cref{egs:kleisli-type}).
\end{itemize}

\begin{rem}\label{rmk:slice}
  In fact, each of the above five subcategories is a slice category $\lnlpoly/\cS$ for some subterminal object $\cS$.
  The terminal object of $\lnlpoly$ has one linear object, one nonlinear object, and all hom-sets singletons; thus a subterminal object has at most one object of each sort and each hom-set a subsingleton.

  The slice category $\lnlpoly/\cS$ over a subterminal is thus the full subcategory of $\lnlpoly$ consisting of those objects \cP whose unique map to the terminal object factors through $\cS$.
  This means that \cP has only objects of the sorts that \cS does, and only morphisms of the arity and co-arity that \cS does.  

  For example, let $\mathsc{sympoly}$ be the subterminal object with one linear object, no nonlinear objects, and all linear homsets singletons.
  Then $\lnlpoly/\mathsc{sympoly}$ consists of \lnl polycategories with no nonlinear objects, i.e.\ symmetric polycategories.
  We can argue similarly for the following suggestively-named subterminals:
  \begin{itemize}
  \item $\mathsc{symmulti}$, which has one linear object, no nonlinear objects, co-unary linear homsets singletons, and others empty.
  \item $\mathsc{cat}$, which has one linear object, no nonlinear objects, and only the identity morphism.
  \item $\mathsc{cartmulti}$, which has one nonlinear object, no linear objects, all nonlinear homsets singletons, and all linear homsets empty.
  \item $\mathsc{lnlmulti}$, which has one linear object, one nonlinear object, all nonlinear homsets and co-unary linear homsets singletons, and others empty.
  \end{itemize}
  For consistency, we may write the terminal object of $\lnlpoly$ as $\mathsc{lnlpoly}$.

  We will consider other slices of $\lnlpoly$ later in the paper.
  For ease of reference, \cref{tab:subterms} on page~\pageref{tab:subterms} summarizes the definitions of all the small \lnl polycategories over which we slice.
\end{rem}

The slice category over any subterminal object \cS is coreflective, with coreflector $(-)\times \cS$.
Thus, all five of these subcategories are coreflective.
In particular, any \lnl polycategory \cP has an underlying symmetric polycategory, which we denote $\Plin$, and an underlying cartesian multicategory, which we denote $\Pnonlin$.

\begin{rem}\label{rmk:planar}
  With a little more work, we can also represent \emph{planar} (i.e.\ non-symmetric) multicategories inside $\lnlpoly$.
  Specifically, any planar multicategory \cM freely generates a symmetric multicategory $\Sigma\cM$, which has the same objects as \cM, and such that a morphism in $\Sigma\cM(\Gamma ; B)$ is a pair $(f,\si)$ where $f\in \cM(\Gamma' ; B)$ and $\si : \Gamma\toiso\Gamma'$ is a structural permutation.
  The functor $\Sigma$ thus defined from planar multicategories to symmetric multicategories (or to \lnl polycategories) is faithful but not full: the morphisms in its image are those that preserve the permutations $\si$.
  But we can enforce this condition by restriction to a suitable slice.

  Let $\mathsc{plmulti}$ be the image under $\Sigma$ of the terminal planar multicategory; thus it has one (linear) object, and its morphisms with arity $n$ and co-arity 1 are labeled by permutations of $n$ objects.
  Then each $\Sigma\cM$ comes with a canonical projection to $\mathsc{plmulti}$ that records the permutations $\si$, and a morphism $\Sigma\cM \to \Sigma\cM'$ is in the image of $\Sigma$ precisely when it commutes with these projections.
  Thus, the category of planar multicategories is equivalent to the slice category of the category of symmetric multicategories, and hence also of \lnlpoly, over $\mathsc{plmulti}$.
  Note that unlike the slices considered in \cref{rmk:slice}, $\mathsc{plmulti}$ is not subterminal, corresponding to the fact that $\Sigma$ is not full.
\end{rem}

\begin{rem}\label{rmk:planar-poly}
  An analogous construction is \emph{not} possible for planar \emph{polycategories}; freely adding symmetric actions to a planar polycategory does not yield a symmetric one, as not all composites are definable~\cite[Example 1.3]{koslowski:polycats}.
  Informally, the gap between planar and symmetric is wider in the classical case than in the intuitionistic one.
  This is one reason that in this paper we focus on the symmetric case.
\end{rem}

\begin{rem}
  As pointed out by a referee, it is natural to also wonder about \emph{cyclic} multicategories~\cite{gk:cyclic-operads,cgr:cyclic,hry:higher-cyc-opd,dh:dk-cyc-opd}.
  These behave very differently, because their cyclic action mixes domains and codomains --- generally with an involution applied to the objects --- thereby enabling them to represent morphisms with codomains of arbitrary arity as well.
  Hence, as shown in~\cite[\S7]{shulman:dialectica}, cyclic \emph{symmetric} multicategories are almost equivalent to symmetric \emph{polycategories} with strict duals (``$\ast$-polycategories''~\cite{hyland:pfthy-abs}).
  The situation with cyclic \emph{planar} multicategories is less clear, but they seem likely to be related to planar polycategories, and hence would suffer from problems akin to those in described in \cref{rmk:planar-poly}.
\end{rem}

\begin{rem}\label{rmk:double-split}
  As noted in \cref{sec:introduction}, \lnl polycategories are a semantic counterpart of ``split-context'' syntaxes such as~\cite{benton:lnl,barber:dill,girard:unity}.
  It may thus be surprising that although we are modeling \emph{classical} linear logic, we have nevertheless only split the \emph{left-hand} context, as is done in \emph{intuitionistic} linear syntaxes such as~\cite{benton:lnl,barber:dill}, rather than splitting both contexts as in~\cite{girard:unity}.
  There are two reasons for this.

  The first is that it is simpler and sufficient.
  As we will see below, even with only one split context we can still characterize \emph{both} modalities $\oc$ and $\wn$ by universal properties.
  This is a polycategorical version of the observation that to model classical linear logic it suffices to have an \lnl adjunction (which models intuitionistic linear logic) whose linear category is $\ast$-autonomous; there is no need to add a second nonlinear category.
  Moreover, most natural examples have this form anyway.

  By the way, note that the apparent asymmetry in splitting the left-hand context, rather than the right-hand one, is really just an artifact of notation.
  We could equally well write $\Pl(\Theta|\Gamma;\Delta)$ as ${\cP\big(\Gamma\mathbin{;} \Delta\mid\Theta\big)}$, reversing the direction of the nonlinear morphisms so they form a ``co-cartesian co-multicategory''.
  But splitting the left-hand context is more intuitive and remains closer to the natural examples.

  The second reason is that ``doubly-split'' \lnl polycategories, at least for one definition of such, are actually a special case of singly-split ones.
  Let $\mathsc{dblsplit}$ be the \lnl polycategory with one linear object, two nonlinear objects, and all homsets singletons.
  Then an object of the slice category $\lnlpoly/\mathsc{dblsplit}$ is an \lnl polycategory equipped with a partition of its nonlinear objects into two subsets, which we may call the ``left-hand objects'' and the ``right-hand objects''.
  Accordingly, if $\Theta$ consists of left-hand objects and $\Upsilon$ of right-hand objects, we can choose to denote the linear homset
  $\Pl(\Theta,\Upsilon | \Gamma;\Delta)$
  by
  ${\cP\big(\Theta\mid \Gamma\mathbin{;} \Delta\mid\Upsilon\big)}$.
  Similarly, if $\Upsilon$ consists of right-hand objects and $Z$ is a right-hand object, we can write the nonlinear homset $\cP\big(\Upsilon \mathbin{;} Z\big)$ as $\cP\big(Z \mathbin{;} \Upsilon\big)$, thereby regarding the right-hand objects as forming a co-cartesian co-multicategory, which acts on the linear homsets ${\cP\big(\Theta\mid \Gamma\mathbin{;} \Delta\mid\Upsilon\big)}$ on the right.

  The only possibly-surprising thing about this notion of ``doubly-split \lnl polycategory'' is that we also have ``mixed nonlinear homsets'' $\Pnl(\Theta,\Upsilon;X)$ (which might perhaps be better written $\cP\big(\Theta \mathbin{;} X \mathbin{;} \Upsilon\big)$) where $\Theta$ consists of left-hand objects, $\Upsilon$ of right-hand objects, and $X$ could be of either sort.
  However, such mixed morphisms arise naturally as the result of weakening a ``pure'' nonlinear morphism of either handedness by objects of the other handedness, and once we have these there is no reason there couldn't be other morphisms of the same sort as well (see, for instance, \cref{thm:ldc-kl}).

  Note also that there is a morphism to $\mathsc{dblsplit}$ from the terminal object $\mathsc{lnlpoly}$ (in fact, two of them), so that our category $\lnlpoly$ is also equivalent to a slice category of this category $\lnlpoly/\mathsc{dblsplit}$ of doubly-split \lnl polycategories.
  Thus, formally we could take either one as the primitive notion and define the other in terms of it.
  We have chosen the singly-split notion as primitive, since it is, as noted above, simpler and sufficient.
\end{rem}

We will see some more examples of \lnl polycategories in \cref{sec:relation-literature}, but first we define the basic universal properties that appear therein.
Inspired by~\cite{bz:bifib-poly}, we say that a morphism $\psi$ in an \lnl polycategory containing an object $R$ (linear or nonlinear) in its domain or codomain is \emph{universal in $R$} if composing along $R$ induces bijections on homsets of all possible types.
For the five possible combination of types for $\psi$ and $R$, this specializes to the following.
\begin{defi}
  Let $X$ be a nonlinear object and $A$ a linear object.
  \begin{itemize}
  \item A nonlinear morphism $\psi\in\Pnl(\Theta;X)$ is \textbf{universal in $X$} if composing with $\psi$ induces bijections
    \begin{align*}
      \Pnl(\Theta',X;Y) &\toiso \Pnl(\Theta',\Theta;Y)\\
      \Pl(\Theta',X|\Gamma;\Delta) &\toiso \Pl(\Theta',\Theta|\Gamma;\Delta).
    \end{align*}
  \item A nonlinear morphism $\psi\in\Pnl(\Theta,X;Y)$ is \textbf{universal in $X$} if composing with $\psi$ induces bijections
    \begin{align*}
      \Pnl(\Theta';X) &\toiso \Pnl(\Theta,\Theta';Y).
    \end{align*}
  \item A linear morphism $\psi \in \Pl(\Theta,X|\Gamma;\Delta)$ is \textbf{universal in $X$} if composing with $\psi$ induces bijections
    \begin{align*}
      \Pnl(\Theta';X) &\toiso \Pl(\Theta,\Theta'|\Gamma;\Delta).
    \end{align*}
  \item A linear morphism $\psi \in \Pl(\Theta|\Gamma;\Delta,A)$ is \textbf{universal in $A$} if composing with $\psi$ induces bijections
    \begin{align*}
      \Pl(\Theta'|\Gamma',A;\Delta') &\toiso \Pl(\Theta',\Theta|\Gamma',\Gamma;\Delta',\Delta).
    \end{align*}
  \item A linear morphism $\psi \in \Pl(\Theta|\Gamma,A;\Delta)$ is \textbf{universal in $A$} if composing with $\psi$ induces bijections
    \begin{align*}
      \Pl(\Theta'|\Gamma';\Delta',A) &\toiso \Pl(\Theta',\Theta|\Gamma',\Gamma;\Delta',\Delta).
    \end{align*}
  \end{itemize}
  A functor is said to \textbf{preserve} a certain kind of universal morphism if it takes any such morphism to a similarly universal morphism.
\end{defi}

Universal morphisms are unique up to unique isomorphism:

\begin{prop}\label{thm:univ-uniq}
  If $\psi \in \Pl(\Theta|\Gamma;\Delta,A)$ and $\psi' \in \Pl(\Theta|\Gamma;\Delta,A')$ are universal in $A$ and $A'$ respectively, then there is a unique isomorphism $\phi : A\cong A'$ such that $\phi \circ_A \psi = \psi'$; and similarly for other kinds of universal morphism.
\end{prop}
\begin{proof}
  As usual, $\phi$ is determined by applying the universal property of $\psi$ to $\psi'$, and conversely for its inverse.
\end{proof}

We now explore the most important cases of universality, starting with versions of the polycategorical representability conditions from~\cite{cs:wkdistrib,bz:bifib-poly}.
For clarity and conciseness, we indicate the object in which a universal morphism is universal by underlining it, e.g.\ $\psi \in \Pl(\Theta|\Gamma,\underline{A};\Delta)$.

\begin{defi}
  Let $A,B$ be linear objects in an \lnl polycategory \cP.
  \begin{itemize}
  \item A \textbf{tensor product} of $A,B$ is a universal morphism $\psi \in \Pl(|A,B;\underline{A\ten B})$.
  \item A \textbf{cotensor product} of $A,B$ is a universal morphism $\psi \in \Pl(|\underline{A\coten B}; A,B)$.
  \item A \textbf{unit} $\unit$ is a universal morphism $\psi\in \Pl(|;\,\underline{\unit})$.
  \item A \textbf{counit} $\counit$ is a universal morphism $\psi\in \Pl(|\underline{\counit};)$.
  \item A \textbf{dual} of $A$ is a universal morphism $\psi\in \Pl(|A,\underline{\d A};)$.
  \end{itemize}
  We say that \cP ``\textbf{has $\ten$}'' if any $A,B$ have a tensor product, and so on.
\end{defi}

A dual is equivalently a universal morphism $\psi\in \Pl(|;\,A,\underline{\d A})$;
see e.g.~\cite{bz:bifib-poly}.

These universal properties specialize in the case $\Theta=\emptyset$ to the like-named ones in the symmetric polycategory $\Plin$.
Thus, as shown in~\cite{cs:wkdistrib,bz:bifib-poly}, if an \lnl polycategory has all $\ten,\coten,\unit,\counit$ then $\Plin$ is a \textbf{linearly distributive category}, and if it also has all $\duals$ then $\Plin$ is \textbf{$\ast$-autonomous}~\cite{barr:staraut,barr:staraut-ll,cs:wkdistrib}.

We similarly have tensors and units of \emph{nonlinear} objects, but these turn out to coincide with cartesian \emph{products}, by the following folklore analogue of the equivalence between positive and negative presentations of product types in structural logic.

\begin{prop}\label{thm:cartprod}
  The following are equivalent for objects $X,Y$ and $X\times Y$ of an \lnl polycategory.
  \begin{enumerate}
  \item There is a universal morphism $\psi \in \Pnl(X,Y;\underline{X\times Y})$.
    In other words, composing with $\psi$ induces bijections\label{item:cart3}
    \begin{align*}
      \Pnl(\Theta,X\times Y;Z)&\toiso \Pnl(\Theta,X,Y;Z)\\
      \Pl(\Theta,X\times Y|\Gamma;\Delta) &\toiso \Pl(\Theta,X,Y|\Gamma;\Delta).
    \end{align*}
  \item There is a morphism $\psi \in \Pnl(X,Y;X\times Y)$ inducing bijections\label{item:cart2}
    \begin{align*}
      \Pnl(\Theta,X\times Y;Z)&\toiso \Pnl(\Theta,X,Y;Z)
    \end{align*}
  \item There are $\pi_1 \in \Pnl(X\times Y;X)$ and $\pi_2 \in \Pnl(X\times Y;Y)$ inducing bijections\label{item:cart1}
    \[ \Pnl(\Theta;X\times Y) \toiso \Pnl(\Theta;X) \times \Pnl(\Theta;Y). \]
  \item There are morphisms $\psi \in \Pnl(X,Y;X\times Y)$ and $\pi_1 \in \Pnl(X\times Y;X)$ and $\pi_2 \in \Pnl(X\times Y;Y)$ such that the composites
    \begin{mathpar}
      (X,Y) \xto{\psi} X\times Y \xto{\pi_1} X\and
      (X,Y) \xto{\psi} X\times Y \xto{\pi_2} Y\and
      (X\times Y, X\times Y) \xto{(\pi_1,\pi_2)} (X,Y) \xto{\psi} X\times Y
    \end{mathpar}
    are the image of identities under structural maps.\label{item:cart4}
  \end{enumerate}
\end{prop}
\begin{proof}
  Of course~\ref{item:cart3} implies~\ref{item:cart2}, so it suffices to prove that~\ref{item:cart2} and~\ref{item:cart1} each imply~\ref{item:cart4} and that~\ref{item:cart4} implies~\ref{item:cart3} and~\ref{item:cart1}.

  Assuming~\ref{item:cart2}, let $\pi_1:X\times Y \to X$ be the image of $1_X$ under the composite
  \[ \Pnl(X;X) \to \Pnl(X,Y;X) \toiso \Pnl(X\times Y;X), \]
  of a structural map and the universal property of~\ref{item:cart2}, and similarly for $\pi_2$.
  The equations in~\ref{item:cart4} hold by the universal property.

  Assuming~\ref{item:cart1}, $\psi : (X,Y) \to X\times Y$ is the image of $(1_X,1_Y)$ under the composite
  \[ \Pnl(X;X) \times \Pnl(Y;Y) \to \Pnl(X,Y;X) \times \Pnl(X,Y;Y) \to \Pnl(X,Y;X\times Y) \]
  of structural maps with the universal property of~\ref{item:cart1}.
  Again, the equations in~\ref{item:cart4} hold by the universal property.

  Conversely, assuming~\ref{item:cart4}, the right-to-left directions of~\ref{item:cart3} are composing with $(\pi_1,\pi_2)$ and a structural map, while the right-to-left direction of~\ref{item:cart1} is composing with $\psi$ and a structural map.
  These are inverses by the equations in~\ref{item:cart4}.
\end{proof}

We will refer to such an $X\times Y$ as a \textbf{product} of $X$ and $Y$.
There is an analogue for nullary products and terminal nonlinear objects, denoted $1$ (not to be confused with the linear $\unit$).
By \cref{thm:cartprod}\ref{item:cart1}, if all $\times,1$ exist then $\Pnonlin$ is a \textbf{cartesian monoidal category}.
Note that these are essentially facts about cartesian multicategories, which extend automatically to an \lnl polycategory \cP from $\Pnonlin$.

\begin{cor}
  Any functor of \lnl polycategories preserves nonlinear products and terminal objects.
\end{cor}
\begin{proof}
  The equations in \cref{thm:cartprod}\ref{item:cart4} are preserved by any functor.
\end{proof}

\begin{rem}
  If we changed notation as suggested in \cref{rmk:double-split} to regard the nonlinear objects (or the ``right-hand'' ones) as instead forming a co-cartesian co-multicategory, then the identical operations $\times$ and $1$ would instead behave like a coproduct and an initial object (and hence would be better denoted $+$ and $\varnothing$).
\end{rem}

We now consider the \textbf{exponential modalities} (a.k.a.\ \textbf{storage modalities}) that relate linear and nonlinear objects.

\begin{defi}
  Let $X$ be a nonlinear object and $A$ a linear one.
  \begin{itemize}
  \item An \textbf{$\foc$-modality} is a universal morphism $\psi\in\Pl(X|;\,\underline{\foc X})$.
  \item A \textbf{$\uoc$-modality} is a universal morphism $\psi\in\Pl(\underline{\uoc A}|;\,A)$.
  \item An \textbf{$\fwn$-modality} is a universal morphism $\psi\in\Pl(X|\underline{\fwn X};)$.
  \item A \textbf{$\uwn$-modality} is a universal morphism $\psi\in\Pl(\underline{\uwn A}|A;)$.
  \end{itemize}
\end{defi}

Thus, the exponential modalities are characterized by natural bijections
\begin{alignat*}{2}
  \Pl(\Theta,X|\Gamma;\Delta) &\cong \Pl(\Theta|\Gamma,\foc X;\Delta) &\qquad
  \Pl(\Theta|{};\,A) &\cong \Pnl(\Theta;\uoc A)\\
  \Pl(\Theta,X|\Gamma;\Delta) &\cong \Pl(\Theta|\Gamma;\Delta,\fwn X) &\qquad
  \Pl(\Theta|A;) &\cong \Pnl(\Theta;\uwn A).
\end{alignat*}
Note that $\foc$ and $\uoc$ are covariant, while $\fwn$ and $\uwn$ are contravariant.
We will see below that these are adjoint in pairs, $\foc\adj\uoc$ and $\uwn\adj\fwn$, and induce the usual comonad $\oc = \foc\uoc$ and monad $\wn = \fwn\uwn$.

We can also consider internal-homs of various sorts.

\begin{defi}
  Let $X,Y$ be nonlinear objects and $A,B$ be linear objects.
  \begin{itemize}
  \item A \textbf{linear hom} is a universal morphism $\psi\in\Pl(|\underline{A\hom B}, A; B)$.
  \item A \textbf{linear co-hom} is a universal morphism $\psi\in\Pl(|B; \underline{B\cohom A}, A)$.
  \item A \textbf{nonlinear hom} is a universal morphism $\psi\in\Pnl(\underline{X\to Y}, X; Y)$.
  \item A \textbf{mixed hom} is one of the following:\footnote{As notational mnemonics,
      the arrowhead in $\to,\mixedhom,\mixedeechom$ indicates the domain object is nonlinear,
      the open circle in $\hom,\mixedhom$ indicates the codomain object and hom-object are both linear,
      and the closed circle in $\eechom,\mixedeechom$ indicates the codomain object is linear but the hom-object is nonlinear.}
    \begin{itemize}
    \item a universal morphism $\psi\in\Pl(X | \underline{X\mixedhom B}; B)$.
    \item a universal morphism $\psi\in\Pl(\underline{A\eechom B}| A ; B)$.
    \item a universal morphism $\psi\in\Pl(\underline{X\mixedeechom B} , X | ;\, B)$.
    \end{itemize}
  \end{itemize}
\end{defi}

Thus, these various kinds of homs are characterized by bijections
\begin{align*}
  \Pl(\Theta|\Gamma,A;\Delta,B) &\cong \Pl(\Theta|\Gamma;\Delta,A\hom B)\\
  \Pl(\Theta|\Gamma,B;\Delta,A) &\cong \Pl(\Theta|\Gamma,B\cohom A;\Delta)\\
  \Pnl(\Theta,X;Y) &\cong \Pnl(\Theta;X\to Y)\\
  \Pl(\Theta,X | \Gamma ;\Delta, B) &\cong \Pl(\Theta|\Gamma;\Delta, X \mixedhom B)\\
  \Pl(\Theta | A ; B) &\cong \Pnl(\Theta ; A \eechom B)\\
  \Pl(\Theta,X | ;\, B) &\cong \Pnl(\Theta ; X \mixedeechom B).
\end{align*}
In particular:
\begin{itemize}
\item If $\ten,\unit,\hom$ exist then the monoidal structure $\ten$ on $\Plin$ is closed.
\item If $\coten,\counit,\cohom$ exist then the monoidal structure $\coten$ on $\Plin$ is coclosed.
\item If $\times,1,\to$ exist then $\Pnonlin$ is cartesian closed.
\end{itemize}

The mixed homs suggest analogous \textbf{mixed tensor products}, such as universal morphisms $\psi\in \Pl(X | A ; \underline{X\rtimes A})$,
or $\psi\in\Pl(X,Y | ; \, \underline{X\boxtimes Y})$.
However, lest we start to feel the zoo of universal properties is too large, we note that the more exotic sorts can be constructed from the simpler ones in the following sense.

\begin{prop}\label{thm:univ-comp}
  If $\psi$ is universal in $R$, while $\phi$ contains $R$ in its domain or codomain and is universal in a different object $S$, then $\psi\circ_R \phi$ is universal in $S$.
\end{prop}
\begin{proof}
  There are a number of different versions of this statement depending on the types of $R,S,\psi,\phi$ and whether the objects occur in domain or codomain, but they all reduce to ``the composite of bijections is a bijection''.
  See \cref{thm:univ-comp2} for a more rigorous proof.
\end{proof}

One instance of this is the associativity of tensors: given universal morphisms
\begin{alignat*}{2}
  \psi_1 &\in \Pl( | A,B; \underline{A\ten B}) &\qquad
  \psi_3 &\in \Pl( | A\ten B, C; \underline{(A\ten B)\ten C}) \\
  \psi_2 &\in \Pl(| B,C; \underline{B\ten C}) &\qquad
  \psi_4 &\in \Pl( | A, B\ten C; \underline{A\ten (B\ten C)})
\end{alignat*}
the two composites
\begin{align*}
  \psi_3 \circ_{A\ten B} \psi_1 &\in \Pl(| A,B,C; \underline{(A\ten B)\ten C}) \\
  \psi_4 \circ_{B\ten C} \psi_2 &\in \Pl(|A,B,C; \underline{A\ten (B\ten C)})
\end{align*}
are both universal, hence by \cref{thm:univ-uniq} there is an induced isomorphism
\[(A\ten B)\ten C \cong A\ten (B\ten C).\]
This is how $(\ten,\unit)$ is shown to be a monoidal structure, and similarly for $(\coten,\counit)$ and (if we like) $(\times,1)$.

Another familiar instance is that in a $\ast$-autonomous category, linear homs can be defined in terms of duals and cotensors if these exist.
Given universal morphisms
\begin{mathpar}
  \psi_1 \in \Pl(| \underline{\d A},A;) \and
  \psi_2 \in \Pl(| \underline{\d A \coten B}; \d A,B)
\end{mathpar}
their composite $\psi_1 \circ_{\d A} \psi_2 \in \Pl(|\underline{\d A \coten B}, A;B)$ is universal in $\d A \coten B$, exhibiting it as $A\hom B$.
Similarly, we have $B \cohom A = \d A \ten B$, and De Morgan duality:
\begin{mathpar}
  A \coten B = \d{(\d A\ten \d B)}
  \and
  \counit = \d \unit
  \and
  \fwn X = \d{(\foc X)}
  \and
  \uwn A = \uoc(\d A)
\end{mathpar}
In particular, $\Plin$ is $\ast$-autonomous as soon as $\cP$ has $\ten,\unit,\duals$.
And as in a $\ast$-autonomous category, duals can be constructed by homming into the counit:
\[ \d A = A \hom \counit . \]

Less familiar instances of \cref{thm:univ-comp} relate the modalities to the tensors and homs, particularly the mixed ones: we have
\begin{alignat*}{2}
  X \mixedhom B &= \foc X \hom B &\qquad
  X \rtimes A &= \foc X \ten A \\
  A \eechom B &= \uoc(A\hom B) &\qquad
  X\boxtimes Y &= \foc(X\times Y) \\
  X \mixedeechom B &= \uoc (\foc X \hom B) &\qquad
  X\boxtimes Y &= \foc X \ten \foc Y\\
  X \mixedeechom B &= X \to \uoc B &\qquad
  \unit &= \foc 1\\
  \uoc A &= \unit \eechom A &\qquad
  \foc X &= X \rtimes \unit\\
  \uoc A &= 1 \mixedeechom A &\qquad
  \foc X &= X \boxtimes 1
\end{alignat*}
whenever all the operations on the right-hand side exist.
In particular, since both $\foc(X\times Y)$ and $\foc X \ten \foc Y$ have the universal property of $X\boxtimes Y$, they are isomorphic if they both exist.
(This is, of course, closely related to Seely's characterization of the modality $\oc$; see \cref{rmk:seely}.)
Thus, if $\ten,\unit,\times,1,\foc$ exist then $\foc$ is a strong monoidal functor.
Similarly, if both $\uoc (\foc X \hom B)$ and $X \to \uoc B$ exist they are isomorphic (which is related to Girard's embedding of nonlinear logic in linear logic); if $\fwn(X\times Y)$ and $\fwn X \coten \fwn Y$ exist they are isomorphic; and so on.

\begin{rem}\label{thm:univ-iso}
  As a trivial instance, a unary co-unary linear morphism, i.e.\ one of the form $\psi \in \Pl( | A; B)$, is universal if and only if it is an isomorphism (and similarly in the nonlinear case).
  Thus, \cref{thm:univ-comp} also implies that universal morphisms are stable under composition with isomorphisms, conversely to \cref{thm:univ-uniq}.
\end{rem}

We can also consider limits and colimits in \lnl polycategories.
In general, we require a \textbf{limit} of a diagram of linear or nonlinear objects (and unary co-unary morphisms) to induce bijections on all hom-sets where it appears in the codomain, and similarly for a \textbf{colimit} whenever it appears in the domain.
(In the case of products and coproducts, this definition appears in~\cite{pastro:sp-poly}.)
The simplest case of this is that a limit of nonlinear objects satisfies
\begin{equation}
  \Pnl(\Theta;\lim_i X_i)\cong \lim_i \Pnl(\Theta;X_i),\label{eq:nllim}
\end{equation}
generalizing \cref{thm:cartprod}\ref{item:cart1} and reducing to an ordinary limit in the cartesian monoidal $\Pnonlin$ if $\times,1$ exist.
However, a colimit of nonlinear objects satisfies both
\begin{align}
  \Pnl(\Theta,\colim_i X_i;Y) &\cong \lim_i \Pnl(\Theta,X_i;Y)\label{eq:nlcolim1}\\
  \Pl(\Theta,\colim_i X_i|\Gamma;\Delta) &\cong \lim_i \Pl(\Theta,X_i|\Gamma;\Delta)\label{eq:nlcolim2}
\end{align}
induced by the same universal cocone.
This implies that the colimit is
\begin{enumerate}
\item preserved in each variable by $\times$, insofar as $\times$ exists;
\item sent by $\foc$ to a colimit in $\Plin$ that is preserved in each variable by $\ten$, insofar as $\foc,\ten$ exist; and
\item sent by $\fwn$ to a limit in $\Plin$ that is preserved in each variable by $\coten$, insofar as $\fwn,\coten$ exist.
\end{enumerate}
Moreover, if all $\times,\foc,\fwn,\ten,\coten$ exist, then a colimit in the ordinary category $\Pnonlin$ is a colimit in $\cP$ if and only if it is preserved in these ways.

Similarly, a colimit of linear objects satisfies
\begin{equation}
  \Pl(\Theta|\Gamma,\colim_i A_i;\Delta) \cong \lim_i \Pl(\Theta|\Gamma,A_i;\Delta)\label{eq:lcolim}
\end{equation}
which implies that it is preserved by $\ten$ in each variable and sent by $\uwn$ to a limit in $\Pnonlin$, insofar as $\ten,\uwn$ exist.
If all $\ten,\coten,\counit,\foc$ exist, then a colimit in the ordinary category $\Plin$ is a colimit in \cP if and only if it is preserved by $\ten$.
Dually, a limit of linear objects satisfies
\begin{equation}
  \Pl(\Theta|\Gamma;\Delta,\lim_i A_i) \cong \lim_i \Pl(\Theta|\Gamma;\Delta,A_i)\label{eq:llim}
\end{equation}
which implies that it is preserved by $\coten$ in each variable and sent by $\uoc$ to a limit in $\Pnonlin$, insofar as $\coten,\uoc$ exist.
And if all $\coten,\ten,\unit,\foc$ exist, a colimit in $\Plin$ is a colimit in \cP if and only if it is preserved by $\coten$.
Note also that $\ten$ preserves all colimits if $\hom$ exists, $\foc$ preserves all colimits if $\uoc$ exists, and so on.

We will write $X+Y$ for the coproduct of nonlinear objects and $\varnothing$ for the initial nonlinear object, and we denote finite products and coproducts of linear objects with Girard's notation for the linear logic additive connectives: $A\with B$ for the product, $A\oplus B$ for the coproduct, $\top$ for the terminal object, and $0$ for the initial object.
Thus the above preservation properties state that
\begin{alignat*}{2}
  X\times (Y+Z) &\cong (X\times Y) + (X\times Z)&\qquad
  X \times \varnothing &\cong \varnothing\\
  \foc(X+Y) &\cong \foc X \oplus \foc Y&\qquad
  \foc\varnothing&\cong 0\\
  \fwn(X+Y) &\cong \fwn X \with \fwn Y&\qquad
  \fwn\varnothing&\cong \top\\
  A \ten (B\oplus C) &\cong (A\ten B) \oplus (A\ten C)&\qquad
  A\ten 0 &\cong 0\\
  \uwn (A\oplus B) &\cong \uwn A \times \uwn B&\qquad
  \uwn 0 &\cong 1\\
  A \coten (B\with C) &\cong (A\coten B) \with (A\coten C)&\qquad
  A \coten \top &\cong \top\\
  \uoc (A\with B) &\cong \uoc A \times \uoc B&\qquad
  \uoc \top &\cong 1
\end{alignat*}

If we specialize the above universal properties to symmetric polycategories, symmetric multicategories, cartesian multicategories, or \lnl multicategories, there are three possible results.
Some universal properties make sense unmodified, such as $\ten,\coten$ in polycategories or $\times,\to$ in cartesian multicategories.
Others make no sense at all, such as $\coten,\counit$ in \lnl multicategories or $\foc,\uoc$ in symmetric polycategories.

A third group can only have a restricted universal property.
Specifically, limits and colimits in a symmetric multicategory or \lnl multicategory can only induce bijections of hom-sets with unary codomain: instead of~\eqref{eq:nlcolim2}--\eqref{eq:llim} we assert only
\begin{align*}
  \Pl(\Theta,\colim_i X_i|\Gamma;B) &\cong \lim_i \Pl(\Theta,X_i|\Gamma;B)\\
  \Pl(\Theta|\Gamma,\colim_i A_i;B) &\cong \lim_i \Pl(\Theta|\Gamma,A_i;B)\\
  \Pl(\Theta|\Gamma;\lim_i A_i) &\cong \lim_i \Pl(\Theta|\Gamma;A_i).
\end{align*}
Since the left- and right-hand sides of~\eqref{eq:nlcolim2}--\eqref{eq:llim} have the same codomain arity, these apparently-weaker universal properties are equivalent to~\eqref{eq:nlcolim2}--\eqref{eq:llim} for limits and colimits over \emph{nonempty} domain categories.
But the limit of the empty diagram of copies of the empty set is no longer empty, so an initial or terminal object in an \lnl multicategory \cE (in the above sense) need not be initial or terminal in \cE \emph{qua} \lnl polycategory.

In fact, an \lnl multicategory \emph{cannot} have a terminal linear object, or an initial linear or nonlinear object, in the \lnl-polycategorical sense.
For example, if $\top$ is a terminal linear object, we must have $\Pl(\Theta | \Gamma ; \Delta,\top) = 1$ for \emph{all} $\Delta$, whereas in an \lnl multicategory we have $\Pl(\Theta | \Gamma ; \Delta,\top) = \emptyset$ if $|\Delta|>0$.
This is already the case for ordinary multicategories and polycategories.

The categorization of universal properties in these four subcategories into these three groups is shown in \cref{fig:univprop}.

\begin{table}
  \begin{tabular}{c|c|c|c}
    & Unmodified & Nonsensical & Modified\\\hline
    polycategories &
    \makecell{$\ten,\unit,\coten,\counit,\duals$,\\$\hom,\cohom,\with,\oplus,\top,0$} &
    $\times,1,\to,\foc,\uoc,\fwn,\uwn,+,\varnothing$ &
    \\\hline
    symm.\ multi. &
    $\ten,\unit,\hom,\with,\oplus$&
    \makecell{$\coten,\counit,\duals,\cohom,\times,1,\to$,\\$\foc,\uoc,\fwn,\uwn,+,\varnothing$} &
    $\top,0$\\\hline
    cart.\ multi. &
    $\times,1,\to,+,\varnothing$&
    \makecell{$\ten,\unit,\coten,\counit,\duals,\hom,\cohom$,\\$\foc,\uoc,\fwn,\uwn,\with,\oplus,\top,0$} &
    \\\hline
    \lnl multi. &
    \makecell{$\times,1,\to,\ten,\unit,\hom$,\\$\with,\oplus,\foc,\uoc$}&
    $\coten,\bot,\duals,\cohom,\fwn,\uwn$ &
    $\top,0$
  \end{tabular}~\\[1em]{}
\caption{Universal properties in subcategories}
\label{fig:univprop}
\end{table}

\section{Relation to the literature}
\label{sec:relation-literature}

By our observations in \cref{sec:lnl-polycategories}, the following categorical structures can be identified with certain \lnl polycategories:
\begin{itemize}
\item Symmetric monoidal categories.
\item Symmetric monoidal categories with any desired limits, and any desired colimits that are preserved in each variable by the tensor product.
\item Closed symmetric monoidal categories, with any desired limits and colimits (the latter automatically preserved by the tensor product, due to closedness).
\item Cartesian monoidal categories.
\item Cartesian monoidal categories with any desired limits, and any desired colimits that are preserved in each variable by the cartesian product.
\item Cartesian closed categories, with any desired limits and colimits.
\item Symmetric linearly distributive categories.
\item Symmetric linearly distributive categories with any desired colimits that are preserved in each variable by the tensor product, and any desired limits that are preserved in each variable by the cotensor product.
\item (Symmetric) $\ast$-autonomous categories, with any desired limits and colimits.
\end{itemize}
The ``strong'' morphisms between these structures (those that preserve all the asserted categorical structure up to coherent isomorphisms) can also be identified with functors of \lnl polycategories that preserve the relevant universal properties, and similarly for the transformations.
In other words, the standard 2-categories of the above structures are equivalent to locally full sub-2-categories of \lnlpoly.

We now add the modalities, starting with the ``intuitionistic'' case of \lnl multi\-categories.
These are designed to model split-context intuitionistic linear logic syntaxes such as~\cite{benton:lnl,barber:dill}, without necessarily assuming that any connectives exist.
But if enough connectives do exist, they reduce to a better-known notion of model for intuitionistic multiplicative-exponential linear logic:

\begin{prop}\label{thm:lnl-adj}
  An \lnl multicategory in which the modality $\foc$ exists is uniquely determined by a functor of symmetric multicategories
  \[ \foc : \Pnonlin \to \Plin \]
  where $\Pnonlin$ is a cartesian multicategory and $\Plin$ a symmetric one.
  Moreover:
  \begin{enumerate}
  \item The modality $\uoc$ also exists if and only if the functor $\foc$ has a right adjoint (in the 2-category of symmetric multicategories).\label{item:lnlmulti1}
  \item If $\times,1,\ten,\unit$ exist, then $\foc$ is equivalently a strong symmetric monoidal functor from a cartesian monoidal category to a symmetric monoidal one.\label{item:lnlmulti2}
  \item Thus, an \lnl multicategory with $\times,1,\ten,\unit,\foc,\uoc$ is equivalently an \textbf{\emph{\lnl adjunction}} \cite{benton:lnl,mellies:catsem-ll}: a symmetric monoidal adjunction from a cartesian monoidal category to a symmetric monoidal one.\label{item:lnlmulti3}
  \end{enumerate}
\end{prop}
\begin{proof}
  Given the modality $\foc$, we make it a functor by composing with $(Y\mid\,) \to \foc Y$ and applying its universal property:
  \[ \Pnl(X_1,\dots,X_n;Y) \to \Pl(X_1,\dots,X_n|{};\, \foc Y) \toiso \Pl(| \foc X_1,\dots,\foc X_n ; \foc Y). \]
  Conversely, given a functor $\foc$, we define the general linear hom-sets by
  \[\Pl(X_1,\dots,X_n | \Gamma;B) = \Plin(\foc X_1,\dots,\foc X_n,\Gamma\mathbin{;}B).\]
  Thus, the universal property of $\foc$ holds by definition.
  Statement~\ref{item:lnlmulti1} is then a multicategorical version of the standard equivalence between adjunctions defined with bijections of hom-sets and with unit and counit.
  We have already noted~\ref{item:lnlmulti2}, and~\ref{item:lnlmulti3} follows immediately.
\end{proof}

\begin{rem}
Benton~\cite{benton:lnl} assumed $\Pnonlin$ cartesian \emph{closed} and $\Plin$ symmetric monoidal \emph{closed}, but later authors such as~\cite{mellies:catsem-ll} have observed that this is unnecessary for the bare definition.
If both categories are closed we will speak of a \textbf{closed \lnl adjunction}.

Since left adjoints preserve colimits and right adjoints preserve limits, the following structures also form locally full sub-2-categories of \lnlpoly:
\begin{itemize}
\item \lnl adjunctions.
\item \lnl adjunctions with any desired limits and colimits in either category, such that colimits are preserved by the product or tensor product in each variable.
\item Closed \lnl adjunctions, with any desired limits and colimits in either category.
\end{itemize}
\end{rem}

The notion of \lnl adjunction does depend on having both $\ten$ and $\times$, whereas \lnl multicategories can specify the correct behavior of $\foc$ and $\uoc$ even if $\ten,\times$ may not exist.
As evidence for this correctness, we note that $\times,1$ are not necessary for the induced comonad on $\lin\cP$ to coincide with a structure also existing in the literature.

\begin{prop}\label{thm:linexpcmnd}
  If \cP is an \lnl multicategory with $\ten,\unit,\foc,\uoc$, the symmetric monoidal category $\Plin$ admits a \textbf{\emph{linear exponential comonad}} \cite{bbph:terms-ill,hs:glue-orth-ll}, i.e.\ it is a \textbf{\emph{linear category}} in the sense of~\cite{benton:lnl}.
\end{prop}
\begin{proof}
  Let $\oc$ be the comonad $\foc\uoc$.
  To give the map $\oc A \ten \oc B \to \oc (A\ten B)$, we act on the $\ten$-universal morphism $(\,\mid A,B) \to A\ten B$ as follows.
  The two noninvertible maps are composition with the $\uoc$-universal morphisms $(\uoc A\mid\,) \to A$ and $(\uoc B\mid\,) \to B$ and with the $\foc$-universal morphism $(\uoc(A\ten B) \mid\, ) \to \foc\uoc(A\ten B)$:
  \begin{alignat*}{2}
    \Pl(|A,B;A\ten B)
    &\to&\;& \Pl(\uoc A, \uoc B | {}; \, A\ten B)\\
    &\toiso&\;& \Pnl(\uoc A, \uoc B;\uoc{(A\ten B)})\\
    &\to&\;& \Pl(\uoc A,\uoc B |{} ;\, \foc\uoc {(A\ten B)})\\
    &\toiso&\;& \Pl(|\foc\uoc A, \foc\uoc B ; \foc\uoc {(A\ten B)})\\
    &\toiso&\;& \Pl(|\foc\uoc A \ten \foc\uoc B ; \foc\uoc {(A\ten B)}).
  \end{alignat*}
  Similarly, to give the map $\oc A \to \oc A \ten \oc A$ we act on the $\ten$-universal morphism $(\oc A, \oc A) \to \oc A \ten \oc A$ as follows.
  The two noninvertible maps are composition with the $\foc$-universal morphism $(\uoc A \mid \,) \to \foc\uoc A = \oc A$ and a structural map.
  \begin{alignat*}{2}
    \Pl(|\oc A, \oc A;\oc A \ten \oc A)
    &=&\;& \Pl(|\foc\uoc A, \foc \uoc A; \oc A \ten \oc A)\\
    &\to&\;& \Pl(\uoc A, \uoc A| {};\, \oc A \ten \oc A)\\
    &\to&\;& \Pl(\uoc A| {};\, \oc A \ten \oc A)\\
    &\toiso&\;& \Pl(|\foc\uoc A; \oc A \ten \oc A).
  \end{alignat*}
  The nullary cases are similar, and the axioms follow by universal properties.
\end{proof}

This implication for \lnl adjunctions was observed in~\cite[\S2.2.1]{benton:lnl}; \lnl multicategories give a way to state and prove it even in the absence of $\times,1$.
Conversely:

\begin{prop}
  The Eilenberg--Moore adjunction of any linear exponential comonad $\oc$ determines an \lnl multicategory with $\times,1,\ten,\unit,\foc,\uoc$, whose underlying linear exponential comonad recovers the given $\oc$.
\end{prop}
\begin{proof}
  Such an Eilenberg--Moore adjunction is an \lnl adjunction (see~\cite[\S2.2.2]{benton:lnl} and \cite[\S7]{mellies:catsem-ll}), hence an \lnl multicategory with $\times,1,\ten,\unit,\foc,\uoc$.
\end{proof}

Moreover, since \emph{any subset of objects of a multicategory determines a sub-multi\-category} (in stark contrast to the situation for monoidal categories), we still obtain an \lnl multicategory with $\ten,\unit,\foc,\uoc$ if we restrict to any subset of the $\oc$-coalgebras containing the cofree ones.
The smallest choice, of course, consists of exactly the cofree coalgebras, so we have:

\begin{cor}\label{thm:klesli}
  The Kleisli adjunction of any linear exponential comonad $\oc$ determines an \lnl multicategory with $\ten,\unit,\foc,\uoc$, whose underlying linear exponential comonad recovers the given $\oc$.\qed
\end{cor}

\begin{rem}\label{rmk:seely}
To include the Kleisli adjunction in the case when both categories are required to be monoidal, one has to assume that cofree coalgebras are closed under products.
This follows for instance if the original monoidal category has products~\cite[\S2.2.3]{benton:lnl}, in which case we recover the notion of \textbf{Seely comonad}, characterized by $\oc A \ten \oc B \cong \oc (A\with B)$.
But \lnl polycategories allow us to include the Kleisli case even when $\with$ doesn't exist.

There are also intermediate choices between the Eilenberg--Moore category (all coalgebras) and Kleisli category (cofree coalgebras), such as the category of finite products of cofree coalgebras (if \cL has finite products), or category of exponentiable coalgebras (if \cL is closed monoidal), as discussed in~\cite[\S2.2.2]{benton:lnl}.
\end{rem}

Here is another situation that \lnl polycategories allow us to treat more generally.

\begin{exa}
  Let \cE be a symmetric multicategory; we can enhance it to an \lnl multicategory with $\foc$ by taking the nonlinear objects to be the commutative comonoids in \cE.
  It may not be immediately obvious how to define a comonoid in a multicategory that lacks $\ten$, but it is possible: $C$ is a comonoid when it is equipped with operations
  \begin{align*}
    \nonlinhom\cE(\Theta_1,C,C,\Theta_2;B) &\to \nonlinhom\cE(\Theta_1,C,\Theta_2;B)\\
    \nonlinhom\cE(\Theta_1,\Theta_2;B) &\to \nonlinhom\cE(\Theta_1,C,\Theta_2;B)
  \end{align*}
  that are associative, unital, and appropriately natural and equivariant.
  Such cocommutative comonoids form a cartesian multicategory with a forgetful multicategory functor to \cE, so by \cref{thm:lnl-adj} it yields an \lnl multicategory.

  If \cE is symmetric monoidal, then cocommutative comonoids form a cartesian monoidal category, so this \lnl multicategory has $\times,1,\ten,\unit,\foc$.
  Thus, if $\foc$ has a right adjoint $\uoc$, i.e.\ if cofree cocommutative comonoids exist, then it is an \lnl adjunction, known as a \textbf{Lafont category}~\cite{lafont:thesis} or a \textbf{free exponential modality}~\cite{mtt:free-exp}.
  But we get an \lnl multicategory even without these assumptions.
\end{exa}

In general, given a category with a linear exponential comonad, we prefer to regard it as an \lnl multicategory via the Kleisli construction rather than the Eilenberg--Moore construction.
The reason for this is the following folklore observation, showing that Kleisli adjunctions can be detected by a purely intrinsic condition:

\begin{lem}\label{thm:kleisli-eso}
  An adjunction $F : \cA \toot \cB : G$ is equivalent to the Kleisli adjunction of the monad $GF$ if and only if its left adjoint $F$ is essentially surjective on objects, and isomorphic to that Kleisli adjunction if and only if $F$ is bijective on objects.
\end{lem}
\begin{proof}
  The ``only if'' direction is clear, so suppose $F$ is essentially surjective on objects, and let $F_T : \cA \toot \cA_T : G_T$ be the Kleisli adjunction of the monad $T=GF$.
  Thus the objects of $\cA_T$ are formal copies ``$A_T$'' of the objects $A\in \cA$, with $\cA_T(A_T,B_T) = \cA(A,T B)$.
  There is a unique comparison functor $H:\cA_T \to \cB$  defined by $H(A_T) = F A$, which is essentially surjective on objects since $F$ is (and bijective on objects if $F$ is).
  But it is also fully faithful, since $\cB(F A, F B) \cong \cA(A, G F B) = \cA(A, T B) = \cA_T(A_T,B_T)$; hence it is an equivalence.
\end{proof}

Thus, applying the Kleisli construction, we have the following locally full sub-2-cat\-e\-gories of \lnlpoly:
\begin{itemize}
\item Symmetric monoidal categories with linear exponential comonad.
  This includes Seely comonads (if the category has finite products) and Lafont comonads (if cofree cocommutative comonoids exist).
\item Symmetric monoidal categories with linear exponential comonad and any desired limits and any desired colimits preserved by the tensor product in each variable.
\item Closed symmetric monoidal categories with linear exponential comonad and any desired limits and colimits.
\end{itemize}
In each case the ``strong'' morphisms, corresponding to functors of \lnl multicategories that preserve (among other things) the exponential modalities $\foc,\uoc$, are those that preserve the comonad up to coherent isomorphism: $F(\oc A) \cong \oc (F A)$.

Note that all of these \lnl polycategories have the following property.

\begin{defi}\label{defn:kleisli-type}
  An \lnl polycategory is of \textbf{Kleisli type} if it is equipped with a choice of $\uoc$ that is bijective on objects.
\end{defi}

\lnl multicategories of Kleisli type correspond to syntaxes for intuitionistic linear logic that have only one class of type, such as~\cite{barber:dill,hasegawa:cll-imp}, rather than two syntactic classes for ``linear types'' and ``nonlinear types''.

\begin{exa}\label{egs:kleisli-type}
  We conjecture that the \textbf{Linear Non-Linear multicategories} suggested by~\cite{ht:lnl-2mnd} are equivalent to \lnl multicategories of Kleisli type.
  In addition, the \textbf{IL-indexed categories} of~\cite{mpr:cm-iltt} are equivalent to \lnl multicategories of Kleisli type having $\ten,\unit,\with,\top,\hom$, and $\mixedhom$ (our $\mixedhom$ being written ``$\to$'').
\end{exa}

We can also attempt to induce an \lnl multicategory from a \emph{monad} on a \emph{cartesian} monoidal category or multicategory.
In fact this is quite easy: the 2-category of symmetric multicategories has Eilenberg--Moore objects, so any monad $T$ therein on a multicategory \cE induces an adjunction of multicategories $\cE \toot \cE^T$.
If \cE is cartesian, by \cref{thm:lnl-adj} this yields an \lnl multicategory with $\foc,\uoc$.
The interesting thing is that if \cE is representable, hence a (cartesian) monoidal category, then a symmetric-multicategory-monad on it is the same as a lax symmetric monoidal monad, and hence by~\cite{kock:strfunc-monmnd} the same as a \emph{commutative strong monad}.

\begin{prop}
  Any commutative strong monad $T$ on a cartesian monoidal category \cE induces an \lnl multicategory \cP having $\foc,\uoc,\times,1,\unit$, where $\Pnonlin =\cE$ and the $\Plin$ is the symmetric multicategory of $T$-algebras.
  Moreover:
  \begin{enumerate}
  \item If \cE is cartesian closed with equalizers, then \cP has $\to,\hom$.\label{item:cm1}
  \item If \cE and $T$ are such that the category of $T$-algebras has coequalizers (e.g.\ \cE is locally presentable and $T$ is accessible, or \cE is cartesian closed with reflexive coequalizers preserved by $T$) then \cP also has $\ten$, and thus is an \lnl adjunction.\label{item:cm2}
  \end{enumerate}
\end{prop}
\begin{proof}
  We have already observed the first statement, except for noting that $\unit = T1$.
  Statements~\ref{item:cm1} and~\ref{item:cm2} follow by results in the literature~\cite{kock:cc-commnd,seal:tens-mnd-act}.
\end{proof}

Of course, we can also restrict to any full sub-multicategory of the Eilenberg--Moore category, such as the Kleisli category, and still have an \lnl multicategory.
As in the comonad case, when given a commutative strong monad on a cartesian monoidal category we generally regard it as an \lnl multicategory via the Kleisli construction; thus we have the following locally full sub-2-categories of \lnlpoly:
\begin{itemize}
\item Cartesian monoidal categories with a commutative strong monad.
\item Cartesian monoidal categories with a commutative strong monad and any desired limits and any desired colimits preserved by the product in each variable.
\item Cartesian closed categories with a commutative strong monad and any desired limits and colimits.
\end{itemize}

A \emph{non-commutative} monad $T$ on a cartesian monoidal category \cE does not induce a multicategory structure on its Eilenberg--Moore category $\cE^T$.
However, as long as $T$ is a \emph{strong} monad, we can still combine \cE with $\cE^T$ to produce an \lnl multicategory, albeit a rather degenerate one.
Specifically, if $A$ and $B$ are $T$-algebras and $X$ is an object of \cE, we can define an \textbf{$X$-indexed family of algebra maps $A\to B$} to be a morphism $f:X\times A \to B$ such that the following diagram commutes:
\[
  \begin{tikzcd}
    X\times T A \ar[r] \ar[d] & T(X\times A) \ar[r,"T f"] & T B \ar[d] \\
    X \times A \ar[rr,"f"'] && B
  \end{tikzcd}
\]
in which the map $X\times T A \to T(X\times A)$ is the monad strength.

\begin{prop}
  Any strong monad $T$ on a cartesian monoidal category \cE induces an \lnl multicategory \cP with $\Pnonlin = \cE$, whose linear objects are the $T$-algebras, with
  \begin{align*}
    \Pl(\Theta| ; \, A) &= \cE(\Theta; A) \\
    \Pl(\Theta| A;B) &= \big\{ (\bigtimes\Theta)\text{-indexed families of algebra maps } A\to B \big\}
  \end{align*}
  and all other linear homsets empty.\qed
\end{prop}

(Here by $\bigtimes\Theta$ we mean the cartesian product of all the objects in $\Theta$, or the terminal object if $\Theta$ is empty.)

This \lnl multicategory is \textbf{linearly subunary}, i.e.\ all its linear morphisms have linear codomain of length 1 (since it is an \lnl multicategory) and linear domain of length $\le 1$.
It has $\times,1,\uoc$, and also an $\foc$ with a weaker universal property:
\begin{equation}
  \Pl(\Theta,X|;\,B) \cong \Pl(\Theta|\foc X;B).\label{eq:cbpv-foc}
\end{equation}
This is similar to the restriction on $\top,0$ in multicategories from \cref{sec:lnl-polycategories}.
It implies there is a $\unit$ (namely $\foc 1$) with a similarly restricted universal property.
Conversely, from $\rtimes$ and a restricted $\unit$, we can construct a restricted $\foc$ as $\foc X = X \rtimes \unit$.

These \lnl multicategories provide semantics for ``call-by-push-value''~\cite{levy:adj-cbpv} and related theories.
In this case, they are usually described as \emph{enriched adjunctions}, analogously to the definition of \lnl adjunctions as \emph{monoidal} adjunctions.
To explain this, recall that if \cE is cartesian monoidal, its Yoneda embedding $\cE \into [\cE\op,\fSet]$ is fully faithful and preserves products; thus any \cE-enriched category can be regarded as an $[\cE\op,\fSet]$-enriched one.
In addition, \cE itself is always $[\cE\op,\fSet]$-enriched, with hom-presheaves $\underline{\cE}(A,B)(X) = \cE(X\times A,B)$.

\begin{prop}\label{rmk:cbpv}
  A linearly subunary \lnl multicategory with $\times,1$ is uniquely determined by a \textbf{CBPV pre-structure}~\cite{levy:adj-cbpv}: a cartesian monoidal category \cE, a category \cL enriched over $[\cE\op,\fSet]$, and an $[\cE\op,\fSet]$-enriched functor $R:\cL \to [\cE\op,\fSet]$.
  Moreover:
  \begin{enumerate}
  \item The modality \uoc exists if and only if $R$ lands inside \cE.
  \item If \uoc exists, then \foc exists with restricted universal property~\eqref{eq:cbpv-foc} if and only if $R:\cL \to \cE$ has an $[\cE\op,\fSet]$-enriched left adjoint.
  \item The hom-objects of \cL lie in \cE if and only if $\eechom$ exists.
  \item \cL has $[\cE\op,\fSet]$-enriched powers by representables if and only if $\mixedhom$ exists.
  \item \cL has $[\cE\op,\fSet]$-enriched copowers by representables if and only if $\rtimes$ exists.
  \item \cL has $[\cE\op,\fSet]$-enriched finite products if and only if $\with,\top$ exist with a restricted universal property respecting the arity restrictions.
  \item \cE is distributive~\cite{clw:ext-dist} and the hom-presheaves of \cL preserve finite coproducts if and only if $+,\varnothing$ exist with a restricted universal property.
  \end{enumerate}
\end{prop}
\begin{proof}
  Of course, \cE corresponds to $\Pnonlin$, which is cartesian monoidal if and only if $\times,1$ exist.
  The arity restrictions then ensure that the linear hom-sets are uniquely determined by those of the form $\Pl(X|A;B)$ and $\Pl(X|;\,B)$.
  The former assemble into an $[\cE\op,\fSet]$-enriched category \cL, and the latter into the functor $R$.

  To say that $R$ lands in \cE is to say that each functor $X\mapsto \Pl(X|;\,B)$ is representable, which is to say that \uoc exists.
  Given this,~\eqref{eq:cbpv-foc} says exactly that $\foc$ is an $[\cE\op,\fSet]$-enriched left adjoint of $\uoc$.
  The other claims follow by similar comparisons of universal properties.
\end{proof}

\begin{cor}\label{thm:cbpv-pow}
  A linearly subunary \lnl multicategory with $\times,1,\uoc,\eechom,\mixedhom,\rtimes$, and restricted $\foc$ (or equivalently $\unit$) is equivalent to a cartesian monoidal category \cE, a \cE-enriched category \cL with powers and copowers, and an object $\unit\in\cL$.
\end{cor}
\begin{proof}
  \cref{rmk:cbpv} implies exactly this characterization except that instead of $\unit$ we have a \cE-enriched adjunction $\foc : \cE \toot \cL: \uoc$.
  But this is uniquely determined by $\foc 1 \cong \unit$, since $\foc X \cong X \rtimes \unit$ and $\uoc A \cong \unit \eechom A$.
\end{proof}

As before, the arity restrictions can be enforced by slicing: if $\mathsc{cbpv}\in \lnlpoly$ is the subterminal with one nonlinear object, one linear object, all nonlinear homsets and co-unary subunary linear homsets singletons, and others empty, then the linearly subunary \lnl multicategories constitute the slice $\lnlpoly/\mathsc{cbpv}$.
By adding appropriate combinations of universal properties, we obtain various related structures in the literature.
Thus we have the following locally full sub-2-categories of \lnlpoly:
\begin{itemize}
\item CBPV pre-structures, as in \cref{rmk:cbpv}.
\item \textbf{CBPV adjunction models} or \textbf{EC+ models}~\cite{ems:eec}, which are CBPV pre-structures having $\uoc,\mixedhom$, and $\foc,+,\varnothing,\with,\top$ with restricted universal properties.
\item \textbf{EEC+ models}~\cite{ems:eec}, which are EC+ models having also $\to,\eechom,\rtimes$ as well as $\oplus,0$ with restricted universal properties.
  Thus they are structures as in \cref{thm:cbpv-pow} where \cE and \cL both have finite products and coproducts.
\item \textbf{MLJ$^\eta_p$ models}~\cite{cfm:adjpol}, which are CBPV pre-structures having only $\uoc,\mixedhom$, and restricted $\foc$.
\item \textbf{LJ$^\eta_p$ models}, which are MLJ$^\eta_p$ models having also restricted $+,\varnothing,\with,\top$.
\item \textbf{ECBV models}~\cite{ms:lin-state}, which are linearly \emph{unary} \lnl multicategories (that is, all linear morphisms have linear domain \emph{and} codomain of length exactly 1) having $\times,1,\eechom,\rtimes$, but no $\foc$ or $\uoc$.
  Of course, this arity restriction is given by slicing over a different object $\mathsc{ecbv}$.
\end{itemize}

We now consider the ``classical'' case: \lnl polycategories that are not co-unary.

\begin{prop}\label{thm:ldc-adj}
  An \lnl polycategory in which the modality $\foc$ exists is uniquely determined by a functor of symmetric multicategories
  \[ \foc : \Pnonlin \to \Umulti(\Plin) \]
  where $\Pnonlin$ is a cartesian multicategory, $\Plin$ a symmetric polycategory, and $\Umulti$ denotes the underlying symmetric multicategory of a symmetric polycategory.  Also:
  \begin{enumerate}
  \item The modality $\uoc$ also exists if and only if the functor $\foc$ has a right adjoint
    \[\Umulti(\Plin) \to \Pnonlin\]
    in the 2-category of symmetric multicategories.
  \item If $\times,1,\ten,\unit,\coten,\counit$ exist, then $\foc$ is equivalently a strong symmetric monoidal functor from a cartesian monoidal category to (the $\ten$ monoidal structure of) a symmetric linearly distributive one.
  \item Thus, an \lnl polycategory with $\times,1,\ten,\unit,\coten,\counit,\foc,\uoc$ is equivalently an \lnl adjunction $\cM \toot \cL$ in which \cL is linearly distributive.\label{item:la3}
    Moreover, it also has $\duals$ if and only if \cL is $\ast$-autonomous.
  \end{enumerate}
\end{prop}
\begin{proof}
  As in \cref{thm:lnl-adj}, we make the modality $\foc$ in an \lnl polycategory into a functor using its universal property; while given a functor as above we define the general linear homsets by
  \[\Pl(X_1,\dots,X_n | \Gamma;\Delta) = \Plin(\foc X_1,\dots,\foc X_n,\Gamma\mathbin{;}\Delta)\]
  so that the universal property of $\foc$ holds by definition.
  The rest is also similar to \cref{thm:lnl-adj}, using the result of~\cite{cs:wkdistrib} that a symmetric polycategory with $\ten,\unit,\coten,\counit$ is equivalently a symmetric linearly distributive category.
  The universal property of $\foc$ relative to linear morphisms with arbitrary codomain ensures that it is uniquely determined by its action on underlying multicategories, while $\uoc$ knows nothing about the non-co-unary morphisms at all.
\end{proof}

Note that since $\fwn$ and $\uwn$ can be defined in terms of $\foc,\uoc,\duals$ by $\fwn X = \d{(\foc X)}$ and $\uwn A = \uoc(\d A)$, an \lnl adjunction with \cL $\ast$-autonomous also has $\fwn,\uwn$.
Thus, we have the following locally full sub-2-categories of \lnlpoly:
\begin{itemize}
\item \textbf{Linearly distributive \lnl adjunctions} and \textbf{$\ast$-autonomous \lnl adjunctions}, defined as in \cref{thm:ldc-adj}\ref{item:la3}.
\item Linearly distributive \lnl adjunctions with any desired limits and colimits in either category, subject to the restrictions that colimits must be preserved by the product or tensor product in each variable, and limits in the linearly distributive category must be preserved by the cotensor product in each variable.
\item $\ast$-autonomous closed \lnl adjunctions with any desired limits and colimits in either category.
\end{itemize}

On the other hand, if we add $\fwn$ and $\uwn$ \emph{without} $\duals$, the induced structure on \cL is also one that appears in the literature:

\begin{prop}\label{thm:ldc-stor}
  If \cP is an \lnl polycategory with $\ten,\unit,\coten,\counit,\foc,\uoc,\fwn,\uwn$, then $\Plin$ is a \textbf{\emph{(symmetric) linearly distributive category with storage}}~\cite{bcs:storage}.
\end{prop}
\begin{proof}
  Note that any \lnl polycategory \cP has an underlying \lnl multicategory \linebreak[4] $\coun\cP$ containing all the objects, all the nonlinear morphisms, but only the co-unary linear morphisms.
  It also has a \textbf{linear opposite} $\cP\lop$ in which the nonlinear morphisms are the same, but $\linhom{\cP\lop}(\Theta|\Gamma;\Delta) = \linhom{\cP}(\Theta|\Delta;\Gamma)$.

  Thus, applying \cref{thm:linexpcmnd} to $\coun\cP$ and $\coun{\cP\lop}$, we obtain a linear exponential comonad $\oc=\foc\uoc$ and a linear exponential monad $\wn = \fwn\uwn$, so it remains only to show that $\wn$ is a $\oc$-strong monad and dually.
  We obtain the morphism $\wn A \ten \oc B \to \wn (A\ten \oc B)$ by acting on the $\uwn$-universal morphism of $(\uwn(A\ten \foc\uoc B)\mid \,) \to A\ten \foc\uoc B$ as follows.
  \begingroup\allowdisplaybreaks
  \begin{alignat*}{2}
    \Pl(\uwn{(A\ten \foc\uoc B)} | A\ten \foc\uoc B ; )
    &\toiso &\;& \Pl(\uwn{(A\ten \foc\uoc B)} | A, \foc\uoc B ; )\\
    &\toiso &\;& \Pl(\uwn{(A\ten \foc\uoc B)}, \uoc B | A ; )\\
    &\toiso &\;& \Pnl(\uwn{(A\ten \foc\uoc B)}, \uoc B ; \uwn A)\\
    &\to &\;& \Pl(\uwn{(A\ten \foc\uoc B)}, \uoc B|\fwn\uwn A ; ) \\
    &\toiso &\;& \Pl(|\fwn\uwn A, \foc\uoc B; \fwn\uwn{(A\ten \foc\uoc B)})\\
    &\toiso &\;& \Pl(|\fwn\uwn A \ten \foc\uoc B; \fwn\uwn{(A\ten \foc\uoc B)})\\
    &= &\;& \Pl(|\wn A \ten \oc B ; \wn {(A\ten \oc B)}).
  \end{alignat*}
  \endgroup
  The noninvertible map above is composition with the $\fwn$-universal $(\uwn A \mid \fwn \uwn A) \to ()$.
  It is straightforward to check the axioms.
  (This is like the proof in~\cite[\S3.1]{bcs:storage} that proof nets with storage boxes form a linearly distributive category with storage.)
\end{proof}

The converse of \cref{thm:ldc-stor} is subtler.
If \cL is a symmetric linearly distributive category with storage, it is in particular a symmetric monoidal category (under $\ten,\unit$) with a linear exponential comonad $\oc$.
Therefore, it gives rise to an \lnl adjunction $\cM\toot\cL$ as above, where \cM is the Eilenberg--Moore category of the comonad $\oc$.
Hence, by \cref{thm:ldc-adj}, any subcategory of this \cM (such as the Kleisli category) yields an \lnl polycategory \cP with $\Plin = \cL$ and having $\ten,\unit,\coten,\counit,\foc,\uoc$.
Similarly, any subcategory of the opposite of the Eilenberg--Moore category of the monad $\wn$ yields an \lnl polycategory \cP with $\Plin=\cL$ and having $\ten,\unit,\coten,\counit,\fwn,\uwn$.

If \cL has duals, hence is $\ast$-autonomous, then by~\cite[Proposition 5.1]{bcs:storage} the modalities $\oc$ and $\wn$ are dual, in that $\wn A \cong \d{(\oc(\d A))}$.
This implies that their Eilenberg--Moore and Kleisli categories are dual to each other, by equivalences that lie over the self-duality $\duals$; hence these two \lnl polycategories coincide and
are a $\ast$-autonomous \lnl adjunction that induces the given $\oc$ and $\wn$.
However, if \cL does not have duals, then the Eilenberg-Moore categories of $\oc$ and $\wn$ need not be dual:

\begin{exa}\label{eg:dlat}
  Let \cL be a distributive lattice that is not a Boolean algebra.
  As in~\cite{cs:wkdistrib}, we can regard \cL as a linearly distributive category with $\mathord{\ten}=\mathord{\wedge}$ and $\mathord{\coten}=\mathord{\vee}$.
  Since $\wedge$ is the cartesian product and $\vee$ the cartesian coproduct, we can equip \cL with storage modalities $\oc$ and $\wn$ that are both just the identity.
  (Thanks to Robin Cockett for pointing out this example.)
  The Eilenberg--Moore categories of this $\oc$ and $\wn$ are then both just $\cL$ itself, which may not be self-dual.

  In fact this \cL cannot occur as $\Plin$ for \emph{any} \lnl polycategory \cP with $\foc,\uoc,\fwn,\uwn$ such that its (identity) modalities $\oc$ and $\wn$ are recovered as $\foc\uoc$ and $\fwn\uwn$ respectively.
  To see this, note that for any nonlinear object $X$ in an \lnl polycategory, if $\foc X$ and $\fwn X$ both exist, then they are dual to each other.
  Thus, if $\foc,\fwn$ both exist, then any object of the form $\foc X$ or $\fwn X$ has a dual --- and hence if $\oc = \foc\uoc$ is the identity, then \emph{every} object has a dual.
  But this would imply that \cL is a Boolean algebra.
\end{exa}

Thus, if we want to embed a general linearly distributive category with storage into an \lnl polycategory, we have to give up on having all $\foc,\uoc,\fwn,\uwn$.
But we can get away with something slightly less:

\begin{prop}\label{thm:ldc-kl}
  A linearly distributive category \cL admits storage modalities if and only if it can occur as $\Plin$ for an \lnl polycategory \cP having $\ten,\unit,\coten,\counit,\uoc,\uwn$ along with $\foc$ defined on the image of $\uoc$ and $\fwn$ defined on the image of $\uwn$.
\end{prop}
\begin{proof}
  For ``if'', just note that the proof of \cref{thm:ldc-stor} uses only this weaker hypothesis.
  For ``only if'', let \cL be a symmetric linearly distributive category with storage, and define an \lnl polycategory $\bkll$ as follows.
  Its linear objects are the objects of \cL, while its nonlinear objects consist of two copies of the objects of \cL denoted $A^\oc$ and $A^\wn$.
  Its homsets are defined by:
  \begin{multline*}
    \linhom{\bkll}(A_1^\oc,\dots,A_p^\oc,B_1^\wn,\dots,B_q^\wn | C_1,\dots,C_m ; D_1,\dots, D_n)\\
    = \cL(\oc A_1\ten \cdots \ten \oc A_p \ten C_1\ten\cdots \ten C_m\,,\, \wn B_1 \coten \cdots \coten \wn B_q \coten D_1 \coten\cdots\coten D_n)
  \end{multline*}
  \begin{align*}
    \nonlinhom{\bkll}(A_1^\oc,\dots,A_p^\oc,B_1^\wn,\dots,B_q^\wn  ; C^\oc )
    &= \cL(\oc A_1\ten \cdots \ten \oc A_p\,,\, \wn B_1 \coten \cdots \coten \wn B_q \coten C)\\
    \nonlinhom{\bkll}(A_1^\oc,\dots,A_p^\oc,B_1^\wn,\dots,B_q^\wn  ; C^\wn )
    &= \cL(\oc A_1\ten \cdots \ten \oc A_p\ten C\,,\, \wn B_1 \coten \cdots \coten \wn B_q)
  \end{align*}
  In particular, we have
  \begin{alignat*}{2}
    \nonlinhom{\bkll}(A^\oc ; C^\oc) &= \cL(\oc A, C) &\qquad
    \nonlinhom{\bkll}(A^\oc ; C^\wn) &= \cL(\oc A\ten C, \counit)\\
    \nonlinhom{\bkll}(B^\wn ; C^\wn) &= \cL(C, \wn B) &\qquad
    \nonlinhom{\bkll}(B^\wn ; C^\oc) &= \cL(\unit, \wn B \coten C).
  \end{alignat*}
  That is, the category of nonlinear objects and unary morphisms consists of a copy of the Kleisli category of $\oc$ (the objects $A^\oc$) and a copy of the opposite of the Kleisli category of $\wn$ (the objects $B^\wn$), with the morphisms between the two defined in a twisted way using the linearly distributive structure.

  Composition of two linear morphisms is defined just as in the ordinary symmetric polycategory underlying \cL.
  To compose a nonlinear morphism with either a linear or nonlinear morphism, we make use of the ``generalized Kleisli lift'': given
  \[ f : \oc A_1\ten \cdots \ten \oc A_p \too \wn B_1 \coten \cdots \coten \wn B_q \coten C \]
  we can construct the composite
  \begin{align*}
    \oc A_1\ten \cdots \ten \oc A_p
    &\to \oc \oc A_1\ten \cdots \ten \oc \oc A_p\\
    &\to \oc (\oc A_1\ten \cdots \ten \oc A_p)\\
    &\xto{\oc f} \oc(\wn B_1 \coten \cdots \coten \wn B_q \coten C)\\
    &\to \wn B_1 \coten \cdots \coten \wn B_q \coten \oc C
  \end{align*}
  where the first map is composed of the comultiplications $\oc A_i \to \oc\oc A_i$ of $\oc$, the second map is the lax monoidal structure of $\oc$, the third in $\oc f$, and the fourth is $q$ applications of the strength $\oc(\wn B \coten C ) \to \wn B \coten \oc C$.
  By first applying this construction to a nonlinear morphism with codomain $C^\oc$, or the dual construction to one with codomain $C^\wn$, we can then compose it along this object with any other morphism as usual in the underlying polycategory of \cL.

  Of course this \lnl polycategory has $\ten,\unit,\coten,\counit$.
  By construction it has $\uoc A = A^\oc$ and $\uwn A = A^\wn$, and partially defined $\foc A^\oc = \oc A$ and $\fwn A^\wn = \wn A$.
  Note that this is very similar to the proof in~\cite[\S3.2]{bcs:storage} that proof nets with storage are sound for linearly distributive categories with storage.
\end{proof}

This ``double Kleisli category'' construction is functorial, and lands inside the slice category $\lnlpoly/\mathsc{dblsplit}$ from \cref{rmk:double-split}.
In terms of this slice, we can describe the restricted domains of $\foc$ and $\fwn$ by saying that $\foc$ is defined on left-hand objects and $\fwn$ on right-hand ones.

Moreover, if \cL is $\ast$-autonomous, then $A^\wn \cong (\d{A})^\oc$ in $\nonlin{(\bkll)}$.
Thus in this case $\bkll$ is equivalent (though not isomorphic) to the Kleisli adjunction of $\oc$ and also to the Kleisli adjunction of $\wn$.

This gives us the following locally full sub-2-categories of \lnlpoly:
\begin{itemize}
\item Linearly distributive categories with storage.
\item $\ast$-autonomous categories with storage.
\item Linearly distributive or $\ast$-autonomous categories with storage, any desired colimits preserved by the tensor product in each variable, and any desired limits preserved by the cotensor product in each variable.
\end{itemize}

\section{Unifying universality}
\label{sec:bifib}

In defining \lnl doctrines, we will want to work generally with classes of universal arrows and colimits in \lnl polycategories.
Unfortunately, the different kinds of objects and morphisms in an \lnl polycategory make such a general treatment quite cumbersome.
For instance, we already saw in \cref{sec:lnl-polycategories} that there are formally five different kinds of ``universal morphism'' in an \lnl polycategory, which has the consequence that a fully formal proof of \cref{thm:univ-comp} (universal morphisms compose) would have on the order of 25 different cases to consider.\footnote{Not exactly 25, of course, since some pairs of universal morphisms will not be composable.}
Similarly, there are four different kinds of limits and colimits, and so on.
Duality doesn't simplify the situation significantly either, since an \lnl polycategory has no ``opposite'' that reverses the nonlinear morphisms.
Nevertheless, there is a clear intuition that this technical multiplicity is in some sense ``inessential'': all the cases behave similarly.
In this section we give an alternative definition of \lnl polycategories that enables us to formally unify these cases.

Given a set of objects partitioned into linear and nonlinear ones, by a \textbf{signed object} we mean an object together with an element of $\{-,+\}$, written $R^+$ or $R^-$, where $R$ is a (linear or nonlinear) object.
We denote general signed objects by letters towards the middle of the Roman alphabet such as $K,L,M,\dots$, and lists of signed objects by the Greek letters $\Phi,\Psi$.
If $K$ is a signed object we write $K\D$ for the result of flipping its sign: $(R^+)\D = R^-$ and $(R^-)\D = R^+$.

\begin{defi}
  A list of signed objects is \textbf{admissible} if
  \begin{enumerate}
  \item it contains at most one positive nonlinear object, and
  \item if it does contain one such, then it contains no linear objects.
  \end{enumerate}
\end{defi}

\begin{lem}
  If $(\Phi,K)$ and $(K\D,\Psi)$ are admissible, so is $(\Phi,\Psi)$.
\end{lem}
\begin{proof}
  If a positive nonlinear object $X^+$ appears in $\Phi$, then $K$ and all other objects in $\Phi$ must be negative nonlinear.
  Hence $K\D$ is positive nonlinear, so all objects in $\Psi$ are also negative nonlinear.
  We can argue similarly if $\Psi$ contains $X^+$.
\end{proof}

By a \textbf{structural map} we mean a morphism $\si : (K_1,\dots, K_m)\to (K_{\si 1},\dots, K_{\si n})$ where $(K_1,\dots, K_m)$ is a list of signed objects and $\si : \{1,\dots,n\} \to \{1,\dots,m\}$ is a function with the property that for any $j$ with $1\le j\le m$, if $|\si^{-1}(j)|\neq 1$ then $K_j$ is negative and nonlinear.

\begin{defi}
  An \textbf{entries-only \lnl polycategory} \cP consists of:
  \begin{itemize}
  \item A set of \textbf{objects} partitioned into linear and nonlinear ones.
  \item For any admissible list of signed objects $(K_1,\dots,K_n)$, a hom-set $\cP(K_1,\dots,K_n)$, with functorial actions $\cP(\Psi)\to\cP(\Phi)$ by structural maps $\si : \Phi\to\Psi$.
  \item For any object $R$ (linear or nonlinear), an identity $1_R \in \cP(R^-,R^+)$.
  \item Whenever $(\Phi,K)$ and $(K\D,\Psi)$ are admissible, a composition map
    \[ \circ_K : \cP(K\D,\Psi)\times \cP(\Phi,K) \to \cP(\Phi,\Psi) \]
    that is associative, unital, and equivariant with respect to the structural actions and permutations that swap the two inputs.
  \end{itemize}
  A \textbf{functor} between entries-only \lnl polycategories consists of functions between their linear and nonlinear objects and morphisms, preserving entries, structural actions, identities, and composites.
\end{defi}

\begin{prop}
  The category of entries-only \lnl polycategory is equivalent to that of \lnl polycategories.
\end{prop}
\begin{proof}
  By structural permutations, the hom-sets of an entries-only \lnl polycategory are uniquely determined (up to isomorphism) by those of the form
  \begin{gather*}
    \cP(X_1^-,\dots,X_m^-,Y^+) \\
    \cP(X_1^-,\dots,X_m^-,A_1^-,\dots,A_n^-,B_1^+,\dots,B_p^+)
  \end{gather*}
  for nonlinear objects $X_i,Y$ and linear objects $A_j,B_k$.
  We can identify these with the hom-sets
  \begin{gather*}
    \Pnl(X_1,\dots,X_m;Y) \\
    \Pl(X_1,\dots,X_m|A_1,\dots,A_n;B_1,\dots,B_p)
  \end{gather*}
  in an ordinary \lnl polycategory, and the identities, compositions, and structural actions correspond.
\end{proof}

Of course, the 2-categorical structure of \lnlpoly that we defined in \cref{sec:lnl-polycategories} can also be transported across this equivalence.
A transformation between functors of entries-only \lnl polycategories thus has components $\al_X \in \cQ((H X)^-, (K X)^+)$ and $\al_A \in \cQ((H A)^- , (K A)^+)$ satisfying suitable axioms.

Henceforth, we will pass freely back and forth between the two definitions, using whichever notation for homsets is more convenient.
We can now define a general notion of universal morphism that encompasses all five cases described in \cref{sec:lnl-polycategories}.

\begin{defi}
  A morphism $f\in \cP(\Phi,K)$ in an entries-only \lnl polycategory is \textbf{universal in $K$} if for any list of signed objects $\Psi$ such that $(K\D,\Psi)$ is admissible, the composition map $(-\circ_K f) : \cP(K\D,\Psi) \to \cP(\Phi,\Psi)$ is bijective, i.e.\ for any $h\in \cP(\Phi,\Psi)$ there exists a unique $g\in \cP(K\D,\Psi)$ such that $g\circ_K f = h$.
\end{defi}

In fact, following~\cite{hermida:fib-multi,lsr:multi,bz:bifib-poly}, it is useful to generalize from \emph{universal} morphisms in one multi- or poly-category to \emph{cartesian} ones relative to a functor.

\begin{defi}\label{defn:cartesian}
  Given a functor $\pi:\cP\to\cQ$ of entries-only \lnl polycategories, a morphism $f\in \cP(\Phi,K)$ is \textbf{$\pi$-cartesian in $K$} if for any list of signed objects $\Psi$ of \cP such that $(K\D,\Psi)$ is admissible, the following square is a pullback:
  \begin{equation}
    \begin{tikzcd}[column sep=huge]
      \cP(K\D,\Psi) \ar[r,"-\circ_K f"] \ar[d,"\pi"'] & \cP(\Phi,\Psi) \ar[d,"\pi"]\\
      \cQ(\pi K\D, \pi\Psi) \ar[r,"-\circ_{(\pi K)} (\pi f)"'] & \cQ(\pi\Phi,\pi\Psi)
    \end{tikzcd}\label{eq:cartesian}
  \end{equation}
  In other words, for any $h\in \cP(\Phi,\Psi)$ and $\ell\in \cQ(\pi K\D,\pi \Psi)$ such that $\ell \circ_{\pi K} \pi f = \pi h$, there exists a unique $g\in \cP(K\D,\Psi)$ such that $g\circ_K f = h$ and $\pi g = \ell$.
\end{defi}

Note that if $\cQ$ is terminal, both sets on the bottom row of~\eqref{eq:cartesian} are singletons; so the square is a pullback just when the morphism on top is a bijection.
Thus, $f$ is universal in $K$ precisely when it is $\pi$-cartesian in $K$ for the unique functor $\pi : \cP\to \mathsc{lnlpoly}$ to the terminal object.

Cartesian morphisms specialize to various notions in the literature:
\begin{itemize}
\item For symmetric multicategories, cartesian morphisms with $K$ positive specialize to the ``strongly cocartesian'' morphisms of~\cite[Remarks 2.2(1)]{hermida:fib-multi}.
\item For cartesian multicategories, cartesian morphisms specialize to the cartesian and opcartesian morphisms of~\cite{lsr:multi}.
\item For symmetric polycategories, cartesian morphisms specialize to the cartesian and opcartesian morphisms of~\cite{bz:bifib-poly}.
\item For categories, cartesian morphisms specialize to the traditional notion of cartesian and opcartesian morphism.
\end{itemize}

\begin{exa}\label{eg:cart-cbpv-foc}
  Cartesian morphisms can express restricted universal properties.
  For instance, in \cref{defn:cartesian} let $\cQ = \mathsc{cbpv}$, and let $f\in \cP(X^-,A^+)$ for a nonlinear $X$ and linear $A$, with vertex $K = A^+$.
  Then the hom-set $\cQ(\pi K\D,\pi\Psi)$ is empty unless $\Psi$ contains exactly one positive linear object and the rest nonlinear.
  Thus, $f$ is cartesian just when it exhibits $A$ as $\foc X$ with the universal property of~\eqref{eq:cbpv-foc}.  
\end{exa}

\begin{exa}\label{eg:smadj}
  Cartesian morphisms can also express adjunctions that behave similarly to $\foc\adj\uoc$ but stay inside the linear or nonlinear world.
  For instance, let $\mathsc{smadj}$ be the \lnl multicategory with two objects $\mathsc{p},\mathsc{n}$, both linear, a unique morphism $\Gamma \to \mathsc{p}$ when $\Gamma$ consists entirely of $\mathsc{p}$'s, and a unique morphism $\Gamma \to \mathsc{n}$ for any $\Gamma$.
  Then an object \cP of $\lnlpoly/\mathsc{smadj}$ is a symmetric multicategory with a partition of its objects into ``positive'' and ``negative'' ones, such that any morphism with a negative object in its domain has a negative codomain.
  Suppose in addition that
  \begin{itemize}
  \item For any positive object $A$, there is a negative object $B$ and a morphism $A\to B$ that is cartesian in $B$ over the unique morphism $\mathsc{p}\to \mathsc{n}$ in $\mathsc{smadj}$.
  \item For any negative object $B$, there is a positive object $A$ and a morphism $A\to B$ that is cartesian in $A$ over the unique morphism $\mathsc{p}\to \mathsc{n}$ in $\mathsc{smadj}$.
  \end{itemize}
  By an argument like that of \cref{thm:lnl-adj}, such a \cP is uniquely determined by an adjunction of symmetric multicategories.
  Further cartesian liftings can specialize this to an adjunction of symmetric monoidal categories, with strong left adjoint and lax right adjoint.
\end{exa}

\begin{exa}\label{eg:adj}
  As an even simpler example, let $\mathsc{adj}$ have two linear objects $\mathsf{p}, \mathsc{n}$ and only one nonidentity morphism $\mathsc{p}\to \mathsc{n}$.
  Then an object of $\lnlpoly/\mathsc{adj}$ is an ordinary category with its objects partitioned into positive and negative ones, such that there are no morphisms from a negative object to a positive one.
  Such a category is precisely the ``collage'' of a profunctor between the categories \cP and \cN of positive and negative objects.
  If all cartesian liftings of the morphism $\mathsc{p}\to \mathsc{n}$ exist in one direction, then the profunctor is representable by a functor $\cP\to \cN$; if they exist in the other direction, it is representable by a functor $\cN\to\cP$; and if both exist, it is representable by an adjunction $\cP \toot \cN$.
\end{exa}

As an example of the value of the entries-only framework, we can now prove (a generalization of) \cref{thm:univ-comp} without a division into 25-odd cases:

\begin{prop}\label{thm:univ-comp2}
  Given $\pi:\cP\to\cQ$, if $f\in \cP(\Phi_1,K)$ is $\pi$-cartesian in $K$ and $g\in\cP(K\D,\Phi_2,L)$ is $\pi$-cartesian in $L$, then their composite $g\circ_K f\in \cP(\Phi_1,\Phi_2,L)$ is $\pi$-cartesian in $L$.
\end{prop}
\begin{proof}
  In the following diagram:
  \[
    \begin{tikzcd}[column sep=huge]
      \cP(L\D,\Psi) \ar[r,"-\circ_L g"] \ar[d,"\pi"'] &
      \cP(K\D,\Phi_2,\Psi) \ar[r,"-\circ_K f"] \ar[d,"\pi"'] &
      \cP(\Phi_1,\Phi_2,\Psi) \ar[d,"\pi"]\\
      \cQ(\pi L\D, \pi\Psi) \ar[r,"-\circ_{(\pi L)} (\pi g)"'] &
      \cQ(\pi K\D,\pi\Phi_2,\pi\Psi) \ar[r,"-\circ_{(\pi K)} (\pi f)"'] &
      \cQ(\pi \Phi_1,\pi\Phi_2,\pi\Psi)
    \end{tikzcd}
  \]
  both squares are pullbacks since $f$ and $g$ are $\pi$-cartesian, hence so is the rectangle.
\end{proof}

Following~\cite{lsr:multi,bz:bifib-poly}, we define:

\begin{defi}\label{def:bifibration}
  A functor $\pi:\cP\to\cQ$ is a \textbf{bifibration} if for any list $\Phi$ of signed objects in \cP and any morphism $g\in \cQ(\pi\Phi,L)$ there exists a $\pi$-cartesian morphism $f\in \cP(\Phi,K)$ such that $\pi(f) = g$.
\end{defi}

When $\cQ$ is one of our distinguished subterminal objects (including the terminal object $\mathsc{lnlpoly}$), bifibrations $\pi:\cP\to\cQ$ reduce to more familiar structures:

\begin{table}
  \centering
  \begin{tabular}{c|c|c}
    Subterminal $\cS$ & Universal properties & Equivalent structure \\\hline
    $\mathsc{lnlpoly}$ & $\times,1,\to,\ten,\unit,\duals,\foc,\uoc$ & $\ast$-autonomous closed \lnl adjunction\\\hline
    $\mathsc{lnlmulti}$ & $\times,1,\to,\ten,\unit,\hom,\foc,\uoc$ & closed \lnl adjunction\\\hline
    $\mathsc{sympoly}$ & $\ten,\unit,\duals$ & $\ast$-autonomous category\\\hline
    $\mathsc{symmulti}$ & $\ten,\unit,\hom$ & closed symmetric monoidal category\\\hline
    $\mathsc{cartmulti}$ & $\times,1,\to$ & cartesian closed category\\\hline
    $\mathsc{cbpv}$ & $\times,1,\to,\mixedhom,\eechom,\rtimes,\unit^\dagger,\foc^\dagger,\uoc$ & structure of \cref{thm:cbpv-pow}
  \end{tabular}~\\[1ex]
  \raggedright
  \small $^\dagger$ with restricted universal property.
  \\[1ex]
  \caption{Bifibrations over subterminals}
  \label{tab:bifib}
\end{table}

\begin{thm}\label{thm:rep}\label{thm:slice-bifib}\label{eg:cbpv-slice}
  For each row in \cref{tab:bifib}, with subterminal object $\cS$ listed in the first column, the following structures are equivalent:
  \begin{enumerate}
  \item A bifibration $\pi:\cP\to \cS$.\label{item:rep1}
  \item An object of $\lnlpoly/\cS$ with the universal properties in the second column.\label{item:rep3}
  \item The categorical structure indicated in the third column.\label{item:rep4}
  \end{enumerate}
\end{thm}
\begin{proof}
  Clearly \ref{item:rep1}$\Rightarrow$\ref{item:rep3}, while~\ref{item:rep3}$\Leftrightarrow$\ref{item:rep4} follows from \cref{sec:relation-literature}.
  The remaining direction \ref{item:rep3}$\Rightarrow$\ref{item:rep1} is similar to the universal characterization of $\ast$-autonomous categories in~\cite{bz:bifib-poly}.
  By $\bigtimes\Theta$, $\bigotimes\Gamma$, or $\bigparr\Delta$ we mean the result of combining all the objects in a list with the given binary operation; if the list contains only one object the result is that object (in which case the binary operation doesn't even need to exist), while if the list is empty the result is the corresponding nullary operation $1$, $\unit$, or $\counit$.
  Now we construct the five possible types of morphism universal in $X$ or $A$ as follows:
  \begin{itemize}
  \item For $\psi\in\Pnl(\Theta;X)$ we take $X = \bigtimes \Theta$.
  \item For $\psi\in\Pnl(\Theta,X;Y)$ we take $X = \bigtimes\Theta\to Y$.
  \item For $\psi \in \Pl(\Theta,X|\Gamma;\Delta)$ we take $X = \bigtimes\Theta \to (\bigotimes\Gamma\eechom \bigparr\Delta)$.
  \item For $\psi \in \Pl(\Theta|\Gamma;\Delta,A)$ we take $A = \bigtimes\Theta \rtimes \bigotimes (\Gamma,\d\Delta)$.
  \item For $\psi \in \Pl(\Theta|\Gamma,A;\Delta)$ we take $A = \bigtimes\Theta \mixedhom \bigparr(\d\Gamma,\Delta)$.
  \end{itemize}
  We leave it to the reader to check that whenever a particular type of universal morphism exists in one of our subterminals $\cS$, the requisite universal operations are among those assumed by~\ref{item:rep3} or can be constructed from them.
  (When $\cS = \mathsc{cbpv}$, we discussed the restricted universal property of $\foc$ in \cref{eg:cart-cbpv-foc}.)
\end{proof}

\begin{defi}\label{defn:birep}
  If \cQ is a fixed object such as those in \cref{tab:bifib} (or more generally \cref{tab:subterms}), we refer to an object $\cP \in \lnlpoly/\cQ$ as \textbf{birepresentable} if the map $\pi :\cP\to\cQ$ is a bifibration.
\end{defi}

For instance, a birepresentable \lnl polycategory is a $\ast$-autonomous closed \lnl adjunction, a birepresentable symmetric polycategory is a $\ast$-autonomous category, a birepresentable cartesian multicategory is a cartesian closed category, and so on.\footnote{In the literature, sometimes ``representable'' means only that ``covariant'' universal arrows exist, e.g.\ a ``representable symmetric multicategory'' is a not-necessarily-closed symmetric monoidal category.  But other times it means that all universal arrows exist, e.g.\ a ``representable polycategory'' is a $\ast$-autonomous category.  Our ``birepresentable'', in analogy to ``bifibration'', avoids ambiguity.}

Similarly, we can define a general notion of limit that encompasses all four cases.
In fact, we can define a general notion that encompasses both universal morphisms \emph{and} (weighted) limits and colimits!

\begin{defi}\label{defn:cones}
  An \textbf{abstract cone} is a small entries-only \lnl polycategory $\cC$ equipped with a specified signed object $K$ called the \textbf{vertex}, such that $\cC(\Phi)$ is empty if $\Phi$ contains any copies of $K\D$ or contains more than one copy of $K$, except that $\cC(K\D,K)=\{1_K\}$.
  Nonidentity morphisms containing $K$ (necessarily exactly once) are called \textbf{abstract projections}, while morphisms not containing $K$ are called \textbf{abstract transitions}.
  Note that no two abstract projections can be composable.
  The \textbf{reduct} of an abstract cone is its sub-\lnl-polycategory obtained by removing the underlying object of $K$, its identity morphism, and all the abstract projections; we denote this by $\reduct{\cC}$.

  An \textbf{expansion} of an abstract cone \cC is determined by a finite number of new objects (each linear or nonlinear) and a sign for each of them, yielding a signed list $\Psi$, such that $(K\D,\Psi)$ is admissible (where $K$ is the vertex of \cC).
  The expansion itself is an entries-only \lnl polycategory denoted $\expand{\cC}{\Psi}$ (which is not itself an abstract cone) obtained by adding the new objects to \cC along with one new morphism $\ftil\in \expand{\cC}{\Psi}(\Phi,\Psi)$ for each abstract projection $f\in \cC(\Phi,K)$, called the \textbf{expanders}, and an additional new morphism $\chi \in \expand{\cC}{\Psi}(K\D,\Psi)$ called the \textbf{factorization}.
  Composition is defined by $\chi \circ_K f = \ftil$, and by $\ftil \circ g = \widetilde{f\circ g}$ when $g$ is an abstract transition.
  The corresponding \textbf{pre-expansion} is the sub-\lnl-polycategory $\preexpand{\cC}{\Psi}\subseteq \expand{\cC}{\Psi}$ obtained by omitting the morphism $\chi$.
  Note that we have inclusions
  \[ \reduct{\cC} \;\subseteq\; \cC \;\subseteq\; \preexpand{\cC}{\Psi} \;\subseteq\; \expand{\cC}{\Psi}. \]
\end{defi}

\begin{defi}
  By a \textbf{concrete cone} we mean a functor whose domain is an abstract cone.
  Let $\pi :\cP\to \cQ$ a functor of (entries-only) \lnl polycategories, and $G:\cC\to \cP$ a concrete cone.
  We say that $G$ is \textbf{$\pi$-extremal} if for any expansion $\expand{\cC}{\Psi}$ of \cC, any commutative square as shown below such that the composite $\cC\to\preexpand{\cC}{\Psi}\to\cP$ is $G$ has a unique diagonal filler.
  \[
    \begin{tikzcd}
      \cC \ar[r] \ar[rr,bend left,"G"] & \preexpand{\cC}{\Psi} \ar[r] \ar[d] & \cP \ar[d,"\pi"] \\
      & \expand{\cC}{\Psi} \ar[r] \ar[ur,dashed,"\exists!"] & \cQ
    \end{tikzcd}
  \]
  If $\cQ=\mathsc{lnlpoly}$ is terminal, instead of $\pi$-extremal we say that $G$ is \textbf{universal}.
\end{defi}

We will be primarily interested in two important classes of abstract cones, which show respectively that the notion of extremal cone includes both cartesian/universal morphisms and limits and colimits.
Here is the first.

\begin{defi}\label{eg:extremal-univ}
  Let $\Phi$ be a finite list of abstract objects and let $K$ be an additional abstract object, such that $K$ and each object of $\Phi$ is either linear or nonlinear and has a chosen sign.
  Let $\cartcone{\Phi}{K}$ be the \lnl polycategory whose objects are those of $\Phi$ and $K$ and having precisely one nonidentity morphism $f\in \cartcone{\Phi}{K}(\Phi,K)$.
  This is an abstract cone with vertex $K$; we call it the \textbf{abstract cartesianness cone} determined by $\Phi$ and $K$.
\end{defi}

Observe that a concrete cone $G:\cartcone{\Phi}{K}\to \cP$ is determined by a single morphism $Gf\in \cP(G\Phi,GK)$.

\begin{prop}\label{thm:extremal-cartesian}
  For any $\phi:\cP\to\cQ$, a concrete cone $G:\cartcone{\Phi}{K}\to \cP$ is $\pi$-extremal if and only if $Gf$ is $\pi$-cartesian in $K$.
\end{prop}
\begin{proof}
  Because there is exactly one abstract projection $f$ in $\cartcone{\Phi}{K}$, an extension of a functor $G:\cC\to \cP$ to some pre-expansion $\preexpandp{\cartcone{\Phi}{K}}{\Psi}$ is uniquely determined by a list of signed objects $\Psi$ in \cP such that $(GK\D,\Psi)$ is admissible, together with a morphism $\ftil \in \cP(G\Phi,\Psi)$.
  A further extension of this to the expansion $\expandp{\cartcone{\Phi}{K}}{\Psi}$ consists of a morphism $\chi \in \cP(GK\D,\Psi)$ such that $\chi \circ Gf = \ftil$.
  Applying these characterizations to \cQ as well, we see that $G$ is $\pi$-extremal if and only if
  \begin{quote}
    For any list of signed objects $\Psi$ in \cP such that $(GK\D,\Psi)$ is admissible, any morphism $\ftil \in \cP(G\Phi,\Psi)$, and any morphism $\xi \in \cQ(\pi G K\D, \pi\Psi)$ such that $\xi \circ \pi G f = \pi \ftil$, there exists a unique morphism $\chi \in \cP(GK\D,\Psi)$ such that $\chi \circ Gf = \ftil$ and $\pi(\chi) = \xi$.
  \end{quote}
  However, this is also exactly what it means for~\eqref{eq:cartesian} (with $f$ replaced by $Gf$) to be a pullback of sets, which is the definition of when $Gf$ is $\pi$-cartesian in $K$.
\end{proof}

Our second important class of abstract cones is the following.

\begin{defi}\label{eg:extremal-limit}
  Let \cA be an ordinary small category, and let $\newterm{\cA}$ denote the result of adjoining a new terminal object $T$.
  If we make $\newterm{\cA}$ an \lnl polycategory by declaring all objects to be linear, it becomes an abstract cone with vertex $T^+$.
  We denote this by $\colimconel{\cA}$ and call it the \textbf{abstract linear colimit cone} determined by $\cA$.

  Dually, if $\newinit{\cA}$ denotes the result of adjoining a new initial object $I$, then with all objects linear it yields an abstract cone with vertex $I^-$.
  We denote this by $\limconel{\cA}$ and call it an \textbf{abstract linear limit cone}.

  Similarly, by declaring all the objects to be nonlinear, we obtain \textbf{abstract nonlinear colimit cones} $\colimconenl{\cA}$ and \textbf{abstract nonlinear limit cones} $\limconenl{\cA}$.
\end{defi}

Observe that a concrete cone $G:\colimconel{\cA} \to \cP$ is determined by a cocone under a $\cA$-shaped diagram in the category of linear objects of \cP, and similarly in the other cases.

\begin{prop}\label{thm:extremal-colimit}\
  \begin{enumerate}
  \item A concrete cone $G:\colimconel{\cA} \to \cP$ is universal if and only if the corresponding cocone is a colimit, in the strong sense of~\eqref{eq:lcolim}.\label{item:ec1}
  \item A concrete cone $G:\limconel{\cA} \to \cP$ is universal if and only if the corresponding cocone is a limit, in the strong sense of~\eqref{eq:llim}.
  \item A concrete cone $G:\colimconenl{\cA} \to \cP$ is universal if and only if the corresponding cocone is a colimit, in the strong sense of~\eqref{eq:nlcolim1}--\eqref{eq:nlcolim2}.
  \item A concrete cone $G:\limconenl{\cA} \to \cP$ is universal if and only if the corresponding cocone is a limit in the sense of~\eqref{eq:nllim}.
  \end{enumerate}
\end{prop}
\begin{proof}
  We prove~\ref{item:ec1}; the others are analogous.
  Because the vertex $T^+$ of $\colimconel{\cA}$ is linear and positive, $(T^-,\Psi)$ is admissible just when $\Psi$ contains no positive nonlinear objects.
  An extension of $G:\colimconel{\cA} \to \cP$ to some pre-expansion $\preexpandp{\colimconel{\cA}}{\Psi}$ thus consists of a list $\Theta$ of nonlinear objects of \cP, lists $\Gamma$ and $\Delta$ of linear objects of \cP, and a morphism $\ftil_i\in \Pl(\Theta | \Gamma,G A_i ; \Delta)$ for each object $A_i\in \cA$, such that $\ftil_i \circ G g = \ftil_j$ for each morphism $g : A_j\to A_i$ in \cA.
  This is precisely an element of $\lim_i \Pl(\Theta|\Gamma,A_i;\Delta)$, the right-hand side of~\eqref{eq:lcolim}.

  A further extension to the expansion $\expandp{\colimconel{\cA}}{\Psi}$ is then determined by a morphism $\chi \in \Pl(\Theta | \Gamma,G T ; \Delta)$ such that $\chi\circ_{GT} f_i = \ftil_i$ for all $A_i\in \cA$.
  To say that there is a unique such morphism is thus precisely to say that the natural map from left-to-right in~\eqref{eq:lcolim} is a bijection.
\end{proof}

\begin{defi}
  If $H:\cC\to\cQ$ is a concrete cone, we say that $\pi:\cP\to\cQ$ \textbf{has extremal lifts of $H$} if for any lift $G:\reduct{\cC}\to \cP$ of the reduct of $\cC$ to $\cP$, there exists a compatible lift of $H$ that is $\pi$-extremal:
  \[
    \begin{tikzcd}[column sep=large]
      \reduct{\cC} \ar[r,"G"] \ar[d] & \cP \ar[d,"\pi"]\\
      \cC \ar[r,"H"'] \ar[ur,dashed,"\pi\text{-}\mathrm{ext}" description] & \cQ
    \end{tikzcd}
  \]
\end{defi}

\begin{exa}\label{eg:extremal-bifib}
  By \cref{thm:extremal-cartesian}, $\pi$ is a bifibration if and only if it has extremal lifts of all the abstract cartesianness cones from \cref{eg:extremal-univ}.
\end{exa}

\begin{defi}\label{defn:bicomplete}
  We say that an \lnl polycategory is \textbf{bicomplete} if its unique map to the terminal object has extremal lifts of all concrete cones for the abstract limit and colimit cones from \cref{eg:extremal-limit} (where \cA is small).
\end{defi}

By \cref{thm:extremal-colimit}, bicompleteness is equivalent to having all small limits and colimits of both kinds of objects, in the sense described in \cref{sec:lnl-polycategories}.

As pointed out by a referee, the generalization of \cref{defn:bicomplete} to a relative notion over an arbitrary base \cQ is a little subtle: there are at least two natural-seeming possibilities.

\begin{defi}
  Let $\pi:\cP\to\cQ$ be a functor of \lnl polycategories.
  \begin{enumerate}
  \item We say $\pi$ is \textbf{relatively bicomplete} if it has extremal lifts of all concrete cones $H:\cC\to\cQ$ where \cC is one of the abstract cones from \cref{eg:extremal-limit} (where \cA is small).
  \item We say $\pi$ is \textbf{fiberwise bicomplete} if it has extremal lifts only of such cones that have the additional property that $H$ factors through the terminal object (equivalently, its image contains only identity maps).
  \end{enumerate}
\end{defi}

The two coincide in the ``absolute'' case when \cQ is terminal, or more generally when it satisfies the following condition.

\begin{prop}\label{thm:rel-fib-bicomplete}
  If \cQ contains no nonidentity unary co-unary morphisms between two objects of the same sort (linear or nonlinear), then a functor $\pi:\cP\to\cQ$ is relatively bicomplete if and only if it is fiberwise bicomplete.
  In particular, this is the case when \cQ is subterminal.\qed
\end{prop}

\begin{exa}\label{eg:limits-slice}
  As noted in \cref{sec:lnl-polycategories}, an \lnl multicategory cannot have a terminal linear object or an initial linear or nonlinear object when considered as an \lnl polycategory.
  However, while a concrete cone $G:\cC\to\cP$ of such a shape in an \lnl multicategory cannot be universal, it can be $\pi$-extremal for the unique functor $\pi:\cP\to\mathsc{lnlmulti}$ (see \cref{rmk:slice}).
  This yields the correct ``modified'' notion of initial and terminal object in an \lnl multicategory as discussed in \cref{sec:lnl-polycategories}, since not all expansions of this cone factor through $\mathsc{lnlmulti}$.
  Since $\mathsc{lnlmulti}$ is subterminal, \cref{thm:rel-fib-bicomplete} applies to \lnl multicategories, so there is no ambiguity in the correct notion of ``bicomplete \lnl multicategory''.

  Similarly, we obtain the correct notions of limit and colimit for symmetric polycategories, cartesian multicategories, symmetric multicategories, and CBPV pre-structures.
  The non-subterminals from \cref{rmk:planar,rmk:double-split} also satisfy the condition of \cref{thm:rel-fib-bicomplete}, so there is no ambiguity in their correct notion of bicompleteness either.
\end{exa}

The potential difference between relative and fiberwise bicompleteness can be attributed to the fact that \cref{eg:extremal-univ,eg:extremal-limit} overlap.
Specifically, the abstract cartesianness cone $\cartcone{\Phi}{K}$ when $\Phi$ is a single object of the same sort and opposite sign as $K$ coincides with an abstract limit or colimit cone where \cA is the terminal category.
In the absolute case, this is a universal unary co-unary morphism between objects of the same sort, as in \cref{thm:univ-iso}, or equivalently a limit or colimit of a single object, which is trivial.
But if $\pi:\cP\to\cQ$ has extremal lifts for these unary co-unary cones, then its underlying ordinary functors between categories of linear and nonlinear objects are each both a fibration and opfibration, in the classical Grothendieck sense.

\begin{exa}
  The non-subterminal $\cQ=\mathsc{smadj}$ from \cref{eg:smadj} contains a nonidentity morphism $\mathsc{p}\to \mathsc{n}$ between linear objects.
  Thus, while a fiberwise bicomplete object of $\lnlpoly/\mathsc{smadj}$ contains only limits and colimits of positive and negative objects individually, a relatively bicomplete one also includes the cartesian lifts mentioned in \cref{eg:smadj} that make it an adjunction of symmetric multicategories.
\end{exa}

Since these adjoint functors relating positive and negative objects are analogous to the exponential modalities relating linear and nonlinear objects, and do not intuitively look like a sort of ``limit'', it is natural to view them as belonging to birepresentability and \emph{not} to ``completeness''.
As pointed out by the referee, this argues for fiberwise bicompleteness as the correct notion of ``bicompleteness'' for general base objects \cQ.

Our general notion of ``extremal cone'' also includes examples that don't fall into either \cref{eg:extremal-univ} or \cref{eg:extremal-limit}.
However, our main purpose in introducing it is to give a common language to talk about these two examples.
To this end, we note that together these two examples suffice to reconstruct all extremal cones.

\begin{thm}\label{thm:all-limits}
  For any functor $\pi:\cP\to\cQ$ of \lnl polycategories, the following are equivalent.
  \begin{enumerate}
  \item \cP has an extremal lift of any concrete cone $H:\cC\to\cQ$ (with \cC small).\label{item:al2}
  \item \cP is a relatively bicomplete bifibration.\label{item:al1}
  \item \cP is a fiberwise bicomplete bifibration.\label{item:al1a}
  \end{enumerate}
\end{thm}
\begin{proof}
  \cref{eg:extremal-bifib,defn:bicomplete} show that~\ref{item:al2}$\Rightarrow$\ref{item:al1}, and clearly~\ref{item:al1}$\Rightarrow$\ref{item:al1a}.
  So let us assume~\ref{item:al1a}, and let $H:\cC\to\cQ$ be a cone and $G:\reduct{\cC}\to \cP$ a lift of its reduct to \cP.
  For any abstract projection $f\in \cC(\Phi,K)$, let $\ftil\in\cP(G\Phi,K_f)$ be $\pi$-extremal in $K_f$ and such that $\pi(\ftil) = H(f)$ and hence $\pi(K_f) = H(K)$, where the sign and linearity of $K_f$ are the same as that of $K$.
  Such a morphism exists because $\pi$ is a bifibration.

  Now for any abstract transition $g\in \cC(\Psi,L)$ and any abstract projection $f\in \cC(L\D,\Phi,K)$ that it is composable with, producing an abstract projection $f\circ_L g \in \cC(\Psi,\Phi,K)$, the composite $\ftil \circ G g \in \cP(G\Psi,G\Phi,K_f)$ satisfies
  \[\pi(\ftil\circ G g) = \pi(\ftil) \circ \pi(G g) = H(f) \circ H(g) = H(f\circ g) .\]
  Thus, by the universal property of $\widetilde{f\circ_L g} \in \cP(G\Psi,G\Phi,K_{f\circ_L g})$ it induces a unique morphism $\gtil \in \cP(K_{f\circ_L g}\D,K_f)$ such that $\pi(\gtil) = 1_K$.

  Now these objects $K_f$ and morphisms $\gtil$ form a small diagram of objects of \cP (linear or nonlinear according as $K$ is such) lying in the fiber over $K$.
  In particular, therefore, the image of this diagram under $\pi$ admits a specified cone (if $K$ is negative) or cocone (if $K$ is positive) with vertex $H(K)$, consisting entirely of identity maps.
  Thus, since $\pi$ is fiberwise bicomplete, this cone of identity maps has a $\pi$-extremal lift.
  Composing the projections of this lift with the morphisms $\ftil$ yields a $\pi$-extremal concrete cone $\cC\to\cP$ extending $G$ and lifting $H$.
\end{proof}

Of course, there are analogous results in which set-theoretic size of the limits and colimits and of the abstract cones are limited in chosen ways.
We also have a version of \cref{thm:univ-uniq} and its converse.

\begin{prop}\label{thm:extremal-uniq}
  Given $\pi:\cP\to\cQ$ and an abstract cone \cC with vertex $K$, if $F,G:\cC \to \cP$ coincide on the reduct $\reduct{\cC}$ and are both $\pi$-extremal, then there is a unique isomorphism $\phi:F(K) \cong G(K)$ such that $\pi(\phi)$ is an identity and such that $\phi \circ_K F(f) = G(f)$ for all abstract projections $f$ in \cC.\qed
\end{prop}

Given $\pi:\cP\to\cQ$, an abstract cone \cC with vertex $K$, a concrete cone $G:\cC\to \cP$, and an isomorphism $\phi:G(K) \cong K'$, there is a concrete cone $G_\phi : \cC\to \cP $ that agrees with $G$ on the reduct $\reduct{\cC}$, sends the vertex to $K'$, and the abstract projections $f$ to $G_\phi(f) = \phi\circ G(f)$.

\begin{prop}\label{thm:extremal-iso}
  If in the above construction $G$ is $\pi$-extremal, so is $G_\phi$.\qed
\end{prop}

And a composition property for functors:

\begin{prop}\label{thm:extremal-fibcomp}
  Suppose $\pi_1:\cP_1\to \cP_2$ and $\pi_2:\cP_2\to\cP_3$, and a concrete cone $G:\cC\to \cP_1$.
  If $G$ is $\pi_1$-extremal and $\pi_1G$ is $\pi_2$-extremal, then $G$ is $\pi_2\pi_1$-extremal.
\end{prop}
\begin{proof}
  In the diagram in \cref{fig:fibcomp},
  to find a unique lift in the rectangle, we first find a unique lower diagonal lift and then a unique upper one.
\end{proof}

\begin{figure}
  \centering
  \[
    \begin{tikzcd}
      \cC \ar[r] \ar[rr,bend left,"G"] & \preexpand{\cC}{\Psi} \ar[r] \ar[dd] & \cP_1 \ar[d,"\pi_1"] \\
      & & \cP_2 \ar[d,"\pi_2"] \\
      & \expand{\cC}{\Psi} \ar[r] \ar[uur,dashed,"\exists!"] \ar[ur,dashed,"\exists!"'] & \cP_3
    \end{tikzcd}
  \]
  \caption{Diagram for \cref{thm:extremal-fibcomp}}
  \label{fig:fibcomp}
\end{figure}

\section{Doctrines and sketches}
\label{sec:sketches}

In \cref{sec:relation-literature} we encountered a long list of categorical structures that form locally full sub-2-categories of \lnlpoly.
In this section and the next we will define a general class of such sub-2-categories, which we call \textbf{(sorted, \lnl) doctrines}.
Inspecting the examples in \cref{sec:relation-literature}, we see that each is characterized by three kinds of data:
\begin{enumerate}
\item Restrictions on the kinds of objects (e.g.\ no nonlinear objects) and the arities of morphisms (e.g.\ all linear morphisms are co-unary).
  We have already remarked that these restrictions can be detected by slicing \lnlpoly over subterminals such as $\mathsc{symmulti}$, $\mathsc{cbpv}$, etc.
  More generally, we can equip the objects or morphisms with \emph{structure} by slicing over a non-subterminal object, such as $\mathsc{plmulti}$, $\mathsc{dblsplit}$, and $\mathsc{smadj}$ in \cref{rmk:planar,rmk:double-split,eg:smadj}.\label{item:doc1}
\item Existence of universal cones, for all cones in some family (e.g.\ existence of tensors, internal-homs, modalities, or limits or colimits).
  Sometimes the universal property of these cones has to be restricted to respect the allowed arities of morphisms, which corresponds to asking for cartesian lifts over the base objects in~\ref{item:doc1}.\label{item:doc2}
\item Requirements that certain adjunctions are of some ``Kleisli type'', hence determined by a monad, a comonad, or both.\label{item:doc3}
\end{enumerate}
In this section we define \lnl doctrines, which encapsulate~\ref{item:doc1} and~\ref{item:doc2}.
In the next section we extend these to ``sorted doctrines'' that incorporate~\ref{item:doc3} as well.

\begin{defi}
  An \textbf{\lnl doctrine} $\dD$ is an \lnl polycategory $\abs{\dD}$ equipped with a family of concrete cones $G:\cC\to\abs{\dD}$, called the \textbf{\dD-cones}.
  We say \dD is \textbf{small} if $\abs{\dD}$ is small and the family of cones is also small.

  Given such a doctrine, a \textbf{\dD-category} is an \lnl polycategory \cP equipped with a functor $\pi:\cP\to\abs{\dD}$ that has extremal lifts of all \dD-cones:
  \[
    \begin{tikzcd}[column sep=large]
      \reduct{\cC} \ar[r] \ar[d] & \cP \ar[d,"\pi"] \\
      \cC \ar[r,"G"'] \ar[ur,dashed,"\exists","\pi\text{-ext}"'] & \abs{\dD}
    \end{tikzcd}
  \]
  A \textbf{\dD-functor} between \dD-categories is a morphism in $\lnlpoly/\abs{\dD}$ that preserves $\pi$-extremal lifts of \dD-cones, and a \textbf{\dD-transformation} between \dD-functors is a 2-cell in $\lnlpoly/\abs{\dD}$.
  This defines a locally full sub-2-category $\dcat\subseteq \lnlpoly$.
\end{defi}

\begin{exa}\label{eg:doctrine-rep}
  Let $\abs{\dD} = \mathsc{lnlpoly}$ be terminal, and let the \dD-cones contain one representative from each isomorphism class of cones\footnote{An isomorphism of abstract cones is an isomorphism of \lnl polycategories that preserves the vertices.} constructed in \cref{eg:extremal-univ}.
  Then by \cref{thm:rep}, a \dD-category is a birepresentable \lnl polycategory.

  Similarly, if $\abs{\dD}=\mathsc{lnlpoly}$ and the \dD-cones contain one representative of each isomorphism class of cones, by \cref{thm:all-limits} a \dD-category is a bicomplete birepresentable \lnl polycategory.
  (Note that this doctrine is not small.)
  We can include more restricted classes of limits as well by combining the cones from \cref{eg:extremal-univ} with some of those from \cref{eg:extremal-limit}; e.g.\ there is a (small) doctrine for birepresentable \lnl polycategories with finite products and coproducts (additives).
\end{exa}

\begin{exa}\label{eg:doctrine-slice}
  Taking $\abs{\dD}$ to be one of the subterminals $\mathsc{sympoly}$, $\mathsc{symmulti}$, $\mathsc{cartmulti}$, $\mathsc{cat}$, and $\mathsc{lnlmulti}$ from \cref{rmk:slice}, we can equip it with a family of cones that specify desired universal morphisms and/or limits and colimits with the appropriately restricted universal properties for the corresponding subclass of \lnl polycategories, which as noted in \cref{thm:slice-bifib,eg:limits-slice} can be characterized by saying that certain cones are $\pi$-extremal rather than globally universal.
  For instance, there is a doctrine $\dD$ with $\abs{\dD}=\mathsc{symmulti}$ for which the \dD-categories are bicomplete closed symmetric monoidal categories; another doctrine with $\abs{\dD}=\mathsc{symmulti}$ for which the \dD-categories are symmetric monoidal categories (not necessarily closed or bicomplete); a doctrine with $\abs{\dD} = \mathsc{lnlmulti}$ for which the \dD-categories are \lnl adjunctions; and so on.
  Similarly, taking $\abs{\dD} = \mathsc{cbpv}$ or $\mathsc{ecbv}$ as in \cref{rmk:cbpv,eg:cbpv-slice}, we have doctrines for CBPV adjunction models, EEC+ models, and ECBV models.
\end{exa}

Non-subterminal examples can incorporate further adjunctions.
For instance, based on \cref{eg:smadj} we can formulate a doctrine for symmetric monoidal adjunctions.
By combining this idea with arity restrictions as in \cref{rmk:cbpv} (CBPV structures), we obtain doctrines for models of polarized linear calculi as in~\cite{cfm:adjpol}:

\begin{exa}\label{eg:pol}
  Let $\mathsc{linpol}$ be the \lnl multicategory with two objects $\mathsc{p},\mathsc{n}$, both linear, a unique morphism $\Gamma \to \mathsc{p}$ when $\Gamma$ consists entirely of $\mathsc{p}$'s, and a unique morphism $\Gamma \to \mathsc{n}$ when $\Gamma$ contains no more than one $\mathsc{n}$.
  If we equip it with the single-projection cones $(\mathsc{p},\mathsc{p}) \to \underline{\mathsc{p}}$ and $() \to \underline{\mathsc{p}}$ (with vertex underlined), we obtain a doctrine whose categories consist of a symmetric \emph{monoidal} category \cE, a category \cL enriched over the \emph{Day convolution} monoidal structure on $[\cE\op,\fSet]$, and an $[\cE\op,\fSet]$-enriched functor $R:\cL \to [\cE\op,\fSet]$.
  As in \cref{rmk:cbpv}, by adding the following cones we enforce additional universal properties:
  \begin{enumerate}
  \item From $\underline{\mathsc{p}}\to\mathsc{n}$ we make $R$ land inside \cE.\label{item:pol1}
  \item From ${\mathsc{p}}\to\underline{\mathsc{n}}$ we give $R:\cL \to \cE$ a left adjoint.\label{item:pol2}
  \item From $(\underline{\mathsc{p}},{\mathsc{n}})\to \mathsc{n}$ we make \cL enriched over \cE.
  \item From $(\mathsc{p},\underline{\mathsc{n}})\to \mathsc{n}$ we give \cL powers by representables.\label{item:pol3}
  \item From $(\mathsc{p},{\mathsc{n}})\to \underline{\mathsc{n}}$ we give \cL copowers by representables.
  \end{enumerate}
  In particular, with~\cref{item:pol1,item:pol2,item:pol3} we obtain a doctrine for the \textbf{IMLL$^\eta_p$ models} of~\cite{cfm:adjpol}.
  And if we additionally include cones for $\oplus,0$ of positive objects and $\with,\top$ of negative ones, we obtain their \textbf{IMALL$^\eta_p$ models}.

  Now let $\mathsc{lnlpol}$ have two linear objects $\mathsc{p},\mathsc{n}$ and one \emph{nonlinear} object $\mathsc{x}$, with all nonlinear homsets singletons, a unique morphism $(\Theta \mid \Gamma) \to \mathsc{p}$ if $\Gamma$ consists entirely of $\mathsc{p}$'s, and a unique morphism $(\Theta\mid\Gamma) \to \mathsc{n}$ when $\Gamma$ contains no more than one $\mathsc{n}$.
  With the above cones for an IMLL$^\eta_p$ model, cones for $\times,1$, and also the morphisms $\underline{\mathsc{x}}\to \mathsc{p}$ and $\mathsc{x}\to \underline{\mathsc{p}}$ representing a $\uoc$ defined on positive objects and an $\foc$ valued in positive objects,
  this yields a doctrine for the \textbf{IMELL$^\eta_p$ models} of~\cite{cfm:adjpol}.
  Adding $\oplus,0$ of positive objects, $\with,\top$ of negative ones, plus $+,\varnothing$, we obtain \textbf{IMLL$^\eta_p$ models}.
\end{exa}

Note that the morphisms in $\dcat$ preserve the specified universal properties up to canonical isomorphism.
This is 2-categorically correct, but means that $\dcat$ is not well-endowed with strict limits and colimits.
Thus, following the philosophy of homotopy theory, we embed it in a larger but better-behaved category.

\begin{defi}
  Given an \lnl doctrine \dD, a \textbf{\dD-sketch} is an \lnl polycategory \cP together with a functor $\pi:\cP\to\abs{\dD}$, and for each \dD-cone $G:\cC\to\abs{\dD}$ a set (perhaps empty) of lifts of $G$ to \cP that we call \textbf{proto-extremal}:
  \[
    \left\{\begin{tikzcd}
      & \cP \ar[d,"\pi"] \\
      \cC \ar[r,"G"'] \ar[ur,dashed]  & \abs{\dD}
    \end{tikzcd}\right\}.
  \]
  A \textbf{morphism of \dD-sketches} is a functor in $\lnlpoly/\abs{\dD}$ that preserves proto-extremal cones; a \textbf{transformation} is an arbitrary 2-cell in $\lnlpoly/\abs{\dD}$.
  This defines a 2-category $\dsketch$.

  A \dD-sketch is \textbf{realized} if every proto-extremal cone is in fact $\pi$-extremal.
  It is \textbf{saturated} if whenever $H:\cC\to\cP$ is proto-extremal, where $K$ is the vertex of \cC, and $\phi : H(K)\cong K'$ is an isomorphism in \cP such that $\pi(\phi)$ is an identity, the cone $H_\phi : \cC\to \cP$ constructed before \cref{thm:extremal-iso} is also proto-extremal.
  It is \textbf{precomplete} if for any \dD-cone $G:\cC\to\abs{\dD}$, any lift of its reduct $\reduct{\cC}\into \cC\to \abs{\dD}$ to \cP can be extended to a proto-extremal cone:
  \[
    \begin{tikzcd}
      \reduct{\cC} \ar[r] \ar[d] & \cP \ar[d,"\pi"] \\
      \cC \ar[r,"G"'] \ar[ur,dashed,"\exists","\text{p.e.}"'] & \abs{\dD}
    \end{tikzcd}
  \]
  Finally, it is \textbf{(\dD-)complete} if it is realized, saturated, and precomplete.
\end{defi}

\begin{prop}
  The 2-category of \dD-complete sketches is equivalent, as a strict 2-category, to the 2-category $\dcat$ of \dD-categories.
\end{prop}
\begin{proof}
  We regard a \dD-category as a sketch by designating every $\pi$-extremal lift of a \dD-cone as proto-extremal.
  This defines a 2-functor $\dcat\to\dsketch$, which lands inside the \dD-complete sketches (using \cref{thm:extremal-iso}) and is an isomorphism on hom-categories.
  Moreover, precompleteness and realization make any \dD-complete sketch into a \dD-category, while in the presence of these properties saturation is equivalent (using \cref{thm:extremal-uniq}) to saying that all $\pi$-extremal lifts of \dD-cones are proto-extremal; hence the functor is essentially surjective as well.
\end{proof}

\dsketch is a complete and cocomplete strict 2-category, with limits and colimits created in \lnlpoly.
If \dD is small, \dsketch is even locally presentable.
It is also better-endowed with adjunctions, particularly ones arising from doctrine morphisms.

\begin{defi}
  Let $\dD_1,\dD_2$ be \lnl doctrines.
  A \textbf{doctrine map} $\F:\dD_1\to\dD_2$ is a functor $\abs{\F}:\abs{\dD_1}\to\abs{\dD_2}$ together with, for each $\dD_1$-cone $G:\cC\to \abs{\dD_1}$, a $\dD_2$-cone $\cC_{\F} \to \abs{\dD_2}$ and an isomorphism of abstract cones $\cC\cong \cC_{\F}$ (preserving the vertex) making the evident square commute.
\end{defi}

\begin{prop}\label{thm:sketch-adj}
  Any doctrine map $\F:\dD_1\to\dD_2$ induces a strict 2-adjunction (i.e.\ an adjunction of \fCat-enriched categories)
  \[\ladj{\F} : \sketch{\dD_1}\toot \sketch{\dD_2} : \radj{\F}.\]
\end{prop}
\begin{proof}
  We have a 2-adjunction
  \[ \ladj{\F} : \lnlpoly/\abs{\dD_1} \toot \lnlpoly/\abs{\dD_2} : \radj{\F} \]
  given by composition with $\abs{\F}$ and pullback along it, so it suffices to lift this to sketches.
  For the right adjoint $\radj{\F}$, we define a lift $\cC \to \radj{\F}\cP$ of some $\dD_1$-cone $\cC \to \abs{\dD_1}$ to be proto-extremal if the composite $\cC_{\F}\cong \cC \to \radj{\F}\cP \to \cP$ is proto-extremal:
  \[
    \begin{tikzcd}
      & \radj{\F}\cP \ar[rr] \ar[d] \drpullback && \cP \ar[d] \\
      & \abs{\dD_1} \ar[rr] & {}& \abs{\dD_2} \\
      \cC \ar[uur,dashed] \ar[ur] \ar[rr,"\cong",<->] && \cC_{\F} \ar[ur] \ar[uur,dashed,crossing over]
    \end{tikzcd}
  \]
  For the left adjoint $\ladj{\F}$, we define a lift $\cD\to \ladj{\F}\cP$ of some $\dD_2$-cone $\cD \to \abs{\dD_2}$ to be proto-extremal if the latter $\dD_2$-cone is the $F$-image of some $\dD_1$-cone $\cC \to \abs{\dD_1}$ and there is a proto-extremal lift $\cC\to \cP$ making the evident diagram commute:
  \[
    \begin{tikzcd}
      & \cP \ar[rr,equals] \ar[d] \drpullback && \ladj{\F}\cP \ar[d] \\
      & \abs{\dD_1} \ar[rr] & {}& \abs{\dD_2} \\
      \cC \ar[uur,dashed] \ar[ur] \ar[rr,"\cong",<->] && \cC_{\F}= \cD \ar[ur] \ar[uur,dashed,crossing over]
    \end{tikzcd}
  \]
  It is straightforward to check that these constructions lift the 2-adjunction.
\end{proof}

We really want an analogous adjunction $\cat{\dD_1} \toot \cat{\dD_2}$, but this can only be expected to be a pseudo 2-adjunction, satisfying its universal property up to equivalence.\footnote{A pseudo 2-adjunction is traditionally called a ``biadjunction'', but this seems inadvisable here since we are using the prefix ``bi-'' with a different connotation in ``bifibration'' and ``bicomplete''.}
We will construct this in \cref{sec:doc-adj}, using the above strict 2-adjunction.

\section{Sorted doctrines}
\label{sec:sorted-doctrines}

In \cref{sec:relation-literature} we chose to represent monads and comonads as their Kleisli adjunction rather than their Eilenberg--Moore adjunction (or any other), due to \cref{thm:kleisli-eso}.
Thus, to impose the third kind of ``Kleisli type'' condition mentioned in \cref{sec:sketches}, it suffices to assert essential-surjectivity properties for some of the modalities.

\begin{defi}
  An \textbf{arrow-type abstract cone} is determined by two signed objects $K,L$ (each linear or nonlinear).
  Its vertex is $K$, and its only nonidentity morphism is an abstract projection in $\cC(L,K)$.
\end{defi}

If a cone belonging to a doctrine \dD is arrow-type determined by $K,L$, then by choosing extremal lifts, any \dD-category can be equipped with a functor from the fiber over $L$ to the fiber over $K$.
This functor is contravariant if $K$ and $L$ have the same sign and covariant if they have different signs.
Of the cones from \cref{eg:extremal-univ} representing the basic universal properties from \cref{sec:lnl-polycategories}, $\foc,\uoc,\fwn,\uwn,\duals$ are arrow-type.

\begin{defi}
  A \textbf{sorted \lnl doctrine} is an \lnl doctrine \dD together with:
  \begin{enumerate}
  \item A partition of the objects of $\abs{\dD}$ (which we call \textbf{sorts}) into \textbf{primitive sorts} and \textbf{derived sorts}.
  \item For each derived sort $R$, there is exactly one \dD-cone $G_R : \cC_R \to \abs{\dD}$ whose concrete vertex $G(K)$ is $R^-$ or $R^+$, and this is an arrow-type cone whose other vertex $G(L)$ is a primitive sort.
    We call it the \textbf{sorting cone} for $R$.
  \end{enumerate}
\end{defi}

\begin{defi}
  Let \dD be a sorted doctrine and $\pi :\cS\to\abs{\dD}$ a \dD-sketch.
  \begin{itemize}
  \item \cS is \textbf{well-sorted} if for every derived sort $R$ and every object $\Rtil \in \pi^{-1}(R)$, there exists a proto-extremal lift of $G_R$ that maps the vertex to $\Rtil$.
  \item \cS is \textbf{strictly well-sorted} if for every derived sort $R$ with corresponding primitive sort $S$, there is a specified bijection between the objects of $\pi^{-1}(R)$ and $\pi^{-1}(S)$ and, for each $\Rtil$ and $\Stil$ that correspond under this bijection, a specified proto-extremal lift of $G_R$ with entries $\Rtil$ and $\Stil$.
  \end{itemize}
  We write $\dscat$ for the 2-category of well-sorted \dD-complete sketches (\dD-categories).
\end{defi}

Thus a \dD-category is well-sorted if and only if the functor $\pi^{-1}(S)\to\pi^{-1}(R)$ induced by each sorting cone is essentially surjective on objects, and strictly well-sorted if a particular choice of this functor has been made that is bijective on objects.
We are ``really'' interested in the strictly well-sorted sketches, but the non-strictly well-sorted ones are more convenient to work with technically.
Fortunately we have the following:

\begin{prop}
  For a sorted doctrine \dD, every well-sorted \dD-category is equivalent in \dsketch to a strictly well-sorted one.
\end{prop}
\begin{proof}
  If $\pi:\cS\to\abs{\dD}$ is well-sorted, for each derived sort $R$ with corresponding primitive sort $S$ we have an essentially surjective functor $\pi^{-1}(S) \to \pi^{-1}(R)$.
  Thus, we can replace $\pi^{-1}(R)$ by an equivalent category whose objects are those of $\pi^{-1}(S)$, making the functor bijective on objects.
  These equivalences on fibers extend to an equivalence of \dD-categories.
\end{proof}

Thus, $\dscat$ is equivalent (as a bicategory) to its full sub-2-category of strictly well-sorted \dD-categories.

\begin{exa}
  Any \lnl doctrine can be made sorted with all sorts primitive, so that all \dD-sketches are (vacuously) strictly well-sorted.
\end{exa}

\begin{exa}\label{eg:kleisli}
  Let \dD be any doctrine for which $\abs{\dD}$ has exactly one nonlinear object $\mathsc{x}$ and one linear object $\mathsc{a}$, such as $\mathsc{lnlmulti}$ or the terminal object $\mathsc{lnlpoly}$.
  Suppose furthermore that the only \dD-cone with vertex $\mathsc{x}^{\pm}$ is an arrow-type cone with vertex $\mathsc{x}^-$ and abstract projection in $\cC(\mathsc{a}^+,\mathsc{x}^-)$ (that is, a $\uoc$-cone).
  Then we can make \dD a sorted doctrine where $\mathsc{a}$ is primitive, $\mathsc{x}$ is derived, and this cone is the sorting cone.

  We call this a \textbf{Kleisli sorted} doctrine.
  Then a \dD-category is strictly well-sorted just when it is of Kleisli type (\cref{defn:kleisli-type}).
  If \dD also contains $\foc$, then by \cref{thm:kleisli-eso} this is equivalent to its being the Kleisli adjunction of the comonad $\oc = \foc\uoc$.
  Thus, the 2-category of symmetric monoidal categories with a linear exponential comonad, and its variants with internal-homs and/or limits and colimits, are equivalent to \dscat for some sorted \lnl doctrine \dD.
  Similarly, by taking an $\foc$-cone as sorting we can represent cartesian monoidal categories with a commutative strong monad.
\end{exa}

\begin{exa}\label{eg:double-kleisli}
  Let \dD be the sorted doctrine defined as follows.
  We take $\abs{\dD} = \mathsc{dblsplit}$, as in \cref{rmk:double-split}; thus a functor $\pi:\cP \to \abs{\dD}$ partitions the nonlinear objects of \cP into left-hand and right-hand ones.
  We equip \dD with cones for $\ten,\unit,\coten,\counit$, as well as $\foc$ defined on left-hand objects, $\uoc$ taking values in left-hand objects, $\fwn$ defined on right-hand objects, and $\uwn$ taking values in right-hand objects.
  And we take the $\uoc$ and $\uwn$ cones as sorting.
  Then a \dD-category is strictly well-sorted just when it has a choice of $\uoc$ and $\uwn$ that are bijective onto the left-hand and right-hand objects respectively.
  A straightforward extension of \cref{thm:kleisli-eso} now shows that this is the same as its being the double-Kleisli adjunction of \cref{thm:ldc-kl} constructed from the linearly distributive category with storage $\Plin$.
  Thus, the 2-categories of linearly distributive or $\ast$-autonomous categories with storage, and their variants with limits and colimits, are equivalent to \dscat for some sorted \lnl doctrine \dD.
\end{exa}

\begin{exa}
  By making one of the sorts in $\mathsc{smadj}$ (\cref{eg:smadj}) derived from the other, we obtain sorted doctrines for lax symmetric monoidal monads or comonads.
\end{exa}

\begin{exa}\label{eg:skew}
  Recall the \lnl multicategory $\mathsc{linpol}$ from \cref{eg:pol}.
  We now rechristen it $\mathsc{symskew}$, calling its two linear objects $\mathsc{l}$ and $\mathsc{t}$; thus there is a unique morphism $\Gamma \to \mathsc{l}$ when $\Gamma$ consists entirely of $\mathsc{l}$'s, and a unique morphism $\Gamma \to \mathsc{t}$ when $\Gamma$ contains no more than one $\mathsc{t}$.
  We make this a sorted doctrine \dD with $\mathsc{t}$ primitive, $\mathsc{l}$ derived, sorting cone $\mathsc{l} \to \mathsc{t}$ (with vertex $\mathsc{l}$), and no other cones.

  A strictly well-sorted \dD-category is determined by the objects over $\mathsc{t}$ and the morphisms with target over $\mathsc{t}$.
  Every object over $\mathsc{l}$ is the image of one over $\mathsc{t}$ by a functor that we may either leave implicit or denote $\mathsf{G}$.
  We call a morphism over $\Gamma \to \mathsc{t}$ \emph{loose} if $\Gamma$ consists entirely of $\mathsc{l}$'s; thus the loose homsets are of the form $\cP(\mathsf{G} A_1,\dots ,\mathsf{G} A_n; B)$.
  We call a morphism over $\Gamma \to \mathsc{t}$ \emph{tight} if $\Gamma$ contains a $\mathsc{t}$; these tight homsets are uniquely determined by those where the \emph{first} element of $\Gamma$ is $\mathsc{t}$, i.e.\ of the form $\cP(A_1, \mathsf{G} A_2,\dots ,\mathsf{G} A_n; B)$.
  This yields a doctrine for the \textbf{symmetric skew multicategories} of~\cite[\S5]{bl:braided-skew}; the morphism $j$ from tight to loose morphisms:
  \[\cP(A_1, \mathsf{G} A_2,\dots ,\mathsf{G} A_n; B) \to \cP(\mathsf{G} A_1, \mathsf{G} A_2, \dots ,\mathsf{G} A_n; B)\]
  is given by composition with the universal arrow $\mathsf{G} A_1 \to A_1$ over the sorting cone.

  In a skew multicategory regarded as an \lnl polycategory over $\mathsc{symskew}$, a tight unit $\unit$ (with restricted universal property) is a ``left universal nullary map classifier''.
  Similarly, for objects $A$ and $B$ over $\mathsc{t}$, with corresponding objects $\mathsf{G} A$ and $\mathsf{G} B$ over $\mathsc{l}$, a tensor product $A\ten \mathsf{G} B$ (which also lies over $\mathsc{t}$) is a ``left universal tight binary map classifier'' (see~\cite[\S4.4]{bl:skew-multi}); and a hom $\mathsf{G} A \hom B$ (also lying over $\mathsc{t}$) corresponds to the notion of ``closedness'' from~\cite[\S4.5]{bl:skew-multi}.
  Thus, by~\cite{bl:skew-multi,bl:braided-skew}, we have sorted \lnl doctrines for (symmetric) skew monoidal categories and (symmetric) skew closed categories.
  In particular, the ``noninvertible associator'' of a skew monoidal category is represented as a comparison map
  \[ (A \ten \mathsf{G} B) \ten \mathsf{G} C \too A \ten \mathsf{G} (B \ten \mathsf{G} C) \]
  whose noninvertibility is unsurprising due to the different placements of $\mathsf{G}$.
  (However, a symmetric closed skew-monoidal category is \emph{not} a bifibration over $\mathsc{symskew}$; it lacks some universal properties, such as a tensor product of two loose objects.)
\end{exa}

\begin{exa}\label{eg:freyd}
  Let \dD be the sorted doctrine with $\abs{\dD} = \mathsc{cbpv}$, with a single cone for $\foc$ that is sorting.
  Thus, a strictly well-sorted \dD-category is a linearly subunary \lnl multicategory with an $\foc$ satisfying a restricted universal property, and such that $\foc$ is bijective from the nonlinear objects to the linear ones.
  Thus, it consists of a cartesian multicategory together with additional linear homsets
  \begin{equation}
    \Pl(X_1,\dots,X_n | ; \, \foc Z).\label{eq:freyd-linhom}
  \end{equation}
  This information uniquely determines the other linear homsets by the \foc-isomorphism:
  \[ \Pl(X_1,\dots,X_n | \foc Y ; \, \foc Z) \cong \Pl(X_1,\dots,X_n, Y | ; \, \foc Z). \]
  However, passing back along these isomorphisms yields multicategorical composition operations on the linear homsets~\eqref{eq:freyd-linhom}:
  \begin{align*}
    \Pl(\Upsilon, X | ; \, \foc Y) \times \Pl(\Theta | ; \, \foc X)
    &\cong \Pl(\Upsilon | \foc X ; \, \foc Y) \times \Pl(\Theta | ; \, \foc X)\\
    &\to \Pl(\Upsilon,\Theta | ; \, \foc Y ).
  \end{align*}
  This composition treats the universal morphisms $\chi \in \Pl(X | ; \, \foc X)$ as identities.
  Moreover, naturality of the $\foc$-isomorphisms implies that these operations are associative in the limited sense that the two composite functions
  \[ \Pl(\Theta_3, Y | ; \, \foc Z) \times \Pl(\Theta_2, X | ; \, \foc Y) \times \Pl(\Theta_1 | ; \, \foc X)
    \to \Pl(\Theta_3,\Theta_2,\Theta_1 | ; \, \foc Z) \]
  are equal.
  However, because of the restricted universal property of $\foc$, nothing forces the two composite functions
  \begin{equation}
    \Pl(\Theta_3, X, Y | ; \, \foc Z) \times \Pl(\Theta_2 | ; \, \foc Y) \times \Pl(\Theta_1 | ; \, \foc X)
    \toto \Pl(\Theta_3,\Theta_2,\Theta_1 | ; \, \foc Z)\label{eq:freyd-nonassoc}
  \end{equation}
  to be equal, as they would be if the homsets~\eqref{eq:freyd-linhom} formed a (cartesian) multicategory.
  This means the linear homsets~\eqref{eq:freyd-linhom} have the structure of a \emph{cartesian pre-multicategory} in the sense of~\cite{sl:univprop-impure}.

  Finally, composing with the universal morphism $\chi \in \Pl(X | ; \, \foc X)$ provides a function
  \[ \Pnl(\Theta ; X) \to \Pl(\Theta | ; \, \foc X) \]
  that respects the cartesian actions, identities, and compositions.
  Moreover, the linear morphisms in the image of this map are \emph{central}, meaning that the two morphisms~\eqref{eq:freyd-nonassoc} are equal if one of the morphisms into $\foc X$ or $\foc Y$ is in this image.
  Thus, we conclude that a strictly well-sorted \dD-category can be identified with a \emph{cartesian Freyd multicategory} in the sense of~\cite{sl:univprop-impure}: a cartesian multicategory $\cV$ of ``values'', a cartesian pre-multicategory $\cC$ of ``computations'', and an identity-on-objects functor $\mathsf{return}: \cV\to \cC$ that preserves centrality.
  (I am indebted to Max New for this observation.)

  A similar doctrine with $\abs{\dD} = \mathsc{symskew}$ yields symmetric Freyd multicategories.
  However, I don't believe there is a sorted doctrine such that the strictly well-sorted \dD-categories can be identified with bare (cartesian or symmetric) pre-multicategories.
  We can ``remove'' the extra information of the nonlinear morphisms by requiring either that the only nonlinear morphisms are projections, or that the nonlinear morphisms coincide with the central linear ones; but neither of these conditions is enforcable doctrinally.
  (Similarly, a \emph{duploid}~\cite{mm:duploids} is an adjunction of ordinary categories with certain restrictions: adjunctions can be modeled doctrinally over the base $\mathsc{adj}$ from \cref{eg:adj}, but the duploid conditions are not doctrinal.)
  
  A nonlinear product $X\times Y$ in a cartesian Freyd multicategory is the same as a \emph{tensor} in the sense of~\cite{sl:univprop-impure}: a (pre)multicategorical tensor in $\cV$ that is preserved by $\mathsf{return}$.
  As shown in~\cite[\S8]{sl:univprop-impure}, a cartesian Freyd multicategory with all such tensors (and units) is equivalent to a Freyd-category in the sense of~\cite{pt:freyd-cats}: a cartesian monoidal category $\cV$, a symmetric premonoidal category~\cite{pr:premonoidal} $\cC$, and an identity-on-objects symmetric premonoidal functor $\mathsf{return}: \cV\to \cC$ that preserves centrality.
  (Alternatively, one can use the characterization of Freyd-categories from~\cite{levy:cbpv}, which is akin to those of CBPV structures in \cref{rmk:cbpv}.)

  Similarly, a nonlinear coproduct $X+Y$ in a cartesian Freyd multicategory is the same as a \emph{sum} in the sense of~\cite{sl:univprop-impure}.
  Finally, a cartesian Freyd multicategory has \emph{function spaces} in the sense of of~\cite[\S6]{sl:univprop-impure} if and only if it has our mixed homs $\mixedeechom$.
  The latter means that for any nonlinear object $X$ and linear object $\foc Y$, there is a nonlinear object $X \mixedeechom \foc Y$, with a universal linear morphism $\chi \in \Pl(X \mixedeechom \foc Y, X | ; \, \foc Y)$ inducing a bijection
  \[ \Pl(\Theta, X | ; \, \foc Y) \cong \Pnl(\Theta ; X \mixedeechom \foc Y) \]
  between computations and values, as in~\cite[(4)]{sl:univprop-impure}.
\end{exa}

Unlike \dD-completeness, well-sortedness is a \emph{coreflective} property.

\begin{prop}\label{thm:sort-corefl}
  For any sorted doctrine \dD, the 2-category of well-sorted \dD-sketches is coreflective in \dsketch, and the coreflector preserves \dD-completeness.
\end{prop}
\begin{proof}
  The coreflection of a \dD-sketch \cS is its full sub-\lnl-polycategory $\cS'$ containing all objects of \cS that lie over primitive sorts, and precisely those objects lying over derived sorts that are the vertex of a proto-extremal lift of the sorting cone.
  Its proto-extremal cones are precisely those of \cS that land in this subcategory.

  If \cS is \dD-complete, $\cS'$ is clearly still realized and saturated.
  To see that $\cS'$ is also still precomplete, note that by construction it still has proto-universal lifts of the sorting cones.
  But by definition, any non-sorting \dD-cone must have a \emph{primitive} vertex, and therefore the proto-universal lifts of such cones in \cS still lie in $\cS'$.
\end{proof}

\begin{exa}
  Over a Kleisli sorted doctrine, the well-sorted coreflection of an \lnl adjunction is the Kleisli adjunction of its comonad.
  Similarly, over the doctrine of linearly distributive categories with storage from \cref{eg:double-kleisli}, the well-sorted coreflection of a linearly distributive \lnl adjunction (\cref{thm:ldc-adj}\ref{item:la3}) is the double-Kleisli adjunction of its induced monad/comonad pair (\cref{thm:ldc-kl}).
\end{exa}

Finally, we remark on what it takes for a doctrine map to preserve well-sortedness.

\begin{defi}\label{defn:docmap-sort}
  Let $\dD_1$ and $\dD_2$ be sorted doctrines.
  A doctrine map $\F:\dD_1\to\dD_2$ is \textbf{sorted} if it preserves primitive sorts, derived sorts, and sorting cones, and moreover for any derived sort $R$ of $\dD_1$, any sorting $\dD_2$-cone with vertex $F(R)$ is the image of some sorting $\dD_1$-cone with vertex $R$.
\end{defi}

\begin{prop}\label{thm:docmap-sort}
  If $\F:\dD_1\to\dD_2$ is a sorted doctrine map, then $\ladj{\F}$ and $\radj{\F}$ from \cref{thm:sketch-adj} preserve well-sortedness.
\end{prop}
\begin{proof}
  For $\ladj{\F}$, let $\pi :\cS \to \abs{\dD_1}$ be a well-sorted $\dD_1$-sketch, let $R$ be a derived $\dD_2$-sort, and let $S \in (F\pi)^{-1}(R)$.
  Then $\pi(S)$ is a derived $\dD_1$-sort.
  So since \cS is well-sorted, there is a proto-extremal lift of its sorting cone $G_R$ that maps the vertex to $S$.
  But by assumption, $F G_R$ is the sorting $\dD_2$-cone of $F(R)$, while by definition this lift of it is also proto-extremal in $\ladj{\F}(\cS)$.
  Thus, $\ladj{\F}(\cS)$ is well-sorted.

  For $\radj{\F}$, let $\pi:\cS\to\abs{\dD_2}$ be a well-sorted $\dD_2$-sketch and $R$ a derived $\dD_1$-sort.
  An object of $\radj{\F}(\cS)$ over $R$ is an object $S\in \pi^{-1}(F(R))$.
  Since $F(R)$ is a derived $\dD_2$-sort and $\cS$ is well-sorted, there is a proto-extremal lift of its sorting cone $G_{F(R)}$ that maps the vertex to $S$.
  By assumption, $G_{F(R)}$ is the image of the sorting $\dD_1$-cone $G_R$, and this proto-extremal lift of $G_{F(R)}$ induces a proto-extremal lift of $G_R$ to $\radj{\F}(\cS)$ mapping the vertex to $S$.
  Thus, $\radj{\F}(\cS)$ is well-sorted.
\end{proof}

\section{The doctrinal completion of a sketch}
\label{sec:rep-cplt}

We will now show that any \dD-sketch can be completed to a \dD-category in a universal way.
Recall (see e.g.~\cite{ar:loc-pres}) that an object \cP of a category is said to be \textbf{injective} with respect to a set of morphisms \cI if for any morphism $\cA\to\cB$ in \cI, any morphism $\cA\to\cP$ can be extended to \cB (not necessarily uniquely):
\[
  \begin{tikzcd}
    \cA \ar[r] \ar[d] & \cP\\
    \cB \ar[ur,dotted]
  \end{tikzcd}
\]
The class of all \cI-injective objects is called a \textbf{small-injectivity class} (``small-'' since $\cI$ is a set rather than a proper class).
If we require the extensions to be \emph{unique}, we obtain the related notions of \textbf{orthogonal} object and \textbf{small-orthogonality class}.
In a category with pushouts, \cP is orthogonal to $\cA\to\cB$ if and only if it is injective with respect to $\cA\to\cB$ and its codiagonal $\cB+_\cA \cB \to\cB$; thus every small-orthogonality class is also a small-injectivity class.

\begin{thm}\label{thm:rep-inj}
  If \dD is small, then the \dD-complete sketches are a small-injectivity class in $\dsketch$.
\end{thm}
\begin{proof}
  Given any \dD-cone $G:\cC\to \abs{\dD}$, we regard it as a \dD-sketch in which the only proto-extremal cone is $G$ itself.
  We also regard its reduct as a \dD-sketch via the composite $\reduct{\cC}\into \cC \to \abs{\dD}$, with no proto-extremal cones at all.
  Then a \dD-sketch \cP is precomplete if and only if it is injective to the inclusions of \dD-sketches $\reduct{\cC}\into \cC$.

  Similarly, given any \dD-cone $G:\cC\to \abs{\dD}$, any expansion of it (\cref{defn:cones}), and any extension of $G$ to $G_\Psi : \expand{\cC}{\Psi} \to \abs{\dD}$, we regard $\expand{\cC}{\Psi}$ and its corresponding pre-expansion $\preexpand{\cC}{\Psi}$ as \dD-sketches via $G_\Psi$ and its restriction to $\preexpand{\cC}{\Psi}$, in which the only proto-extremal cone is $G$.
  Then a \dD-sketch \cP is realized if and only if it is \emph{orthogonal} to the set of inclusions of \dD-sketches $\preexpand{\cC}{\Psi}\into \expand{\cC}{\Psi}$, indexed over all $G$, $\Psi$, and $G_\Psi$.

  Finally, given an abstract cone $\cC$ with vertex $K$, let $\cC_{\cong}$ denote the \lnl polycategory that is \cC with an additional signed object $K'$ isomorphic to $K$.
  There is a fold map $\cC_{\cong} \to \cC$ that collapses $K$ and $K'$ both to $K$, which has two sections $s,s':\cC\to\cC_{\cong}$ sending $K$ to $K$ and $K'$ respectively.
  If $G:\cC\to \dD$ is a \dD-cone, we can regard $\cC_{\cong}$ as a \dD-sketch via the composite $\cC_{\cong} \to \cC\to \abs{\dD}$, in which both $s$ and $s'$ are proto-extremal.
  We can also regard it as a \dD-sketch in which only $s$ is proto-extremal; we denote this sketch by $\cC_{\cong}'$.
  Then a \dD-sketch is saturated if and only if it is injective with respect to the set of inclusions of \dD-sketches $\cC_{\cong}' \into \cC_{\cong}$.

  Let $\cI_\dD$ denote the set of all the morphisms
  \begin{alignat*}{2}
    \reduct{\cC}&\into\cC &
    \preexpand{\cC}{\Psi} &\into \expand{\cC}{\Psi} \\
    \cC_{\cong}' &\into \cC_{\cong} &\qquad
    \expand{\cC}{\Psi} +_{\preexpand{\cC}{\Psi}}\expand{\cC}{\Psi} &\to \expand{\cC}{\Psi}
  \end{alignat*}
  as \cC ranges over the \dD-cones.
  Then a sketch is \dD-complete if and only if it is injective with respect to $\cI_\dD$.
\end{proof}

\begin{rem}\label{rmk:orth}
  The proof shows that realized \dD-sketches are actually a small-ortho\-gonality class.
  Saturated \dD-sketches are also a small-orthogonality class, since the inclusions $\cC_{\cong}' \into \cC_{\cong}$ are epimorphic (being the identity on underlying \lnl polycategories).
\end{rem}

\begin{cor}\label{thm:soa}
  If \dD is small, then every \dD-sketch \cS has a weak \dD-reflection, i.e.\ a map $\cS \to \sdhat$ such that $\sdhat$ is \dD-complete and any map from \cS to a \dD-complete sketch factors through $\sdhat$.
\end{cor}
\begin{proof}
  This is a standard construction applying to any small-injectivity class, known as Quillen's small object argument; see e.g.~\cite[2.1.14]{hovey:modelcats} or~\cite[10.5.16]{hirschhorn:modelcats} or~\cite[12.2.2]{riehl:cht}.
  Let $\cS_0 = \cS$.
  Given $\cS_n$, define inductively $\cS_{n+1}$ as the pushout
  \[
    \begin{tikzcd}
      \coprod_{\iota,u} \cA_{\iota} \ar[r] \ar[d] & \cS_n \ar[d] \\
      \coprod_{\iota,u} \cB_{\iota} \ar[r] & \cS_{n+1} \ulpullback
    \end{tikzcd}
  \]
  where the coproducts are over all $\iota:\cA\to\cB$ in the generating set $\cI_\dD$ and all $u:\cA \to \cS_n$.
  Continue the iteration into transfinite ordinals $n$ by taking colimits at limit stages.
  Then since $\dsketch$ is locally presentable, there is a sufficiently large ordinal $\ka$ such that any map $\cA \to \cS_\ka$, for any $i:\cA\to\cB$, factors through $\cS_n$ for some $n<\ka$, and hence extends to $\cB$ through $\cS_{n+1}$.
  Thus, if we define $\sdhat = \cS_\ka$, it is \dD-complete.
  Moreover, given a \dD-complete sketch $\cT$, we can extend a map $\cS\to \cT$ to each stage $\cS_n$ inductively, using the completeness of \cT at successor stages.
\end{proof}

The factorization $\sdhat \to \cT$ constructed in \cref{thm:soa} is not in general unique,
but we will show that it is
unique up to unique isomorphism.

There is an additional wrinkle, however: if $\dD$ contains operations such as $\hom,\duals$ that are contravariant in some arguments, then \dD-completion cannot be expected to behave well with respect to \emph{noninvertible} 2-cells.
Thus we have to formulate its universal property with respect to $\dsketch_g$, where $\sK_g$ denotes the underlying (2,1)-category of a 2-category $\sK$, containing only the invertible 2-cells.

\begin{thm}\label{thm:bicat-refl}
  For any small \lnl doctrine \dD and \dD-sketch \cS, there is a \dD-complete sketch $\sdhat$ and a map $\cS\to \sdhat$ such that for any \dD-complete sketch $\cP$, the precomposition functor $\dsketch_g(\sdhat,\cP) \to \dsketch_g(\cS,\cP)$ is a surjective equivalence of categories.
  In particular, the sub-2-category of \dD-complete sketches in $\dsketch_g$ (which, recall, is equivalent to $\dcat_g$) is pseudo-reflective.
\end{thm}
\begin{proof}
  In \cref{thm:soa}, $\sdhat$ was constructed as a transfinite composite of pushouts of the generators.
  Since surjective equivalences are closed under pullbacks and inverse transfinite composites, it suffices (see e.g.~\cite[4.2.4]{hovey:modelcats}) to show that for any \dD-complete sketch $\pi:\cP\to\abs{\dD}$ and any morphism $\iota : \cA \to \cB$ in $\cI_\dD$, the induced map $\dsketch_g(\cB,\cP) \to \dsketch_g(\cA,\cP)$ is a surjective equivalence.
  Since it is always surjective on objects, it remains to prove that it is fully faithful.
  Referring to the construction of $\cI_\dD$, there are four cases we need to consider.

  When $\iota$ is an inclusion $\reduct{\cC}\into \cC$ for some \dD-cone $G:\cC\to\abs{\dD}$, we must show that given two $\pi$-extremal lifts $H,K:\cC\to\cP$ of $G$, any isomorphism $\al:H'\cong K'$ between their reducts $H',K' : \reduct{\cC}\to\cP$ can be uniquely extended to a compatible isomorphism $H\cong K$.
  By composing the transitions of $K$ with the components of $\al$ and their inverses (depending on the sign of the relevant signed object), we obtain the data for a pre-expansion of $H$ by a single object, namely the vertex of $K$.
  Thus, extremality of $H$ induces a map between the vertices of $H$ and $K$ (with direction depending on the sign of that vertex).
  Similarly, we obtain a map in the other direction, and the two are inverses.

  When $\iota$ is an inclusion $\preexpand{\cC}{\Psi}\into\expand{\cC}{\Psi}$, we must show that given two expansions $H,K:\expand{\cC}{\Psi} \to \cP$ of $\pi$-extremal lifts, any isomorphism $\al:H'\cong K'$ between their corresponding pre-expansions $H',K': \preexpand{\cC}{\Psi} \to \cP$ is also an isomorphism $H\cong K$.
  Since the inclusion $\preexpand{\cC}{\Psi}\into \expand{\cC}{\Psi}$ is bijective on objects, this is just an extra naturality condition with respect to the factorization morphism.
  But the two sides of this desired naturality square each fit into an expansion of $H$ whose expanders are those of $K$ composed with components of $\alpha$ or their inverses; hence they are equal.

  Finally, when $\iota$ is a codiagonal $\expand{\cC}{\Psi} +_{\preexpand{\cC}{\Psi}} \expand{\cC}{\Psi} \to \expand{\cC}{\Psi}$ or an inclusion $\cC_{\cong}'\into\cC_{\cong}$, full-faithfulness is automatic since these $\iota$'s are bijective on objects and full.
\end{proof}

\begin{prop}\label{thm:cplt-sort}
  For any sorted doctrine \dD and any well-sorted \dD-sketch \cS, the completion $\sdhat$ is also well-sorted.
\end{prop}
\begin{proof}
  Let \cS be well-sorted, and let $(\sdhat)' \to \sdhat$ be the well-sorted coreflection of $\sdhat$.
  Since \cS is well-sorted, the map $\cS \to\sdhat$ factors through $(\sdhat)'$.
  But by \cref{thm:sort-corefl}, $(\sdhat)'$ is \dD-complete, so the universal property of $\sdhat$ induces a map $\sdhat \to (\sdhat)'$ that is a section of the coreflection, up to isomorphism.
  This implies that $\sdhat$ is also well-sorted.
\end{proof}

\section{The sequent calculus of a doctrine}
\label{sec:seq-calc}

Let \dD be an \lnl doctrine and \cS an \lnl polycategory with a map $\pi : \cS \to \abs{\dD}$, which we regard as a \dD-sketch with no proto-extremal cones.
Then \cref{thm:bicat-refl} implies that \cS generates a free \dD-category $\sdhat$.
We now extract a sequent calculus that presents such free \dD-categories from the proof of \cref{thm:bicat-refl}.

For simplicity, for now we suppose that
  \dD is unsorted,
  $\abs{\dD}$ is subterminal, and
  all the cones of \dD are \emph{discrete} (have no nonidentity abstract transitions) and also \emph{finite}.
  This restriction on cones includes cones for universal morphisms, as in \cref{eg:extremal-univ}, and also for finite products and coproducts, as in \cref{eg:extremal-limit}.
These are the primary universal properties that are traditionally considered in logic.
Under these assumptions, we can replace the construction of \cref{thm:soa} by the following simplified version.
\begin{enumerate}
\item First perform the small object argument starting at $\cS_0 = \cS$, using only the inclusions $\reduct{\cC}\into\cC$ for \dD-cones \cC, and when $n>0$ restricting the coproduct to include only the morphisms $u:\reduct{\cC}\to\cS_n$ that do not factor through $\cS_{n-1}$.
  After a countable iteration, this produces a precomplete sketch $\cS_\om$.\label{item:ssoa1}
\item Next perform the small object argument starting at $\cS_\om$, using only the inclusions $\preexpand{\cC}{\Psi}\into\expand{\cC}{\Psi}$ and their codiagonals $\expand{\cC}{\Psi} +_{\preexpand{\cC}{\Psi}} \expand{\cC}{\Psi} \to \expand{\cC}{\Psi}$.
  After a further countable iteration, this produces a realized sketch $\cS_{\om+\om}$.
  Moreover, since these inclusions and codiagonals are bijective on objects and each $\reduct{\cC}$ is discrete, $\cS_{\om+\om}$ is still precomplete.\label{item:ssoa2}
\item Finally, perform one step of the small object argument using the map $\cC_{\cong}'\into\cC_{\cong}$.
  This is sufficient to produce a saturated sketch $\sdhat = \cS_{\om+\om+1}$, which is still precomplete and realized, and hence \dD-complete.\label{item:ssoa3}
\end{enumerate}
In particular, these changes make the argument completely constructive.
(The negation in~\ref{item:ssoa1} may not seem constructive, but the inclusion of $\cS_{n-1}$ into $\cS_{n}$ is decidable on objects because each $\reduct{\cC}\into\cC$ is.)

\begin{figure}
  \centering
  \begin{subfigure}{\textwidth}
    \begin{mathpar}
      \infer{A \in \cS^\tau}{A \ttype}
      \and
      \infer{\cC\text{ a \dD-cone} \\ \reduct{\cC} = \{r_1^{\tau_1},\dots, r_n^{\tau_n}\} \\
        R_1 \istype{\tau_1} \\ \cdots \\ R_n \istype{\tau_n}}
      {\boc[R_1,\dots,R_n] \istype{\tau_\cC}}      
    \end{mathpar}
    \caption{Type-forming rules}
    \label{fig:formation}
  \end{subfigure}
  \\[1em]
  \begin{subfigure}{\textwidth}
    \centering
    \begin{mathpar}
      \infer{R \ttype}{\vdash R^-, R^+}
      \and
      \infer{\vdash \Phi, K \\ \vdash K\D, \Psi}{\vdash \Phi,\Psi}
      \and
      \infer{\vdash \Psi \\ \si : \Phi \to \Psi \text{ a structural map}}{\vdash \Phi}
    \end{mathpar}
    \caption{Structural rules}
    \label{fig:struc}
  \end{subfigure}
  \\[1em]
  \begin{subfigure}{\textwidth}
    \begin{mathpar}
      \infer{f \in \cS(\Phi)}{\vdash \Phi}
    \end{mathpar}
    \caption{Generator rule}
    \label{fig:gen}
  \end{subfigure}
  \\[1em]
  \begin{subfigure}{\textwidth}
    \centering
    \begin{mathpar}
    \mprset{vskip=1ex}
      \infer{\cC\text{ a \dD-cone with vertex } r^\ep \\ \reduct{\cC} = \{r_1^{\tau_1},\dots, r_n^{\tau_n}\} \\\\
        R_1 \istype{\tau_1} \\ \cdots \\ R_n \istype{\tau_n} \\
        f\in \cC(r_{i_1}^{\ep_1},\dots,r_{i_\ell}^{\ep_\ell},r^{\ep}) \text{ an abstract projection}
      }
      {\vdash R_{i_1}^{\ep_1},\,\dots,\, R_{i_\ell}^{\ep_\ell},\, \boc[R_1,\dots,R_n]^{\ep}}
    \end{mathpar}
    \caption{Noninvertible logical rule}
    \label{fig:noninv}
  \end{subfigure}
  \\[1em]
  \begin{subfigure}{\textwidth}
    \centering
    \begin{mathpar}
    \mprset{vskip=1ex}
      \infer{\cC\text{ a \dD-cone with vertex } r^\ep \text{ of class $\tau_\cC$} \\ \reduct{\cC} = \{r_1^{\tau_1},\dots, r_n^{\tau_n}\} \\\\
        R_1 \istype{\tau_1} \\ \cdots \\ R_n \istype{\tau_n} \\
        S_1 \istype{\sigma_1} \\ \cdots \\ S_m \istype{\sigma_m} \\\\
        \abs{\dD}(\tau_\cC^{-\ep}, \sigma_1^{\eta_1},\dots, \sigma_m^{\eta_m}) \neq \emptyset\\\\
        \big\{\vdash R_{i_1}^{\ep_1},\,\dots,\, R_{i_\ell}^{\ep_\ell},\, S_1^{\eta_1},\,\dots,\, S_m^{\eta_m} \big\}_{f\in \cC(r_{i_1}^{\ep_1},\dots,r_{i_\ell}^{\ep_\ell},r^\ep) \text{ an abstract projection}}
      }{\vdash \boc[R_1,\dots,R_n]^{-\ep}, S_1^{\eta_1},\dots, S_m^{\eta_m}}
    \end{mathpar}
    \caption{Invertible logical rule}
    \label{fig:inv}
  \end{subfigure}
  \caption{LNL Sequent calculus}
  \label{fig:seqcalc}
\end{figure}

We can now describe $\sdhat$ using a sequent calculus, defined formally in \cref{fig:seqcalc}.
There are two classes of types, linear and nonlinear, written $A\ltype$ and $X\nltype$.
Generically, we write $R \ttype$ for an arbitrary class $\tau \in \{\mathrm{L},\mathrm{NL}\}$.
The first rule in \cref{fig:formation} says that every object of \cS determines a type of the appropriate class.

By assumption, the reduct $\reduct{\cC}$ of each \dD-cone is a discrete \lnl polycategory with finitely many objects.
We assume the objects of each $\reduct{\cC}$ are ordered as $\{r_1^{\tau_1},\dots,r_n^{\tau_n}\}$, the notation meaning that $r_i$ is of class $\tau_i$, and the vertex $k$ of class $\tau_\cC$.
The second rule in \cref{fig:formation} says that every such cone induces an operation on types.
The notation $\boc[R_1,\dots,R_n]$ is chosen to be generic over the cone \cC, but for particular choices of \cC we use the notations of \cref{sec:lnl-polycategories}, e.g.\ $A\ten B$, $\foc X$, $X \rtimes A$, $A \with B$, etc.

\begin{prop}
  There is a bijection between the valid judgments $R \ttype$ and the $\tau$-objects of $\sdhat$.
\end{prop}
\begin{proof}
  Define the \emph{height} of $R \ttype$ recursively: the height of an object of \cS is zero, while that of $\boc[R_1,\dots,R_n]$ is one more than the maximum height of $R_1,\dots, R_n$.
  (If $n=0$, the height of $\boc[\,]$ is 1.)
  I claim that there is a bijection between the valid judgments $R \ttype$ of height $\le n$ and the $\tau$-objects of $\cS_n$.
  This is true for $n=0$.
  The objects of $\cS_{n+1}$ are those of $\cS_n$ plus a new vertex for each $u:\reduct{\cC}\to \cS_n$ not factoring through $\cS_{n-1}$.
  But the latter are the applications of the $\boc$-rule with at least one premise of height $n$, hence whose conclusion has height $n+1$.
\end{proof}

We denote the sequents in entries-only style as $\vdash \Phi$, where $\Phi$ is an admissible list of signed types, defined analogously to the semantic case in \cref{sec:bifib}.
The structural rules are shown in \cref{fig:struc}.
The first is the identity rule and the second is the cut rule.
The third incorporates exchange for all types, plus contraction and weakening for nonlinear types, as in \cref{sec:bifib}.
Similarly, the generator rule in \cref{fig:gen} says that every morphism of \cS induces a derivation of a sequent.

We may write $\Th \mid \Gm \vdash \De$ for $\vdash \Th^-, \Gm^-, \De^+$, and $\Theta\vdash X$ for $\vdash \Th^-,X^+$.
In this notation, the identity and cut rules multifurcate into linear and nonlinear versions:
\begin{mathpar}
  \infer{A\ltype}{\cdot \mid A \vdash A}\and
  \infer{X \nltype}{X \vdash X}\and
  \infer{\Upsilon \vdash X \\ \Theta,X \vdash Y}{\Theta,\Upsilon \vdash Y}\and
\end{mathpar}
\begin{mathpar}
  \infer{\Theta'\mid\Gamma' \vdash \Delta',A \\ \Theta\mid \Gamma,A\vdash \Delta}{\Theta,\Theta' \mid \Gamma,\Gamma' \vdash \Delta,\Delta'}\and
  \infer{\Upsilon\vdash X \\ \Theta,X\mid \Gamma\vdash \Delta}{\Theta,\Upsilon \mid \Gamma\vdash\Delta}.
\end{mathpar}

We divide the logical rules into \emph{invertible} (right rules for negative types and left rules for positive types) and \emph{noninvertible} (left rules for negative types and right rules for positive types).
The generic noninvertible rule is in \cref{fig:noninv}.
Here $\ep$ and the $\ep_j$'s are signs $+,-$.
For instance, if \cC is the cone for $\ten$, with objects $a,b$ and vertex $c$, there is one abstract projection $f\in \cC(a^-,b^-,c^+)$ and the rule becomes
\[\infer{A \ltype \\ B\ltype}{\cdot\mid A,B\vdash A\ten B}. \]
If \cC is the cone for $\with$,
with objects $a,b$ and vertex $c$, there are two abstract projections $f\in \cC(a^+,c^-)$ and $g\in \cC(b^+,c^-)$, and the rule becomes two:
\begin{mathpar}
  \infer{A \ltype \\ B\ltype}{\cdot \mid A\with B \vdash A}\and
  \infer{A \ltype \\ B\ltype}{\cdot \mid A\with B \vdash B}.
\end{mathpar}
The rules for the modalities are
\begin{mathpar}
  \infer{X \nltype}{X \mid \cdot \vdash \foc X}\and
  \infer{A \ltype}{\uoc A\mid \cdot \vdash A} \and
  \infer{X \nltype}{X \mid \fwn X \vdash \cdot} \and
  \infer{A \ltype}{\uwn A\mid A \vdash\cdot}
\end{mathpar}

Unlike noninvertible rules in most common sequent calculi, ours does not build in a cut.
But we can always apply a cut afterwards, since the latter is primitive in our system.
(We leave cut-elimination for future study.)
Since the modalities are the most novel aspect of this calculus, we list their derived cut-containing rules:
\begin{mathpar}
  \infer{\Theta \vdash X}{\Theta \mid \cdot \vdash \foc X}\and
  \infer{\Theta\mid \Gamma,A \vdash \Delta}{\Theta,\uoc A\mid \Gamma \vdash \Delta} \and
  \infer{\Theta\vdash X}{\Theta \mid \fwn X \vdash \cdot} \and
  \infer{\Theta\mid \Gamma \vdash\Delta,A}{\Theta,\uwn A\mid \Gamma \vdash\Delta}.
\end{mathpar}
If $\abs{\dD}=\mathsc{lnlmulti}$, so $\Delta$ is a singleton, these rules for $\foc$ and $\uoc$ specialize to the noninvertible rules of~\cite{benton:lnl}.
If instead $\abs{\dD}=\mathsc{cbpv}$, so $\Delta$ is a singleton and $\Gamma$ is empty, we obtain the rules of~\cite{levy:adj-cbpv}.

\begin{prop}\label{thm:surj-noninv}
  There is a surjection from the derivations of $\vdash \Phi$ using only the structural, generator, and noninvertible rules to the hom-set $\cS_\om(\Phi)$.
\end{prop}
\begin{proof}
  Such a function is defined by induction on derivations: the structural rules use that $\cS_\om$ is an \lnl polycategory, the generator rule uses the functor $\cS \to \cS_\om$, and the noninvertible rule uses the images of abstract projections under the proto-extremal cones of $\cS_\om$, which exist (by construction, in fact uniquely) since it is precomplete.
  We show inductively that it is surjective onto morphisms in $\cS_n$.

  For $n=0$ this follows from the generator rule.
  Since $\cS_{n+1}$ is a pushout, its morphisms are generated by the operations in an \lnl polycategory (identities, composition, and structural actions) from those of $\cS_n$ and those of the cones \cC.
  The latter arise from the noninvertible rules, while the \lnl polycategory operations are reflected by the structural rules.
\end{proof}

Finally, the generic invertible rule is shown in figure \cref{fig:inv}, where $-\ep$ reverses a sign.
The requirement $\abs{\dD}(\tau_\cC^{-\ep}, \sigma_1^{\eta_1},\dots, \sigma_m^{\eta_m}) \neq \emptyset$ ensures that we do not produce sequents not allowed by $\abs{\dD}$, e.g.\ the universal properties of limits and colimits are restricted as necessary in an \lnl multicategory.
(Recall we are assuming $\abs{\dD}$ to be subterminal, so its nonempty homsets are singletons.)

For instance, if \cC is the cone for $\ten$ as above, the rule becomes
\[\infer{\vdash A^-,B^-,\Psi}{\vdash (A\ten B)^-,\Psi} \qquad = \qquad
  \infer{\Theta \mid \Gamma, A,B \vdash \Delta}{\Theta \mid \Gamma, A\ten B \vdash \Delta}
\]
while if \cC is the cone for $\with$ as above, the rule becomes
\[\infer{\vdash A^+,\Psi \\ \vdash B^+,\Psi}{\vdash (A\with B)^+, \Psi}
  \qquad = \qquad
  \infer{\Theta \mid \Gamma \vdash \Delta,A \\ \Theta\mid\Gamma\vdash\Delta,B}{\Theta\mid\Gamma\vdash\Delta,A\with B.}
\]
Similarly,
the rules for other common connectives such as $\hom,\oplus,\unit,\counit,\coten,\times,\to,1$ specialize to the usual ones for classical or intuitionistic multiplicative-additive linear logic or intuitionistic nonlinear logic.

For the modalities, the invertible rules are:
\begin{mathpar}
  \infer{\Theta,X \mid \Gamma\vdash \Delta}{\Theta\mid \Gamma,\foc X \vdash \Delta}\and
  \infer{\Theta \mid \cdot \vdash A}{\Theta \vdash \uoc A}\and
  \infer{\Theta,X \mid \Gamma\vdash \Delta}{\Theta\mid \Gamma \vdash \Delta,\fwn X}\and
  \infer{\Theta \mid A \vdash \cdot}{\Theta \vdash \uwn A}\and
\end{mathpar}
As before, if $\abs{\dD}=\mathsc{lnlmulti}$ or $\abs{\dD}=\mathsc{cbpv}$, these rules for $\foc$ and $\uoc$ specialize to those of~\cite{benton:lnl} or~\cite{levy:adj-cbpv} respectively.
Similarly, the rules for $\eechom$ and $\rtimes$, with appropriate cuts added:
\begin{mathpar}
  \infer{\Theta \vdash A\eechom B \\ \Theta'\mid \Gamma \vdash A}{\Theta,\Theta' \mid \Gamma \vdash B}\and
  \infer{\Theta \mid A \vdash B}{\Theta \vdash A\eechom B}\\
  \infer{\Theta \vdash X \\ \Theta'\mid\Gamma\vdash A}{\Theta,\Theta' \mid\Gamma \vdash X\rtimes A}\and
  \infer{\Theta\mid\Gamma \vdash X\rtimes A \\ \Theta',X \mid \Gamma', A \vdash \Delta}{\Theta,\Theta' \mid \Gamma, \Gamma' \vdash \Delta}\and
\end{mathpar}
specialize when $\abs{\dD}=\mathsc{ecbv}$ (so $\Gamma$ is a singleton and $\Gamma'=\emptyset$) to those of~\cite{ms:lin-state} (modulo changes of notation, and additive maintenance for the nonlinear context).

\begin{prop}
  There is a surjection from derivations of $\vdash\Phi$, in the full sequent calculus of \cref{fig:seqcalc}, to the hom-set $\sdhat(\Phi)$.
\end{prop}
\begin{proof}
  As before, the function is defined inductively on derivations, with the invertible logical rule resulting from realizedness.
  Also as before, we prove surjectivity onto $\cS_{\om+n}$ by induction.
  The base case $\cS_\om$ is \cref{thm:surj-noninv}; while the morphisms of $\cS_{\om+n+1}$ are generated by the \lnl polycategory operations (structural rules) from those of $\cS_{\om+n}$ and the factorizations in each $\expand{\cC}{\Psi}$ (invertible logical rules).
\end{proof}

The equivalence relation on derivations of $\vdash\Phi$ whose quotient is $\sdhat(\Phi)$ can also be described syntactically.
It is generated by the composition operation of \cS, the structural axioms of an \lnl polycategory, the principal ``$\beta$-reduction'' rule that reduces a cut of the form
\[\small
  \infer{
      \infer*{\dots \\
        f\in \cC(r_{i_1}^{\ep_1},\dots,r_{i_\ell}^{\ep_\ell},r^{\ep}) \text{ abs.~proj.}
      }
      {\vdash R_{i_1}^{\ep_1},\,\dots,\, R_{i_\ell}^{\ep_\ell},\, \boc[R_1,\dots,R_n]^{\ep}}
      \\
      \infer*{\dots \\
        \big\{\vdash R_{i_1}^{\ep_1},\,\dots,\, R_{i_\ell}^{\ep_\ell},\, \Psi \big\}_{f
          \text{ abs.~proj.}}
      }{\vdash \boc[R_1,\dots,R_n]^{-\ep}, \Psi}
  }{\vdash R_{i_1}^{\ep_1},\,\dots,\, R_{i_\ell}^{\ep_\ell},\, \Psi }
\]
to the derivation of $\vdash R_{i_1}^{\ep_1},\dots, R_{i_\ell}^{\ep_\ell},\Psi$ on the right that is indexed by the specific abstract projection $f$ specified on the left, and the ``$\eta$-conversion'' rule that two derivations of $\vdash \boc[R_1,\dots,R_n]^{-\ep}, S_1^{\eta_1},\dots, S_m^{\eta_m}$ are equal if they become equal upon cutting with the noninvertible rule $\vdash R_{i_1}^{\ep_1},\,\dots,\, R_{i_\ell}^{\ep_\ell},\, \boc[R_1,\dots,R_n]^{\ep}$.

\begin{rem}
  We have constructed $\sdhat$ by a categorical iterative procedure, and then shown that we can extract a sequent calculus from this construction.
  As pointed out by a referee, we could also have specified the sequent calculus first and then used it to construct the free \dD-completion $\sdhat$.
  We regard the \emph{equivalence} between the two as the most interesting observation.
  It is ultimately a matter of personal preference which side of the equivalence one prefers to start from, although the categorical approach does have the advantage of quotienting the morphisms by the appropriate equivalence relation automatically.
\end{rem}

We have described this sequent calculus for a restricted class of doctrines, to reduce the syntactic bureaucracy.
However, analogous calculi can be formulated for any doctrine, with the following modifications.

If $\dD$ contains infinite cones, its sequent calculus has infinitely many rules, some with infinitely many premises.
This is hard to implement, of course, but mathematically unproblematic.
If $\dD$ contains non-discrete cones, the type-formation rules have sequents and equalities of sequents as premises.
Thus both judgments and their equalities are mutually inductive, as in a dependent type theory.

If $\abs{\dD}$ is not subterminal, then the syntactic classes of types must be indexed by objects of $\abs{\dD}$, and the sequents must likewise be indexed by morphisms of $\abs{\dD}$.
The result is a ``fibrational'' calculus similar to that of~\cite{lsr:multi}, though without 2-cells in the ``mode theory'' $\abs{\dD}$.
For instance, if $\abs{\dD} = \mathsc{plmulti}$ as in \cref{rmk:planar}, each sequent is labeled by a permutation of its context; this essentially serves to neuter the exchange rule, leading to a variant of ordered logic.
Similarly, if $\abs{\dD} = \mathsc{linpol}$ or $\mathsc{lnlpol}$ as in \cref{eg:pol}, each linear type is labeled as positive or negative.

Finally, if $\dD$ is sorted and \cS lies only over primitive sorts, we can omit the syntactic classes of types corresponding to derived sorts, or equivalently consider the action of sorting cones to be an implicit coercion.
In addition, in this case usually some of the sequents will be redundant, corresponding to hom-sets that are always canonically isomorphic to some other hom-sets, and can be omitted from the syntax.

For example, a Kleisli sorted doctrine with $\abs{\dD}=\mathsc{lnlmulti}$ yields split-context calculi for intuitionistic linear logic like those of~\cite{barber:dill,wadler:syn-ll}, with only one class of types that can appear in both parts of the context.
Types in the nonlinear part have an implicit application of $\uoc$, so it makes sense to change notation and write $\foc A$ as $\oc A$.
Moreover, since $\Pnl(\Theta;\uoc A) \cong \Pl(\Theta| ;\, A)$, the nonlinear morphisms are determined by the linear ones; thus we can dispense with the nonlinear sequents entirely, essentially defining them by the invertible rule for $\uoc$.
The remaining logical rules for the exponentials then become:
\begin{mathpar}
  \infer{\Theta \mid \cdot \vdash A}{\Theta \mid \cdot \vdash \oc A}\and
  \infer{\Theta,A \mid \Gamma\vdash \Delta}{\Theta\mid \Gamma,\oc A \vdash \Delta}\and
  \infer{\Theta\mid \Gamma,A \vdash \Delta}{\Theta, A\mid \Gamma \vdash \Delta} \and
\end{mathpar}
The first two appear verbatim in~\cite{barber:dill,wadler:syn-ll}, while the third is admissible~\cite[Lemma 2.5]{barber:dill}.
The cut rule that mixes linear and nonlinear sequents also has to be restated in this notation, alongside the one for purely linear sequents:
\begin{mathpar}
  \infer{\Theta'\mid\Gamma' \vdash \Delta',A \\ \Theta\mid \Gamma,A\vdash \Delta}{\Theta,\Theta' \mid \Gamma,\Gamma' \vdash \Delta,\Delta'}\and
  \infer{\Upsilon\mid \cdot \vdash A \\ \Theta,A\mid \Gamma\vdash \Delta}{\Theta,\Upsilon \mid \Gamma\vdash\Delta}.
\end{mathpar}
These cut rules both appear in~\cite[Lemma 3.1]{barber:dill} (``Linear Cut'' and ``Intuitionistic Cut'') and in~\cite{wadler:syn-ll} (``Cut'' and the derivable ``Cut-Int'').

Something similar happens in~\cite{ems:eec} with $\abs{\dD}=\mathsc{cbpv}$, although in this case the computation types are merely \emph{included} in the value types by an implicit $\uoc$, rather than identified with them.
This includes the above rules for $\oc A$ (meaning $\foc A$) with $\Gamma=\emptyset$, and the (arity-restricted, cut-including) rules for $\mixedhom$ (their ``$\to$''):
\begin{mathpar}
  \infer{\Theta\vdash X \\ \Theta'\mid \Gamma\vdash X\mixedhom B}{\Theta,\Theta' \mid \Gamma \vdash B}\and
  \infer{\Theta,X \mid \Gamma \vdash B}{\Theta\mid\Gamma\vdash X\mixedhom B}.
\end{mathpar}
Likewise, for \cref{eg:skew} with $\abs{\dD}=\mathsc{symskew}$, the rules for restricted $\ten$ and $\hom$ (with one tight input --- the ``stoup'' --- and the other loose) specialize to those of~\cite{uvz:skewmon,uvz:skewclosed,veltri:symskew,uvw:skew-mill}.

As a final example, in the double-Kleisli sorted doctrine of \cref{eg:double-kleisli}, we can write the sequents as $\Theta \mid \Gamma \vdash \Delta \mid \Upsilon$, where $\Theta$ and $\Upsilon$ consist of types lying over the ``left-hand'' and ``right-hand'' derived sorts respectively.
Types in $\Theta$ have an implicit $\uoc$ and types in $\Upsilon$ have an implicit $\uwn$, so we write $\foc$ and $\fwn$ as $\oc$ and $\wn$ respectively.
Again we can define the nonlinear sequents by the invertible rules for $\uoc$ and $\uwn$ --- although when translating a nonlinear sequent $\Theta,\Upsilon \vdash A$ in this way, we have to pay attention to whether $A$ is being regarded as a left-hand type or a right-hand type: in the former case the sequent becomes $\Theta \mid \cdot \vdash A \mid \Upsilon$, while in the latter case it becomes $\Theta \mid A \vdash \cdot \mid \Upsilon$ (due to the different universal properties of $\uoc$ and $\uwn$).
The remaining logical rules then become:
\begin{mathpar}
  \infer{\Theta \mid \cdot \vdash A \mid \Upsilon}{\Theta \mid \cdot \vdash \oc A\mid  \Upsilon}\and
  \infer{\Theta,A \mid \Gamma\vdash \Delta\mid  \Upsilon}{\Theta\mid \Gamma,\oc A \vdash \Delta\mid  \Upsilon}\and
  \infer{\Theta\mid \Gamma,A \vdash \Delta\mid  \Upsilon}{\Theta, A\mid \Gamma \vdash \Delta\mid  \Upsilon} \\
  \infer{\Theta\mid A \vdash \cdot \mid \Upsilon}{\Theta \mid \wn A \vdash \cdot \mid \Upsilon} \and
  \infer{\Theta \mid \Gamma\vdash \Delta\mid\Upsilon,A}{\Theta\mid \Gamma \vdash \Delta,\wn A\mid\Upsilon}\and
  \infer{\Theta\mid \Gamma \vdash\Delta,A \mid\Upsilon}{\Theta\mid \Gamma \vdash\Delta\mid\Upsilon,A}\and
\end{mathpar}
and the cut rules multifurcate further into:
\begin{mathpar}
  \infer{\Theta'\mid\Gamma' \vdash \Delta',A \mid\Upsilon' \\ \Theta\mid \Gamma,A\vdash \Delta\mid\Upsilon}{\Theta,\Theta' \mid \Gamma,\Gamma' \vdash \Delta,\Delta'\mid\Upsilon,\Upsilon'}\\
  \infer{\Theta'\mid \cdot \vdash A\mid\Upsilon' \\ \Theta,A\mid \Gamma\vdash \Delta\mid\Upsilon}{\Theta,\Theta' \mid \Gamma\vdash\Delta\mid\Upsilon,\Upsilon'}\and
  \infer{\Theta'\mid A \vdash \cdot\mid\Upsilon' \\ \Theta\mid \Gamma\vdash \Delta\mid\Upsilon,A}{\Theta,\Theta' \mid \Gamma\vdash\Delta\mid\Upsilon,\Upsilon'}.
\end{mathpar}
These are all precisely the relevant logical and structural rules of~\cite{girard:unity}.

\section{Adjunctions induced by doctrine maps}
\label{sec:doc-adj}

Our last goal is to show that a doctrine map $\F:\dD_1\to\dD_2$ induces a pseudo 2-adjunction relating $\dD_1$-categories to $\dD_2$-categories, combining the adjunctions
from \cref{thm:sketch-adj,thm:bicat-refl}.

\begin{thm}\label{thm:cplt-adj}
  For any morphism $\F:\dD_1\to\dD_2$ of small doctrines, there is an induced pseudo 2-adjunction
  \[ \ladj{\Fhat} : \cat{\dD_1}_g \toot  \cat{\dD_2}_g : \radj{\Fhat}. \]
\end{thm}
\begin{proof}
  Identifying $\dD_i$-categories with $\dD_i$-complete sketches, we define $\radj{\Fhat}$ to be the $\radj{\F}$ from \cref{thm:sketch-adj} restricted to $\dD_2$-complete inputs.
  This takes values in $\dD_1$-complete sketches because the $\ladj{\F}$ from \cref{thm:sketch-adj} maps $\cI_{\dD_1}$ into $\cI_{\dD_2}$, up to isomorphism.
  Now we can define $\ladj{\Fhat}(\cS) = \widehat{(\ladj{\F}\cS)}_{\dD_2}$, and compute
  \begin{multline*}
    \cat{\dD_2}_g(\ladj{\Fhat}(\cS),\cT) =
    \cat{\dD_2}_g(\widehat{(\ladj{\F}\cS)}_{\dD_2},\cT) \simeq
    \sketch{\dD_2}_g(\ladj{\F}\cS, \cT)\\ \cong
    \sketch{\dD_1}_g(\cS,\radj{\F}\cT) \cong
    \cat{\dD_1}_g(\cS,\radj{\Fhat}\cT). \tag*{\qedhere} 
  \end{multline*}
\end{proof}

\begin{thm}\label{thm:sort-adj}
  For any sorted map $\F:\dD_1\to\dD_2$ of small sorted doctrines, there is an induced pseudo 2-adjunction
  \[ \ladj{\Ftil} : \scat{\dD_1}_g \toot  \scat{\dD_2}_g : \radj{\Ftil}. \]
\end{thm}
\begin{proof}
  It suffices to show that both functors in \cref{thm:cplt-adj} preserve well-sortedness.
  For $\radj{\Fhat} = \radj{\F}$ this follows from \cref{thm:docmap-sort}.
  For $\ladj{\Fhat}$, let \cS be a well-sorted $\dD_1$-complete sketch.
  By \cref{thm:docmap-sort}, $\ladj{\F}(\cS)$ is a well-sorted (incomplete) $\dD_2$-sketch; thus by \cref{thm:cplt-sort}, $\ladj{\Fhat}(\cS) = \widehat{(\ladj{\F}{\cS})}_{\dD_2}$ is also well-sorted.
\end{proof}

\begin{rem}\label{rmk:dir-adj}
  If $\dD_2$ (hence also $\dD_1$) contains only ``totally covariant'' operations, then \cref{thm:cplt-adj,thm:sort-adj} extend to pseudo 2-adjunctions $\cat{\dD_1}\toot\cat{\dD_2}$ and $\scat{\dD_1}\toot\scat{\dD_2}$ including the noninvertible 2-cells.
\end{rem}

We conclude with examples.
In fact, nearly all the obvious forgetful functors between classes of \lnl polycategories discussed in \cref{sec:relation-literature} are of the form $\radj{\Fhat}$ for some (sorted) doctrine map $\F$, and therefore have left pseudo-adjoints.

To start with, we consider maps between doctrines that have no cones.
These induce $\radj{\Fhat}$ functors including the following.
\begin{itemize}
\item The underlying \lnl multicategory of an \lnl polycategory.
\item The underlying cartesian multicategory, and the underlying symmetric polycategory, of an \lnl multicategory or \lnl polycategory.
\item The underlying symmetric multicategory of a symmetric polycategory, \lnl multicategory, or \lnl polycategory.
\end{itemize}
Thus, all of these forgetful functors have left pseudo-adjoints, which extend to non-invertible 2-cells as in \cref{rmk:dir-adj}.

By adding appropriate cones to the doctrines, we obtain more $\radj{\Fhat}$ functors, such as the following.
In each case we must check that the putative doctrine map actually preserves the specified cones.
This basically means that every specified kind of universal property in the domain doctrine is also specified in the codomain, which is essentially just the assertion that the forgetful functor in question exists.
\begin{itemize}
\item The underlying symmetric monoidal category of a linearly distributive category.
\item The underlying closed symmetric monoidal category of a $\ast$-autonomous category.
  To represent this using a doctrine morphism, we need to explicitly include a $\hom$-cone in the doctrine for $\ast$-autonomous categories (to be the image of the $\hom$-cone in the doctrine for closed symmetric monoidal categories).
  Since internal-homs can be derived from duals, and hence are automatically preserved by $\ast$-autonomous functors, this yields an equivalent 2-category of \dD-categories.
\item The underlying linearly distributive category of a $\ast$-autonomous category.
  As in the previous example, for this we need to include redundant $\coten$- and $\counit$-cones in the doctrine for $\ast$-autonomous categories.
\item The underlying symmetric monoidal category, and the underlying cartesian monoidal category, of an \lnl adjunction.
\item The underlying $\ast$-autonomous category, and the underlying cartesian monoidal category, of a $\ast$-autonomous \lnl adjunction.
\item The underlying CBPV pre-structure of an \lnl adjunction, the underlying EEC+ model of a closed \lnl adjunction with products and coproducts, and so on.
\end{itemize}
Thus, all of these forgetful functors have left pseudo-adjoints as well.
Those with no contravariant operations (such as $\hom$ and $\duals$) extend to non-invertible 2-cells as in \cref{rmk:dir-adj}.
We can also add any desired limits and colimits to these doctrines.

Finally, we consider sorted maps of doctrines containing some derived sorts.
In the simplest case, the domain doctrine has all sorts primitive, in which case a doctrine map is sorted just when it maps every sort to a primitive one.
This yields $\radj{\Fhat}$ functors such as the following.
\begin{itemize}
\item The underlying (closed) symmetric monoidal category of a (closed) symmetric monoidal category with a linear exponential comonad.
\item The underlying linearly distributive category of a linearly distributive category with storage.
\item The underlying (symmetric) multicategory of a (symmetric) skew multicategory.
\end{itemize}
If the domain has primitive sorts, we have to check the rest of \cref{defn:docmap-sort}.
This yields $\radj{\Fhat}$ functors such as the following, all with left pseudo-adjoints.
\begin{itemize}
\item The underlying symmetric monoidal category with linear exponential comonad of a linearly distributive category with storage.
  Here the unique derived (nonlinear) sort in the domain maps to the derived nonlinear sort of left-hand objects in the codomain (see \cref{eg:double-kleisli}).
\item The underlying linearly distributive category with storage of a $\ast$-autonomous category with storage.
\item The underlying (symmetric) skew monoidal category of a lax (symmetric) monoidal comonad, as in~\cite[Definition 7.4]{szlachanyi:skew} or~\cite[Example 2]{veltri:symskew}.
  Here the underlying functor of the doctrine map $\mathsc{symskew} \to \mathsc{smadj}$
  is defined by $\mathsc{l}\mapsto \mathsc{p}$ and $\mathsc{t}\mapsto \mathsc{n}$, where $\mathsc{smadj}$ has $\mathsc{p}$ derived and $\mathsc{n}$ primitive.
\end{itemize}

\section*{Acknowledgments}
\noindent I would like to thank Robin Cockett, Max New, Paul Blain Levy, Noam Zeilberger, Christine Tasson, and Martin Hyland for helpful conversations and comments, and Nicolas Blanco for a careful reading and very helpful suggestions.
I would also like to thank the referees for very helpful suggestions.

\bibliographystyle{alphaurl}
\bibliography{lnlpoly}

\begin{table}
  \centering
  \begin{sideways}
  \begin{tabular}{c|c|p{6in}}
    \textbf{Name}&\textbf{Reference}&\textbf{Definition}\\\hline
    $\mathsc{lnlpoly}$ & \cref{rmk:slice} &  one linear object, one nonlinear object, all homsets singletons.\\\hline
    $\mathsc{lnlmulti}$ & \cref{rmk:slice} &  one linear object, one nonlinear object, all nonlinear homsets and co-unary linear homsets singletons.\\\hline
    $\mathsc{sympoly}$ & \cref{rmk:slice} &  one linear object, no nonlinear objects, and all linear homsets singletons. \\\hline
    $\mathsc{symmulti}$ & \cref{rmk:slice} &  one linear object, no nonlinear objects, co-unary linear homsets singletons, and others empty.\\\hline
    $\mathsc{cartmulti}$  & \cref{rmk:slice} &  one nonlinear object, no linear objects, all nonlinear homsets singletons, and all linear homsets empty.\\\hline
    $\mathsc{cat}$  & \cref{rmk:slice} &  one linear object, no nonlinear objects, and only the identity morphism. \\\hline
    $\mathsc{plmulti}$ & \cref{rmk:planar} &  one linear object, and morphisms with arity $n$ and co-arity 1 labeled by permutations of $n$ objects.\\\hline
    $\mathsc{dblsplit}$ & \cref{rmk:double-split} &  one linear object, two nonlinear objects, and all homsets singletons.\\\hline
    $\mathsc{cbpv}$ & \makecell[t]{after\\\cref{thm:cbpv-pow}} &  one nonlinear object, one linear object, all nonlinear homsets and subunary co-unary linear homsets singletons, and others empty.\\\hline
    $\mathsc{ecbv}$ & \makecell[t]{after\\\cref{thm:cbpv-pow}} &  one nonlinear object, one linear object, all nonlinear homsets and unary co-unary linear homsets singletons, and others empty.\\\hline
    $\mathsc{smadj}$ & \cref{eg:smadj} &  two linear objects $\mathsc{p},\mathsc{n}$, a unique morphism $\Gamma \to \mathsc{p}$ when $\Gamma$ consists entirely of $\mathsc{p}$'s, and a unique morphism $\Gamma \to \mathsc{n}$ for any $\Gamma$.\\\hline
    $\mathsc{adj}$ & \cref{eg:adj} & two linear objects $\mathsc{p},\mathsc{n}$, a unique nonidentity morphism $\mathsc{p}\to\mathsc{n}$.\\\hline
    $\mathsc{linpol}$ & \cref{eg:pol} &  two linear objects $\mathsc{p},\mathsc{n}$, a unique morphism $\Gamma \to \mathsc{p}$ when $\Gamma$ consists entirely of $\mathsc{p}$'s, and a unique morphism $\Gamma \to \mathsc{n}$ when $\Gamma$ contains no more than one $\mathsc{n}$.\\\hline
    $\mathsc{symskew}$ & \cref{eg:skew} &  same as $\mathsc{linpol}$.\\\hline
    $\mathsc{lnlpol}$ & \cref{eg:pol} &  two linear objects $\mathsc{p},\mathsc{n}$, one nonlinear object $\mathsc{x}$, all nonlinear homsets singletons, a unique morphism $(\Theta \mid \Gamma) \to \mathsc{p}$ if $\Gamma$ consists entirely of $\mathsc{p}$'s, and a unique morphism $(\Theta\mid\Gamma) \to \mathsc{n}$ when $\Gamma$ contains no more than one $\mathsc{n}$.
  \end{tabular}
\end{sideways}
\caption{Subterminal and other small \lnl polycategories}
  \label{tab:subterms}
\end{table}

\end{document}